\documentclass[12pt,reqno]{amsart} % Standard document class for mathematical papers

\makeatletter
\@namedef{subjclassname@2020}{\textup{2020} Mathematics Subject Classification}
\makeatother

\usepackage{amsmath,amssymb,amsthm,amscd}
\usepackage[T1]{fontenc}
\usepackage{blkarray} % For matrices with labeled rows and columns

\usepackage{mathrsfs,euscript} % Additional script fonts
\usepackage{graphicx} % Graphics inclusion
\usepackage[colorlinks=true,urlcolor=blue,citecolor=red,linkcolor=blue,hyperfootnotes=true]{hyperref} % Hyperlinks

\let\mc\mathcal

\let\eu\EuScript
\let\nc\newcommand

\newtheorem{thm}{Theorem}[section]
\newtheorem{cor}[thm]{Corollary}
\newtheorem{lem}[thm]{Lemma}
\newtheorem{prop}[thm]{Proposition}
\newtheorem{conj}[thm]{Conjecture}

\theoremstyle{definition}
\newtheorem{defn}[thm]{Definition}
\newtheorem{rem}[thm]{Remark}

\newtheorem{example}[thm]{Example}

\numberwithin{equation}{section}

\def\beq{\begin{equation}}
\def\eeq{\end{equation}}
\def\be{\begin{equation*}}
\def\ee{\end{equation*}}
\nc{\bea}{\begin{eqnarray*}}
\nc{\eea}{\end{eqnarray*}}

\let\al\alpha
\let\bt\beta
\let\dl\delta
\let\Dl\Delta
\let\eps\varepsilon
\let\gm\gamma
\let\Gm\Gamma

\let\la\lambda
\let\La\Lambda

\let\phi\varphi
\let\si\sigma

\let\thi\vartheta

\let\om\omega
\let\Om\Omega

\let\der\partial
\let\Hat\widehat
\let\Tilde\widetilde

\let\on\operatorname

\def\End{\on{End}}

\def\Hom{\on{Hom}}

\def\N{{\mathbb N}}
\def\C{{\mathbb C}}
\def\Z{{\mathbb Z}}
\def\Q{{\mathbb Q}}

\def\Pb{{\mathbb P}}
\def\R{{\mathbb R}}

\usepackage[OT2,T1]{fontenc}
\usepackage[all,cmtip]{xy} % Commutative diagrams
\usepackage{cancel}
\usepackage[vcentermath]{youngtab} % Young tableaux
\usepackage{empheq} % Boxed equations
\usepackage{changepage}
\usepackage{bm} % Bold math symbols

\newlanguage\fakelanguage
\newcommand\cyr{\fontencoding{OT2}\fontfamily{wncyr}\selectfont
   \language\fakelanguage}
\DeclareTextFontCommand{\textcyr}{\cyr}

\usepackage{calligra}

\DeclareMathOperator{\HOM}{\mathscr{H}\text{\kern -3pt {\calligra\large om}}\,}

\makeatletter
\newsavebox{\@brx}
\newcommand{\llangle}[1][]{\savebox{\@brx}{$\m@th{#1\langle}$}%
  \mathopen{\copy\@brx\kern-0.5\wd\@brx\usebox{\@brx}}}
\newcommand{\rrangle}[1][]{\savebox{\@brx}{$\m@th{#1\rangle}$}%
  \mathclose{\copy\@brx\kern-0.5\wd\@brx\usebox{\@brx}}}
\makeatother

\newcommand\xqed[1]{%
  \leavevmode\unskip\penalty9999 \hbox{}\nobreak\hfill
  \quad\hbox{#1}}
\newcommand\qetr{\xqed{$\triangle$}}
\newcommand\qrem{\xqed{$\spadesuit$}}

\let\bi\bibitem

\usepackage{scalerel,stackengine}  %QUESTO COMANDO SERVE A FARE WIDEHAT VERAMENTE GRANDE
\stackMath
\newcommand\reallywidehat[1]{%
\savestack{\tmpbox}{\stretchto{%
  \scaleto{%
    \scalerel*[\widthof{\ensuremath{#1}}]{\kern.1pt\mathchar"0362\kern.1pt}%
    {\rule{0ex}{\textheight}}%WIDTH-LIMITED CIRCUMFLEX
  }{\textheight}% 
}{2.4ex}}%
\stackon[-6.9pt]{#1}{\tmpbox}%
}
\usepackage{pifont}

\newcommand{\sq}{\mathbin{\text{\tiny\ding{73}}}}
\newcommand{\bqf}{\mathbin{\text{\tiny\ding{72}}^{\scriptscriptstyle{\rm Fib}}}}
\newcommand{\sqf}{\mathbin{\text{\tiny\ding{73}}^{\scriptscriptstyle \rm Fib}}}
\newcommand{\sqqf}[1]{\mathbin{\text{\tiny\ding{73}}_{#1}^{\scriptscriptstyle \rm Fib}}}

\DeclareRobustCommand{\lcyr}{\textnormal{\textcyr{l}}}
\DeclareRobustCommand{\Lcyr}{\textnormal{\textcyr{L}}}
\DeclareRobustCommand{\Lcyrit}{\textnormal{\LARGE\it\textcyr{l}}}
\DeclareRobustCommand{\tlcyr}{\tilde{\textnormal{\textcyr{l}}}}
\DeclareRobustCommand{\tLcyr}{\widetilde{\textnormal{\textcyr{L}}}}
\DeclareRobustCommand{\tLcyrit}{{\textnormal{\LARGE{\it \textcyr{\~l}}}}}

\usepackage{booktabs}
\usepackage{adjustbox}  % per ridimensionare la tabella
\usepackage{lscape}     % per usare orientamento orizzontale se necessario
\usepackage{longtable}  % utile per tabelle lunghe che vanno su più pagine
\usepackage{bbm}
\usepackage{enumitem}

\begin{document}
\title[Fiberwise GW theory of flag bundles, and prime factorization]{Fiberwise Gromov--Witten theory,  quantum spectra of flag bundles, and prime factorization of integers}
\author[Giordano Cotti]{Giordano Cotti$\>^{\circ}$}

{\let\thefootnote\relax
\footnotetext{\vskip5pt 
\noindent
$^\circ\>$\textit{ E-mail}:  giordano.cotti@tecnico.ulisboa.pt}}

\maketitle
\begin{center}
\textit{ 
$^\circ\>$Grupo de F\'isica Matem\'atica \\
Departamento de Matem\'atica, Instituto Superior T\'ecnico\\
Av. Rovisco Pais, 1049-001 Lisboa, Portugal\/\\}
\end{center}
\vskip2cm

\begin{abstract}
This paper is devoted to the study of the vertical quantum cohomology and quantum spectra of flag bundles, establishing new connections between the Gromov--Witten theory of homogeneous fibrations and prime number theory.

Building on the constructions of Astashkevich--Sadov \cite{AS95} and Biswas--Das--Oh--Paul \cite{BDOP25}, we first prove functorial and inductive properties of vertical quantum cohomology, and relate vertical and absolute quantum spectra.  
Consequently, we show that the degeneracy of the small vertical quantum spectrum of a Grassmann bundle -- that is, the occurrence of eigenvalues with higher-than-expected multiplicities --   is governed by the prime decomposition of the involved ranks, extending previous results for Grassmannians of \cite{Cot22} to the relative setting. This applies, in particular, to partial flag varieties, viewed as total spaces of suitable Grassmann bundles.

We then introduce three families of double sequences, denoted by $\lcyr(n,N)$, $\tlcyr(n,N)$, and $\ell(n,N)$, which enumerate partial flag varieties according to different quantum spectral and combinatorial conditions.  
We analyse their recursive, combinatorial, and arithmetic properties via ordinary and Dirichlet generating functions.  
The sequence $\lcyr$ satisfies a Pascal-type recursion, enabling a detailed study of its partial Dirichlet series, whose analytic continuations exhibit logarithmic singularities determined by the non-trivial zeros of the Riemann zeta function.  
Furthermore, we establish that, for every fixed integer shift~$k$, the diagonal subsequences 
$N \mapsto \lcyr(N+k,N)$, $N \mapsto \tlcyr(N+k,N)$, and $N \mapsto \ell(N+k,N)$ 
exhibit eventual polynomial behaviour, which can be naturally interpreted in terms of weighted walks on graphs.  
Finally, we study the vanishing pattern of $\ell$, deriving equivalent formulations of Goldbach’s conjecture.

Overall, our results reveal a deep interplay between enumerative geometry, quantum spectral degeneracy, and classical problems in analytic number theory.
\end{abstract}

\vskip0,3cm
\begin{adjustwidth}{35pt}{35pt}
{\footnotesize {\it Key words:} Gromov--Witten invariants, quantum cohomology, quantum spectra, partial flag varieties, prime numbers}
\vskip2mm
\noindent
{\footnotesize {\it 2020 Mathematics Subject Classification:} Primary: 53D45 Secondary: 11N99}
\end{adjustwidth}

\tableofcontents

\section{Introduction}

\noindent1.1.\,\,{\bf From Gromov--Witten theory to quantum spectra.} Gromov--Witten theory provides a powerful algebro--geometric framework for the virtual enumeration of curves inside smooth projective varieties. Its fundamental objects, the Gromov--Witten invariants, are defined as intersection numbers on moduli spaces of stable maps and can be regarded as virtual counts of curves of prescribed genus and degree subject to incidence conditions. These invariants encode subtle geometric information, and when organized into generating functions they give rise to rich algebraic and analytic structures. Among the most prominent are quantum cohomology, which deforms the classical cohomology ring by incorporating curve-counting data, and more refined constructions such as quantum spectra.  See e.g.\,\,\cite{KM94,Man99}.

The \emph{small quantum cohomology} of a smooth projective variety $X$ is a family of algebra structures 
\[
(H_X, \sq_{\bm q}) \quad \text{parametrized by } \bm q \in (\C^*)^{D}, \qquad D=\dim_\C H^{1,1}(X,\C),
\] 
supported on the finite-dimensional $\C$-vector space $H_X = H^\bullet(X,\C)$. The product $\sq_{\bm q}$ encodes information about rational curves on $X$ with three incidence constraints, while the parameters $\bm q$ serve as deformation variables. In the classical limit (identifiable with the regime $\bm q\to 0$), the quantum product reduces to the ordinary cup product.  

The \emph{quantum spectrum} of $X$ is defined as the spectrum (i.e.~the multiset of eigenvalues) of the endomorphism  
\[
c_1(X)\sq_{\bm q}\colon H_X \longrightarrow H_X,
\]
given by quantum multiplication by the first Chern class of $X$. The study of the quantum spectrum has recently attracted considerable attention: it is conjectured to encode deep aspects of the complex geometry of $X$, its derived category, and even its birational geometry. Indeed, the structure of the spectrum is the object of a growing number of conjectures and constructions of new invariants.  

Several motivations for this interest can be highlighted:  

\begin{enumerate}
    \item \textbf{Quantum differential equations.}  
    The operator $c_1(X)\sq_{\bm q}$ governs the asymptotics and the Stokes phenomena of solutions to the quantum differential equation associated with $X$, since it appears as the dominant term in the corresponding differential operator.  See e.g.\,\,\cite{CDG24}.

    \item \textbf{Conjecture ${\mc O}$ (Galkin--Golyshev--Iritani).}  
    For a Fano variety $X$, Conjecture $\mc O$ of \cite{GGI16} predicts a precise structure for the spectrum at the special point $\bm q=(1,\dots,1)$. Namely, the spectral radius 
    $ \delta_0 = \max\{|x| : x \ \text{eigenvalue of } c_1(X)\sq_{\bm q}\}$
    is expected to itself be an eigenvalue, and any other eigenvalue of maximal modulus should be of the form $\xi \cdot \delta_0$, where $\xi$ is an $r_X$-th root of unity ($r_X$ being the Fano index of $X$).  

    \item \textbf{Derived categories and exceptional collections.}  
    The multiplicity and symmetry structure of the spectrum are conjecturally related to the geometry of the derived category $D^b(X)$. For instance, if $(H_X,\sq_{\bm q})$ is semisimple for some $\bm q$, and the spectrum has an eigenvalue of algebraic multiplicity $m>1$, a refined version of Dubrovin’s conjecture predicts the existence of $m$-block exceptional collections in $D^b(X)$, see \cite{CDG20,CDG24}. Furthermore, in the conjecture of A.\,Kuznetsov and M.\, Smirnov, the spectrum plays a central role in linking quantum cohomology to semiorthogonal decompositions \cite{KS21}.  

    \item \textbf{Birational geometry and blowups.}  
    The behavior of the quantum spectrum under birational transformations has been conjectured to reflect semiorthogonal decompositions of derived categories. A conjecture due to M.\,Kontsevich asserts that, for the blowup $\widetilde X$ of $X$ along a subvariety $Z \subset X$, the spectrum of $\widetilde X$ should decompose in a manner compatible with D.\,Orlov’s description of $D^b(\widetilde X)$ in terms of $D^b(X)$ and $D^b(Z)$ \cite{Orl92}. Partial confirmations of this conjecture are known, in particular in the surface case, see \cite{GS25} and references therein.  

    \item \textbf{Birational invariants from atoms and chemical formulas.}  
    More recently, L.\,Katzarkov, M.\,Kontsevich, T.\,Pantev, and T.Y.\,Yu have constructed new birational invariants of algebraic varieties, known as \emph{atoms} and \emph{chemical formulas} \cite{KKPY25}. The eigenspace decomposition of the operator $c_1(X)\sq_{\bm q}$ is closely related to this construction, providing yet another bridge between quantum invariants and birational geometry.  
\end{enumerate}

\noindent1.2.\,\,{\bf Fiberwise Gromov--Witten theory, and vertical quantum cohomology. }
Beyond the absolute case, recent years have witnessed a growing interest in developing \emph{family} versions of Gromov--Witten theory, where the target space varies over a base rather than being fixed once and for all. This perspective resonates with Grothendieck’s philosophy that the “correct” form of a geometric statement is often relative: it should be formulated not for an isolated object, but for morphisms or families. Placing Gromov--Witten theory in this relative setting provides a more flexible framework, within which specialization, degeneration, and deformation phenomena can be treated in a systematic way.  

The first appearance of such a relative version can be traced back to the notion of \emph{vertical quantum cohomology} of an algebraic bundle $\pi \colon X \to B$
with fiber $F$, introduced by A.\,Astashkevich and V.\,Sadov \cite{AS95}. In this setting, the classical limit of the vertical quantum cohomology of $(\pi,X,B,F)$ recovers the cohomology of the total space $X$. The vertical quantum product (denoted by $\sqqf{\bm q}$) is defined as a deformation of the usual cup product on $H^\bullet(X,\C)$, governed by contributions of \emph{vertical} rational curves $C \subset X$, i.e.\,\,satisfying $\pi(C) = \mathrm{pt}$. Moreover, the construction naturally contains $H^\bullet(B,\C)$ as a subring, so that the vertical quantum cohomology acquires the structure of an $H^\bullet(B,\C)$-algebra.  

The approach of Astashkevich and Sadov largely relied on expected properties of a relative moduli space of stable maps to the fibers of $\pi$, without providing a rigorous construction. This gap has been recently addressed by I.\,Biswas, N.\,Das, J.\,Oh, A.\,Paul  \cite{BDOP25}, who constructed a genuine moduli space of stable maps to the fibers of a fiber bundle. This new space serves as a family version of the classical moduli space of stable maps to a smooth projective variety, and it carries a virtual fundamental class. On this basis, the authors define analogues of Gromov--Witten invariants in the relative/family setting, thereby placing vertical quantum cohomology on firm mathematical foundations.

Remarkably, as already observed by Astashkevich and Sadov, vertical quantum cohomology enjoys more natural properties than its absolute counterpart. For instance, it satisfies {\it functorial properties} with respect to base changes, as well as an {\it induction property}: given two algebraic bundles $(\pi,X,B,F)$ and $(\pi', B, B', F')$, the vertical quantum cohomology of $(\pi,X,B,F)$ can be identified with a suitable quotient of the vertical quantum cohomology of the locally trivial fibration $(\pi'\circ\pi, X,B', \pi^{-1}(F'))$. In particular, this makes possible to identify the vertical quantum cohomology of $(\pi,X,B,F)$ with a suitable {\it partially classical limit} of the absolute quantum cohomology of $X$.

In the first part of this paper, we relate the constructions of Astashkevich--Sadov and Biswas--Das--Oh--Paul, and we review and generalize the functorial and induction properties of vertical quantum cohomology.

\noindent1.3.\,\,{\bf Results on vertical quantum spectra of flag bundles.}  
Our first main result concerns the vertical quantum spectrum of flag bundles. Let $E \to X$ be an algebraic (or holomorphic) vector bundle of rank ${\rm rk}\,E=n$ over a smooth projective variety $X$. For any \emph{composition} of $n$ into $N$ positive parts, i.e.\ an $N$-tuple $\bm\la=(\la_1,\dots,\la_N)\in\Z^N_{>0}$ with $\la_1+\dots+\la_N=n$, denote by $\eu F_{\bm\la}(E)$ the fiber bundle over $X$ whose fiber over $p\in X$ is the partial flag variety $F_{\bm\la}(E_p)\cong F_{\bm\la}$ parametrizing filtrations
\[
0=V_0 \subset V_1 \subset \dots \subset V_N=\C^n, 
\qquad \dim_{\C}(V_i/V_{i-1})=\la_i, \quad i=1,\dots,N.
\]
When $N=2$ and $\bm\la=(k,n-k)$, this construction recovers the Grassmann bundle $\eu G_k(E)\to X$, with fiber the Grassmannian ${\rm Gr}(k,n)$ of $k$-dimensional subspaces in $\C^n$.

\medskip

Consider the vertical quantum multiplication operator
\[
c_1(\eu F_{\bm\la}(E))\sqqf{\bm q}\colon H^\bullet(\eu F_{\bm\la}(E),\C)\to H^\bullet(\eu F_{\bm\la}(E),\C),
\]
together with its associated \emph{vertical quantum characteristic polynomial}
\[
f_{(\eu F_{\bm\la}(E),X,F_{\bm\la})}(\zeta;\bm q)
=\det\!\left(\zeta\cdot{\rm Id}-c_1(\eu F_{\bm\la}(E))\sqqf{\bm q}\right).
\]
Similarly, for the fiber $F_{\bm\la}$, consider the (absolute) quantum multiplication operator
\[
c_1(F_{\bm\la})\sq_{\bm q}\colon H^\bullet(F_{\bm\la},\C)\to H^\bullet(F_{\bm\la},\C),
\]
and its \emph{quantum characteristic polynomial}
\[
f_{F_{\bm\la}}(\zeta;\bm q)=\det\!\left(\zeta\cdot{\rm Id}-c_1(F_{\bm\la})\sq_{\bm q}\right).
\]

\medskip

Our first main theorem provides an explicit relation between the two characteristic polynomials $f_{(\eu F_{\bm\la}(E),X,F_{\bm\la})}(\zeta;\bm q)$ and $f_{F_{\bm\la}}(\zeta;\bm q)$. For $n\in\N_{>0}$, denote by $p_1(n)$ the smallest prime divisor of $n$.

\begin{thm}[Thm.\,\ref{thmvqcp}, Cor.\,\ref{corexcee1}, Cor.\,\ref{corvqsgrb}]$\,$\newline
\noindent $(1)$ We have
\[
f_{(\eu F_{\bm\la}(E),X,F_{\bm\la})}(\zeta;\bm q)
=\bigl[f_{F_{\bm\la}}(\zeta;\bm q)\bigr]^{\dim_\C H^\bullet(X,\C)}.
\]
In particular, every element of the vertical quantum spectrum of $(\eu F_{\bm\la}(E),X,F_{\bm\la})$ has algebraic multiplicity at least $\dim_\C H^\bullet(X,\C)$.  

\smallskip
\noindent $(2)$ The vertical quantum spectrum of $(\eu F_{\bm\la}(E),X,F_{\bm\la})$ is \emph{exceeding} (i.e.\ some eigenvalue has algebraic multiplicity $>\dim_\C H^\bullet(X,\C)$) if and only if the fiber $F_{\bm\la}$ does not have simple quantum spectrum.  

\smallskip
\noindent $(3)$ The Grassmann bundle $(\eu G_k(E),X,{\rm Gr}(k,{\rm rk}\,E))$ has exceeding vertical quantum spectrum if and only if
\[
p_1({\rm rk}\,E)\leq k\leq {\rm rk}\,E-p_1({\rm rk}\,E).
\]
\end{thm}

\medskip

\noindent1.4.\,\,{\bf Semiclassical spectra of partial flag varieties, and prime numbers.} Already point~(3) of Theorem~\ref{thmvqcp} provides a direct extension of the results in~\cite{Cot22} 
from complex Grassmannians to Grassmann \emph{bundles}. 
This generalization shows that the correspondence between the prime decomposition of the rank of $E$ 
and the structure of the quantum spectrum persists in the relative setting. 
In particular, it reveals that the phenomenon relating prime numbers to the degeneracy of quantum spectra 
is intrinsic to the geometry of homogeneous fibrations, rather than being specific to absolute Grassmannians.

This result can therefore be used to further deepen the connection between 
the enumerative geometry of more general homogeneous varieties and prime number theory. 
As a concrete application, we shall consider in Section~\ref{secpfv} the case of {partial flag varieties} themselves, 
and study suitable partially classical limits of their quantum spectra. 
These limits provide additional insight into how the arithmetic structure of the parameters 
governs the quantum geometry of flag manifolds.

Given a composition $\bm\lambda \in \mathbb{Z}_{>0}^N$ of $n$, 
the small quantum cohomology of the partial flag variety $F_{\bm\lambda}$ 
is parametrized by points $\bm q = (q_1, \dots, q_{N-1}) \in (\mathbb{C}^*)^{N-1}$. 
For each $i = 1, \dots, N$, consider the limiting operator
\[
A_i(q_i) := \lim_{q_j \to 0,\, j \neq i} 
c_1(F_{\bm\lambda}) \sq_{\bm q} 
\in \operatorname{End}_{\mathbb{C}} H^\bullet(F_{\bm\lambda}, \mathbb{C}).
\]
We refer to the spectrum of $A_i$ as the \emph{$i$-th semiclassical spectrum} of $F_{\bm\lambda}$.

This limiting procedure admits an enumerative--geometric reinterpretation.  
In Section~\ref{secpfv}, we show that each operator $A_i$ coincides with the vertical quantum product---and its semiclassical spectrum with the vertical quantum spectrum---of one of $N-1$ distinct fiber bundles, all having $F_{\bm\lambda}$ as total space.

For any fixed $i = 1, \dots, N-1$, define the composition
\[
\bm\lambda_{/i} = (\lambda_1, \dots, \lambda_{i-1}, 
\lambda_i + \lambda_{i+1}, \lambda_{i+2}, \dots, \lambda_N)
\in \mathbb{Z}_{>0}^{N-1}.
\]
We have a natural forgetful morphism
\[
F_{\bm\lambda} \to F_{\bm\lambda_{/i}}, \qquad
(V_0 \subset \dots \subset V_i \subset \dots \subset V_N)
\mapsto
(V_0 \subset \dots \subset V_{i-1} \subset V_{i+1} \subset \dots \subset V_N).
\]
This morphism defines a Grassmann bundle over $F_{\bm\lambda_{/i}}$, with total space
\[
F_{\bm\lambda} = \mathcal{G}_{\lambda_i}(\mathcal{Q}_i), 
\qquad \mathcal{Q}_i \to F_{\bm\lambda_{/i}}
\text{ the $i$-th tautological quotient bundle.}
\]

\begin{thm}[Thm.\,\ref{THMSEMICLSPEC}]\label{thmintro2}
Each eigenvalue in the $i$-th semiclassical spectrum of $F_{\bm\lambda}$ has algebraic multiplicity 
at least $\dim_{\mathbb{C}} H^\bullet(F_{\bm\lambda_{/i}}, \mathbb{C})$. 
The semiclassical spectrum is \emph{exceeding} if and only if
\[
p_1(\lambda_i + \lambda_{i+1})
\leq \lambda_i, \lambda_{i+1}
\leq \lambda_i + \lambda_{i+1} - p_1(\lambda_i + \lambda_{i+1}).
\]
\end{thm}

\medskip

\noindent 1.5.\,\,{\bf Three double sequences. } 
In the second part of the paper, we focus on three distinct double sequences, viewed as functions of the two parameters $(n, N)$. 
These sequences are defined by counting partial flag varieties $F_{\bm\lambda}$, parametrizing $N$-step chains of subspaces in $\mathbb{C}^n$, that satisfy three different types of conditions.

For any $2 \leq N \leq n$, we denote by:
\begin{itemize}
\item $\lcyr(n,N)$ the number of partial flag varieties $F_{\bm\lambda}$, with $\bm\lambda \in \mathbb{Z}_{>0}^N$ and $|\bm\lambda| = n$, admitting \emph{at least one} non-exceeding semiclassical spectrum;

\item $\tlcyr(n,N)$ the number of such $F_{\bm\lambda}$ admitting \emph{only} non-exceeding semiclassical spectra;

\item $\ell(n,N)$ the number of $F_{\bm\lambda}$, with $\bm\lambda \in \mathbb{Z}_{>0}^N$ and $|\bm\lambda| = n$, such that for every index $i = 1, \dots, N-1$, the pair of subspaces $(V_{i-1}, V_{i+1})$ in the associated flag -- that is, subspaces separated by one intermediate step -- has prime-dimensional gap, namely
\[
\dim_{\mathbb{C}}(V_{i+1}/V_{i-1}) = \lambda_i + \lambda_{i+1} \text{ is prime.}
\]

\end{itemize}

All these sequences are \emph{triangular}, in the sense that they vanish unless $N \leq n$. Moreover, they satisfy
\[
0 \leq \ell(n,N) \leq \tlcyr(n,N) \leq \lcyr(n,N) \leq \binom{n-1}{N-1},
\]
where the only non-trivial inequality $\ell \leq \tlcyr$ follows from Theorem~\ref{thmintro2}. 

\medskip

Quite remarkably, these sequences intrinsically encode information about prime numbers -- a fact that is far from obvious from their definitions. 
To reveal their underlying arithmetic and combinatorial structure, we shall study them through suitable generating functions, obtained by collecting the above counting data in various ways.

We begin by extending their definition to the degenerate case $N=1$, by setting
\[
\lcyr(n,1)=\tlcyr(n,1)=\ell(n,1)=1, \qquad n \geq 1.
\]
Among the three, the sequence $\lcyr$ turns out to exhibit the richest arithmetic behaviour.  
Its key feature is a Pascal–type identity,
\[
\lcyr(n,N) + \lcyr(n,N+1) = \lcyr(n+1,N+1), \qquad 2 \leq N \leq n,
\]
which endows $\lcyr$ with the combinatorial structure of a genuine Pascal–Tartaglia triangle.  
As a consequence, the entire array $\{\lcyr(n,N)\}_{n,N}$ can be reconstructed recursively from the single sequence $\lcyr(n,2)$, according to
\[
\sum_{k=N}^n \lcyr(k,N) = \lcyr(n+1,N+1), \qquad N \geq 2.
\]
Thus, the double sequence $\lcyr$ is completely determined by its second column, in perfect analogy with the binomial triangle. See Section \ref{OGFPascal} for detailed proofs.

\medskip

This recursive behaviour extends naturally to the generating functions associated with $\lcyr$.  
For each $N \geq 1$, we consider both the ordinary and Dirichlet generating series:
\[
\Lcyrit_N(z):=\sum_{n=N}^\infty \lcyr(n,N)z^n,
\qquad 
\Lcyr_N(s):=\sum_{n=N}^\infty \frac{\lcyr(n,N)}{n^s}.
\]
They satisfy
\[
\Lcyrit_1(z)=\frac{z}{1-z}, \qquad \Lcyr_1(s)=\zeta(s)\text{ (Riemann zeta function)},
\]
and the recursion
\[
\Lcyrit_N(z)=\Lcyrit_2(z)\Lcyrit_1(z)^{N-2}, \qquad 
\Lcyrit_{N+1}(z)=\Lcyrit_N(z)\Lcyrit_1(z), \qquad N\geq 2.
\]
Introducing $\widehat{\Lcyr}_N(s):=\Gamma(s)\Lcyr_N(s)$, one may equivalently write
\[
\widehat{\Lcyr}_{N+1}(s)=\bigl(\widehat{\Lcyr}_N*\widehat{\Lcyr}_1\bigr)(s), \qquad N\geq 2,
\]
where $*$ denotes convolution along a vertical line within the common domain of holomorphy of $\Lcyr_1$ and $\Lcyr_N$. See Theorem \ref{thmHaL} and Theorem \ref{iterint} for more details.

\medskip

The Dirichlet series $\Lcyr_N(s)$ reveal an unexpectedly rich arithmetic structure.  
They intertwine divisor statistics, prime factorizations, and additive properties of integers in subtle ways.  
For instance, when $N=2$, one finds an explicit identity linking $\Lcyr_2(s)$ to classical arithmetic functions.  
If we let $\zeta(s)$ be the Riemann zeta function, then
\[
\Lcyr_2(s)\zeta(s)
=2\!\left(\sum_{n=1}^\infty\frac{d(n)}{n^s}\right)
\!\!\sum_{\substack{p\text{ prime}}}\frac{p-1}{p^s}\prod_{\substack{q\text{ prime}\\ q< p}}\left(1-\frac{1}{q^s}\right)%{p^s\cdot\text{\textnormal{\textcyr{p}}}(s,p-1)}
+\sum_{n=1}^\infty\frac{\omega_0(n)}{n^s}
-\sum_{n=1}^\infty\frac{\omega_1(n)}{n^s},
\]
where $d(n)$ counts the divisors of $n$, while $\omega_0(n)$ and $\omega_1(n)$ respectively count and sum the distinct prime factors of $n$ (Proposition \ref{identityLcyr2}).  
This formula exemplifies how the $\Lcyr_N(s)$, though defined through geometric data, encode deep arithmetic information about divisibility and the distribution of primes.

\medskip

Analytically, the functions $\Lcyr_N(s)$ admit a meromorphic continuation beyond their line of absolute convergence ${\rm Re}(s)=N$.  
Their analytic continuation is far from regular: it exhibits a dense pattern of logarithmic singularities, located at points determined by the non-trivial zeros of the Riemann zeta function.  
The following theorem describes this phenomenon.

\begin{thm}[Thm.\,\ref{mthLcyrN}]
For every $N\geq 2$, the function $\Lcyr_N(s)$ is holomorphic on the line ${\rm Re}(s)=N$, except at $s=N$.  
Moreover, in a neighbourhood of $s=N$ and for ${\rm Re}(s)>N$, one has
\[
\Lcyr_N(s)\sim \frac{1}{(N-1)!}\log\!\left(\frac{1}{s-N}\right),\qquad s\to N.
\]
By analytic continuation, $\Lcyr_N(s)$ extends to the universal cover of the punctured half-plane
\[
\{s\in\C : {\rm Re}(s)>\bar{\sigma}+N-2\}\setminus Z_N,
\]
where $\bar{\sigma}\in[1,\tfrac{3}{2}]$ is defined in \eqref{barsi}, and 
\[
Z_N=\Bigl\{\,s=\tfrac{\rho}{k}+N-1 \;\big|\;
\rho\text{ zero or pole of }\zeta(s),\,
k\text{ squarefree positive integer}\Bigr\}.
\]
\end{thm}

In particular, the Riemann Hypothesis can be reformulated in terms of the analytic behaviour of the functions $\Lcyr_N(s)$ (Corollary~\ref{coRH}).

\medskip

The asymptotic behaviour of $\Lcyr_N(s)$ near $s=N$ implies
\[
\lcyr(n,N)\sim \frac{1}{(N-1)!}\frac{n^{N-1}}{\log n},\qquad N\geq 3.
\]
Hence, the set of partial flag varieties possessing at least one non-exceeding semiclassical spectrum has density zero within the set of all partial flag varieties. See Corollary \ref{cordensity}.

\medskip

The rigidity induced by the Pascal identity also produces a striking combinatorial phenomenon:  
for any fixed integer $k$, the sequence $N\mapsto \lcyr(N+k,N)$ is eventually polynomial.  
Even more surprisingly, the same property holds for $\tlcyr$ and $\ell$, despite the absence of any comparable recursive structure.

\begin{thm}[Thms.\,\,\ref{thmpolylcyr}, \ref{thmpolytlcyr}, \ref{thmpolyell}]
For every integer $k$, there exist three polynomials $P_k, \widetilde{P}_k, \mathscr{P}_k \in \mathbb{Q}[n]$ and integers $N_1(k), N_2(k), N_3(k)$ such that
\begin{align*}
\lcyr(N+k,N)&=P_k(N),\quad N\geq N_1(k),\\
\tlcyr(N+k,N)&=\widetilde{P}_k(N),\quad N\geq N_2(k),\\
\ell(N+k,N)&=\mathscr{P}_k(N),\quad N\geq N_3(k).
\end{align*}
\end{thm}

While the eventual polynomiality of $\lcyr$ follows naturally from its recursive structure, in the cases of $\tlcyr$ and $\ell$ the phenomenon is far more elusive.  
Here, the proof relies on an interpretation of $\tlcyr(n,N)$ and $\ell(n,N)$ as counting numbers for weighted walks on two suitable graphs endowed with fixed monomial weights. See Section \ref{secgraphGm} and Section \ref{secgraphPim}.

\medskip

Finally, the double sequence $\ell$ exhibits a remarkably subtle arithmetic behaviour.  
For certain pairs $(n,N)$, one finds $\ell(n,N)=0$.  
For example,
\begin{multline*}
\ell(n,2)=0 \ \text{whenever $n$ is not prime}, \\ 
\ell(11,4)=\ell(17,4)=\ell(23,4)=\ell(29,4)=\ell(35,4)=0,\dots
\end{multline*}
This vanishing pattern reveals a deep connection between $\ell$ and additive prime number theory, leading to equivalent formulations of Goldbach’s conjecture.  
Indeed, we show that $\ell(n,N)$ can vanish only for $N=2,4,6$, and that:
\begin{itemize}
\item for $N=2$, $\ell(n,2)=0$ if and only if $n$ is not prime;
\item for $N=4$, $\ell(n,4)=0$ if and only if $n$ cannot be written as a sum of two primes, so that Goldbach’s conjecture is equivalent to the statement that $\ell(n,4)=0$ implies that $n$ is odd;
\item for $N=6$, the non-vanishing of $\ell(n,6)$ for all $n$ is again equivalent to Goldbach’s conjecture.
\end{itemize}

See Theorem \ref{thmsum1}, Theorem \ref{thmell7+.1}, and Theorem \ref{thmequivgoldconj}.
\medskip

\noindent1.6.\,\,{\bf Structure of the paper}

Section~\ref{sec2} reviews the fiberwise Gromov--Witten theory as developed by Astashkevich--Sadov and by Biswas--Das--Oh--Paul. 
We formulate several functorial and inductive properties of vertical quantum cohomology that will be used throughout the paper.

Section~\ref{sec3} is devoted to the study of vertical quantum spectra of flag bundles, and to their connection with the semiclassical spectra of partial flag varieties. 
This section contains the statements and proofs of the main theorems.

Sections~\ref{sec4}--\ref{sec6} introduce the three double sequences $\lcyr$, $\tlcyr$, and $\ell$, and investigate their combinatorial and arithmetic properties through both ordinary and Dirichlet generating functions. 
In particular, we relate these sequences to several aspects of prime number theory, unveiling unexpected links between enumerative geometry and arithmetic phenomena.

Appendix~A collects well-known facts on the cohomology of fiber bundles, 
while Appendix~B recalls basic definitions and identities for double Schubert polynomials.

Appendix~C summarizes results of A.\,Varchenko and V.\,Tarasov relating dynamical operators, stable envelopes, and quantum products in the cohomology of flag varieties.

Appendix~D gathers background material on generating functions and Mellin transform techniques.  

Finally, Appendix~E recalls a classical theorem of E.\,Fabry and E.\,Lindel\"of on boundary singularities of power series, and presents numerical evidence suggesting that the generating functions $\Lcyrit_N(z)$ admit a natural boundary.

\medskip

\noindent{\bf Acknowledgements.}
The author wishes to thank P.\,Jossen, D.\,Masoero, M.\,Mendes Lopes, L.\,Monsaingeon, Y.-G.\,Oh, T.\,Pantev, V.\,Roubtsov, G.\,Ruzza, C.\,Sabbah, and A.\,Varchenko for valuable discussions. The author is also grateful to the Institute for Basic Science, Center for Geometry and Physics at Pohang University of Science and Technology, for its hospitality during July 2025, when a substantial part of this work was carried out. This research was supported by the FCT -- Portuguese national funding, UID/00208/2025.

\section{Fiberwise Gromov--Witten theory, and vertical quantum cohomology}\label{sec2}
\subsection{Fiberwise Gromov--Witten theory}Consider the datum of three smooth complex projective varieties $E,F,B$, and assume $\pi\colon E\to B$ is an $F$-fiber bundle over $B$. Following the works \cite{AS95,BDOP25}, a relative version of Gromov--Witten theory can be developed, in order to count curves on $E$ satisfying suitable incidence conditions and {\it vertical} with respect to the fibration $\pi$. 

Namely, given $g, n \geq 0$ and $\beta \in H_2(F, \mathbb{Z})$, one aims to define a moduli space parametrizing stable maps $f \colon (C, \bm{p}) \to E$ such that $\pi \circ f$ is constant, and $f_*[C] = \beta$, by identifying $F$ with the fiber containing the image of $f$. Such a moduli space (when it exists) is expected to be a $\overline{\mc{M}}_{g,n}(F,\beta)$-fiber bundle over $B$.

As pointed out in \cite{BDOP25}, this setup presents the following issue. Given two local trivializations of the fibre bundle $E$, namely $\phi_i\colon \pi^{-1}(U_i)\cong U_i\times F$ with $i=1,2$ and $U_1,U_2\subseteq B$ open sets with non-empty intersection,  the map $\phi_2\circ \phi_1^{-1}\colon U_1 \cap U_2\to{\rm Aut}(F)$ induces an isomorphism $\overline{\mc M}_{g,n}(F,\bt)\cong \overline{\mc M}_{g,n}(F,[\phi_2\circ \phi_1^{-1}(b)]_*\bt)$ for any $b\in U_1\cap U_2$. Thus, the transition maps of the desired moduli space may depend on the choice of local trivializations of $\pi$. To avoid this, we impose the following assumption:

\noindent{\bf Assumption G:} $E$ is an $F$-fiber bundle over $B$ with structure group
\[G:=\{\phi\in{\rm Aut}(F)\colon \phi_*\in {\rm Aut}(H_2(F,\Z))\text{ is the identity map}\}.
\]
\vskip2mm
\begin{rem}If $B$ is simply connected, then Assumption G holds.
Also, if the automoprhism group ${\rm Aut}(F)$ is connected, then Assumption G holds, as any $\phi\in {\rm Aut}(F)$ is isotopic to the identity.
\qrem\end{rem}

Under Assumption G, a detailed construction of the desired moduli space, denoted here by $\overline{\mathcal{M}}_{g,n}^{\mathrm{Fib}}(E,B,F,\beta)$, is given in \cite{BDOP25}, along with its key properties. We summarize the main result as follows.

\begin{thm}\label{fibervirt}\cite{BDOP25}
If Assumption G holds, given $g,n\geq 0$ and an effective class $\bt\in H_2(F,\Z)$, the moduli space $\overline{\mc M}_{g,n}^{\rm Fib}(E,B,F,\bt)$ exists as a proper Deligne--Mumford stack, and it defines a $\overline{\mc M}_{g,n}(F,\bt)$-bundle over $B$. Moreover, it comes equipped with a natural virtual fundamental class
\begin{multline*}[\overline{\mc M}_{g,n}^{\rm Fib}(E,B,F,\bt)]^{\rm virt}\in A_{D}\left(\overline{\mc M}_{g,n}^{\rm Fib}(E,B,F,\bt)\right)\otimes_\Z\Q,\\ D:=\dim B+ \int_\bt c_1(F)+(\dim F-3)(1-g)+n.\tag*{\qed}
\end{multline*}
\end{thm}

The moduli space is naturally equipped with evaluation maps ${\rm ev}_i\colon \overline{\mc M}_{g,n}^{\rm Fib}(E,B,F,\bt)\to E$, $i=1,\dots,n$, mapping the point $[f\colon (C,\bm p)\to E]$ to $f(p_i)$. We can thus introduce the following fiberwise version of Gromov--Witten invariants.

\begin{defn}
Given cycles $\gm_1,\dots,\gm_n\in H^\bullet(E,\Q)$, the {\it fiberwise Gromov--Witten invariant} $\langle\gm_1,\dots,\gm_n\rangle_{g,n,\bt}^{{\rm Fib}}\in \Q$ is the rational number
\beq\label{fibGW}
\langle\gm_1,\dots,\gm_n\rangle_{g,n,\bt}^{{\rm Fib}}:=\int_{[\overline{\mc M}_{g,n}^{\rm Fib}(E,B,F,\bt)]^{\rm virt}}\prod_{i=1}^n{\rm ev}_i^*\gm_i.
\eeq
\end{defn}

\begin{rem}
The invariant $\langle\gm_1,\dots,\gm_n\rangle_{g,n,\bt}^{{\rm Fib}}$ vanishes unless $\sum_{i=1}^n\deg(\gm_i)=2D$ -- where $\deg$ denotes the cohomological degree, $D$ is the virtual dimension in Theorem \ref{fibervirt} -- and $\bt$ is an effective class (i.e.\,\,represented by an algebraic/holomorphic curve). 
\qrem\end{rem}

\begin{rem}
If $B={\rm Spec}(\C)$, the moduli space $\overline{\mc M}_{g,n}^{\rm Fib}(E,B,F,\bt)$ equals the classical moduli space $\overline{\mc M}_{g,n}(E,\bt)=\overline{\mc M}_{g,n}(F,\bt)$, the virtual fundamental class $[\overline{\mc M}_{g,n}^{\rm Fib}(E,B,F,\bt)]^{\rm virt}$ coincides with the Behrend--Fantechi \cite{BF97} and Li--Tian \cite{LT98} classes, and the fiberwise Gromov--Witten invariant $\langle\gm_1,\dots,\gm_n\rangle_{g,n,\bt}^{{\rm Fib}}$ equals the classical one $\langle\gm_1,\dots,\gm_n\rangle_{g,n,\bt}$.
\qrem\end{rem}

\begin{rem}
To the best of our knowledge, the first discussion of the fiberwise (or vertical) version of Gromov--Witten theory for a fibration appears in \cite{AS95}, at least in the genus-zero sector. That work outlines the expected properties of the moduli space $\overline{\mc M}_{0,n}^{\rm Fib}(E,B,F,\bt)$, although no stack-theoretic or algebro-geometric detailed constructions are provided. 
In particular, rather than defining a virtual fundamental class on $\overline{\mc M}_{0,n}^{\rm Fib}(E,B,F,\bt)$, A.\,Astashkevich and V.\,Sadov introduce an integration along the fibers morphism $\tau_!$ associated with the projection $\tau \colon \overline{\mc M}_{0,n}^{\rm Fib}(E,B,F,\bt) \to B$, and essentially define their {\it vertical} Gromov--Witten invariants as $\tau_!\left(\prod_{i=1}^n{\rm ev}_i^*\gm_i\right)\in H^\bullet(B,\Q)$.
The fiberwise Gromov--Witten invariant $\langle\gm_1,\dots,\gm_n\rangle_{0,n,\bt}^{\rm Fib}$ from equation \eqref{fibGW} is then related to the Astashkevich--Sadov invariant via the ``Fubini formula'':
\beq\label{Fubini}
\langle\gm_1,\dots,\gm_n\rangle_{0,n,\bt}^{\rm Fib} = \int_B \tau_!\left(\prod_{i=1}^n {\rm ev}_i^* \gm_i\right),
\eeq
see \cite[formulas (3.3) and (3.4)]{AS95}.
\qrem\end{rem}

The fiberwise Gromov--Witten invariants $\langle\gm_1,\dots,\gm_n\rangle_{g,n,\bt}^{{\rm Fib}}$ satisfy analogues of the classical axioms such as the string, divisor, point mapping, and splitting axioms. We highlight below a few key properties. Further details and complete proofs can be found in \cite[Sec.\,3.3]{BDOP25}.

\begin{prop}\label{axioms}\cite{AS95,BDOP25}
 For any $\gm_1,\dots,\gm_n\in H^\bullet(E,\Q)$ we have:
\begin{enumerate}
\item If $n>3$ or $\bt\neq 0$, then $\langle\pi^*\dl,\gm_2,\dots,\gm_n\rangle^{\rm Fib}_{g,n,\bt}=0$ for any $\dl\in H^\bullet(B,\Q)$.
\item More generally, in terms of the natural projection $\tau \colon \overline{\mc M}_{g,n}^{\rm Fib}(E,B,F,\bt) \to B$, we have
\[\langle\pi^*\dl\cdot\gm_1,\gm_2,\dots,\gm_n\rangle^{\rm Fib}_{g,n,\bt}=\int_{[\overline{\mc M}_{g,n}^{\rm Fib}(E,B,F,\bt)]^{\rm virt}}\tau^*\dl\cdot \prod_{i=1}^n{\rm ev}_i^*\gm_i.
\]
\item If $n>3$ or $\bt\neq 0$, and if $\iota_b^*\gm_1\in H^2(F,\Q)$ where $\iota_b\colon F\hookrightarrow E$ is an identification of $F$ with the fiber $\pi^{-1}(b)$, with $b\in B$, then 
\[\langle\gm_1,\gm_2,\dots,\gm_n\rangle^{\rm Fib}_{g,n,\bt}=\left(\int_\bt\iota_b^*\gm_1\right)\langle\gm_2,\dots,\gm_n\rangle^{\rm Fib}_{g,n-1,\bt}.
\]
\item If $n>3$, we have $\langle\gm_1,\dots,\gm_n\rangle^{\rm Fib}_{0,n,0}=0$. Moreover, $\langle\gm_1,\gm_2,\gm_3\rangle^{\rm Fib}_{0,3,0}=\int_E\gm_1\gm_2\gm_3$. \qed
\end{enumerate}
\end{prop}

\begin{rem}
The definition of fiberwise Gromov--Witten invariants naturally extends to classes in $H^\bullet(B, R)$ for any $\Q$-algebra $R$. The properties above still hold in this more general setting.
\qrem\end{rem}

\subsection{Vertical quantum cohomology} Let $\pi\colon E\to B$ be a locally trivial algebraic (or holomorphic) $F$-bundle satisfying Assumption G, as in the previous section. For simplicity, we also assume:

\noindent{\bf Assumption F:} for any $n\in\N$, there exist a finite number of effective classes $\bt$ such that $\int_\bt c_1(F)\leq n$. 
\begin{rem}\label{remF}If the fiber $F$ is Fano or a homogeneous space, then Assumption F automatically holds. See the arguments of \cite[Lemma 15]{FP97}\cite[Prop.\,8.1.3]{CK99}.
\qrem\end{rem}

\subsubsection{Big vertical quantum cohomology} Fix a basis $(T_0=1,T_1,\dots, T_N)$ of $H^\bullet(E,\C)$, and denote by $\bm t=(t^0,\dots, t^N)$ the dual coordinates. Denote by $\eta\colon H^\bullet(E,\C)\times H^\bullet(E,\C)\to\C$ the $\C$-bilinear non-degenerate Poincar\'e pairing, with Gram matrix ${\bm \eta}:=\left(\eta_{ij}\right)_{ij=0}^N$, $\eta_{ij}:=\int_ET_iT_j$, and inverse matrix ${\bm\eta}^{-1}=\left(\eta^{ij}\right)_{i,j=0}^N$.
\vskip2mm
The {\it big fiberwise} (or {\it vertical}) {\it quantum cohomology} of $(E,B,F)$ is the algebra structure $\left(H^\bullet(E,\C[\![\bm t]\!]),\,\bqf\right)$ defined by
\beq\label{qprod}
T_i\bqf T_j=\sum_{n\geq 0}\sum_{\al_1,\dots,\al_n=0}^N\sum_{h,\ell=0}^N\sum_\bt\frac{t^{\al_1}\dots t^{\al_n}}{n!}\langle T_{\al_1},\dots,T_{\al_n},T_i,T_j,T_h\rangle^{\rm Fib}_{0,n+3,\bt}\eta^{h\ell}T_{\ell}.
\eeq
Assumption F ensures that, for fixed $i$, $j$, $n$, $\alpha_1,\dots,\alpha_n$, $h$, and $\ell$, the sum over $\bt$ is finite, so the product $\bqf$ is well defined.

\begin{thm}\cite{AS95,BDOP25}
The algebra $\left(H^\bullet(E,\C[\![\bm t]\!]),\,\bqf,\,\eta\right)$ is a Frobenius super-algebra: it is super-commutative, associative, unital (with unit $T_0=1$), and the product is compatible with the Poincar\'e pairing, namely
\[\eta\left(x_1\bqf x_2,x_3\right)=\eta\left(x_1,x_2\bqf x_3\right).\tag*{\qed}
\]
\end{thm}

We will denote ${\rm QH}^{\rm Fib}_{\rm big}(E,B,F)$ this Frobenius super-algebra. 

When $B={\rm Spec}(\C)$, the Frobenius super-algebra above defines the (ordinary) {\it big quantum cohomology} ${\rm QH_{big}}(E)$ of $E$.

\begin{rem}
We have $T_i\bqf T_j=T_i\cup T_j+O(\bm t)$, by Proposition \ref{axioms}. Thus the $\bqf$-product defines a deformation of the classical cohomological Frobenius super-algebra $(H^\bullet(E,\C),\cup,\eta)$.
\qrem\end{rem}

Consider the pullback morphism $\pi^*\colon H^\bullet(B,\C[\![\bm t]\!]) \to H^\bullet(E,\C[\![\bm t]\!])$, in general not injective.
\begin{thm}\label{thm1}
$\quad$
\begin{enumerate}
\item If $a_1$ or $a_2$ lies in $\pi^*H^\bullet(B,\C[\![\bm t]\!])$, then $a_1 \bqf a_2 = a_1 \cup a_2$.
\item The pullback $\pi^*$ induces an isometric morphism of Frobenius super-algebras:
\[
(H^\bullet(B,\C[\![\bm t]\!]),\cup,\eta_B) \rightarrow (H^\bullet(E,\C[\![\bm t]\!]),\bqf,\eta_E).
\]
Thus, the big vertical quantum cohomology ${\rm QH}^{\rm Fib}_{\rm big}(E)$ carries a natural $H^\bullet(B,\C[\![\bm t]\!])$-algebra structure.
\end{enumerate}
\end{thm}
\proof
The fiberwise Gromov--Witten invariant
$\langle \pi^*\delta, T_{i_1}, \dots, T_{i_h} \rangle^{\rm Fib}_{0,h+1,\beta},
$
with $\delta \in H^\bullet(B,\C)$, is nonzero only if $h = 2$ and $\beta = 0$, in which case it reduces to a classical triple intersection:
\[
\langle \pi^*\delta, T_{i_1}, T_{i_2} \rangle^{\rm Fib}_{0,3,0} = \int_E \pi^*\delta \cdot T_{i_1} \cdot T_{i_2}.
\]
This proves point (1). For point (2), we have
\[
\int_E \pi^*\tau \cdot \omega = \int_B \tau \cdot \pi_*\omega, \quad \text{for } \tau \in H^\bullet(B,\C),\ \omega \in H^\bullet(E,\C),
\]see Appendix \ref{projApp}. 
Setting $\omega = \pi^*\tau'$, we find that $\pi^*$ preserves the Poincar\'e pairings $\eta_B$ and $\eta_E$, hence it is an isometry.
This completes the proof of point (2).
\endproof

We say that the fiber bundle $(E, B, F)$ is \emph{cohomologically decomposable} if  
\[
H^\bullet(E, \C) \cong H^\bullet(B, \C) \otimes_\C H^\bullet(F, \C)
\]
as $H^\bullet(B, \C)$-modules (not necessarily as rings). This holds if and only if there exist classes $e_1, \dots, e_k \in H^\bullet(E, \C)$ such that:
\begin{itemize}
  \item their restrictions $\iota_b^* e_1, \dots, \iota_b^* e_k$ form a basis of $H^\bullet(F, \C)$ for every fiber $\pi^{-1}(b)$,
  \item and every class in $H^\bullet(E, \C)$ can be uniquely written as $\sum_{j=1}^k \pi^* b_j \cup e_j$ for suitable $b_j \in H^\bullet(B, \C)$, see Theorem~\ref{equivs}.
\end{itemize}

\begin{rem}\label{remG}
If $(E,B,F)$ is cohomologically decomposable, then assumption G holds. Moreover, if $B$ is simply connected, then $(E,B,F)$ is automatically cohomologically decomposable. See Theorem \ref{equivs}.
\qrem\end{rem}

\begin{thm}\label{elemprod}
If $(E, B, F)$ is cohomologically decomposable, then the map $\pi^*$ injects $(H^\bullet(B,\C), \cup, \eta_B)$ into ${\rm QH}^{\rm Fib}_{\rm big}(E,B,F)$ as a subalgebra. Moreover, the big quantum $\bqf$-product is uniquely determined by the fiberwise products $e_i \bqf e_j$.
\end{thm}

\begin{proof}
Cohomological triviality implies the injectivity of $\pi^*$, see Theorem~\ref{piinj}.  
Let $e_1, \dots, e_k \in H^\bullet(E, \C)$ be as above. Then every class in $H^\bullet(E, \C)$ has a unique expression $\sum_j \pi^* b_j \cup e_j$.

Consider the product $(\pi^* b_i \cup e_i) \bqf (\pi^* b_j \cup e_j)$. By associativity and super-commutativity of $\bqf$, and using Theorem~\ref{thm1}(1), we compute:
\begin{multline*}
(\pi^* b_i \cup e_i) \bqf (\pi^* b_j \cup e_j) 
= (\pi^* b_i \bqf e_i) \bqf (\pi^* b_j \bqf e_j) \\
= (-1)^{\eps_{ij}} (\pi^* b_i \bqf \pi^* b_j) \bqf (e_i \bqf e_j) 
= (-1)^{\eps_{ij}} (\pi^* b_i \cup \pi^* b_j) \cup (e_i \bqf e_j),
\end{multline*}
where the sign $(-1)^{\eps_{ij}}$ accounts for super-commutativity.

This shows that the $\bqf$-product on $H^\bullet(E, \C)$ is fully determined by the fiberwise products $e_i \bqf e_j$.
\end{proof}

\subsubsection{Small vertical quantum cohomology}\label{secsqc} Let us now consider a ``restriction to the $H^2$-locus'' of the quantum product $\bqf$, under the following mild additional assumption:

\medskip
\noindent{\bf Assumption G':} For any element $\phi$ of the monodromy group ${\rm im}(\pi_1(B)\to G)$, the induced morphism $\phi^*\in{\rm Aut}(H^2(F,\C))$ is the identity.

\begin{rem}\label{remG'}
Assumption G' is clearly independent of the choice of base point $b \in B$ used to define $\pi_1(B,b)$. Moreover, if $(E,B,F)$ is cohomologically decomposable then Assumption G' automatically holds. 
\qrem\end{rem}

Deligne's theorem \cite{Del68} asserts that the Leray spectral sequence associated with the fibration $\pi\colon E \to B$ degenerates at $E_2$, yielding the decomposition
\[
H^2(E,\C)\cong H^2(B,\C)\oplus H^1(B,R^1\pi_*\underline{\C})\oplus H^0(B,R^2\pi_*\underline{\C}).
\]
Assumption G' implies that the space of global sections of the local system $R^2\pi_*\underline{\C}$ -- naturally identified with the monodromy-invariant subspace $H^2(F,\C)^{\pi_1(B)} \subseteq H^2(F,\C)$ -- coincides with the full cohomology group $H^2(F,\C)$. Consequently, for any $b \in B$, the canonical restriction map $\iota_b^*\colon H^2(E,\C) \to H^2(F,\C)$
is surjective. See Appendix \ref{Deligneapp}.

Let us now introduce another mild additional assumption:

\noindent{\bf Assumption F':} The fiber $F$ is simply connected.

\begin{rem}\label{remF'}
If $F$ is Fano, then Assumption F' automatically holds. 
\qrem\end{rem}

Fix bases $\tilde T_1,\dots,\tilde T_k$ of $H^2(F,\C)$ and $\tilde T_{k+1},\dots,\tilde T_{k+m}$ of $H^2(B,\C)$. Under Assumption F', and by the surjectivity established above, we can choose a basis $T_1,\dots,T_{k+m}$ of $H^2(E,\C)$ such that
\[
\iota^* T_i = \tilde T_i,\quad i = 1,\dots,k, \qquad \pi^* \tilde T_j = T_j,\quad j = k+1,\dots,k+m.
\]

We now specialize the r.h.s.\,\,of \eqref{qprod} to those tuples $\bm t = (t^i)_{i=0}^N$ for which $t^i = 0$ unless $i = 1,\dots,k+m$. By the string and divisor properties of fiberwise Gromov--Witten invariants (points (1) and (3) in Prop.~\ref{axioms}), the infinite sum reduces to a finite one:
\begin{multline}\label{restrqprod}
\left. T_i \bqf T_j \right|_{\substack{t^i=0 \\ i \notin \{1,\dots,k+m\}}}
= \sum_{\ell=0}^N c^\ell_{ij}(t^1,\dots,t^k)\, T_\ell,\\
c^\ell_{ij}(t^1,\dots,t^k)
:= \sum_{d=0}^N \sum_{\bt} \exp\left(\sum_{u=1}^k t^u \int_{\bt} \tilde T_u \right)
\langle T_i, T_j, T_d \rangle^{\rm Fib}_{0,3,\bt} \, \eta^{d\ell}.
\end{multline}
Each sum $\sum_{\bt}$ has finite support, by Assumption~F.

\begin{rem}\label{remperiodicity}
Since $\bt$ is an effective class, the integral $\int_\bt \tilde T_u$ is non-zero only if $\tilde T_u$ is a $(1,1)$-class. Therefore, without loss of generality, we may assume that $\tilde T_1, \dots, \tilde T_k$ form a basis for the subspace $H^{1,1}(F, \C) \subseteq H^2(F, \C)$. Moreover, we can further assume, still without loss of generality, that $\tilde T_1, \dots, \tilde T_k$ lie in the lattice $H^{1,1}(F, \C) \cap H^2(F, \Z)$. 
With such a choice, the structure constants $c^\ell_{ij}$ appearing in \eqref{restrqprod} satisfy the periodicity conditions
\begin{equation}\label{periodicity}
c^\ell_{ij}(t^1, \dots, t^a + 2\pi\sqrt{-1}, \dots, t^k) = c^\ell_{ij}(t^1, \dots, t^k), \quad a = 1, \dots, k.
\end{equation}
Furthermore, if $F$ is Fano, we can refine our choice even more: the classes $\tilde T_1, \dots, \tilde T_k$ can be taken in both $H^{1,1}(F, \C) \cap H^2(F, \Z)$ and in the NEF cone. In this case, we also have $\int_\bt \tilde T_u \in \Z_{\geq 0}$ for all $u = 1, \dots, k$. This follows from the Mori Cone Theorem \cite{KM98, Laz04}.
\qrem\end{rem}

Let ${\rm Eff}_1(F)\subseteq H_2(F,\Z)$ be the cone of effective 1-cycles (the additive semigroup generated by homological classes of effective algebraic curves on $F$), and introduce the semigroup ring $\La_F:=\C[{\rm Eff}_1(F)]=\C[{\bf q}^\bt\colon \bt\in {\rm Eff}_1(F)]$, where $\bf q$ is an indeterminate. Formula \eqref{restrqprod} suggests the following definition.

\begin{defn}The {\it small vertical quantum cohomology} ${\rm QH}^{\rm Fib}(E,B,F)$ is the Frobenius super-algebra structure $\left(H^\bullet(E,\La_F),\sqf,\eta_E\right)$ defined by the small product
\beq\label{sqprod}
T_i \sqf T_j:=\sum_{d,\ell=0}^N\sum_\bt {\bf q}^\bt\langle T_i, T_j, T_d \rangle^{\rm Fib}_{0,3,\bt} \, \eta^{d\ell} T_\ell,\qquad i,j=0,\dots,N.
\eeq
When $B={\rm Spec}(\C)$, we obtain the (ordinary) {\it small quantum cohomology} ${\rm QH}(E)$.
\end{defn}
For each $\gm\in \iota^*H^2(F,\C)$, the ring morphism $\La_F\to\C$ defined by the evaluation ${\bf q}^\bt\mapsto \prod_{i=1}^k e^{\int_\bt \gm}$ %-- indepdendent on the choice of $(\tilde{T}_1,\dots,\tilde{T}_k)$ -- 
induces a family of products on $H^\bullet(E,\C)$, identified with \eqref{restrqprod}, and labelled by points of $\iota^*H^2(E,\C)$.

Introduce the torus $(\C^*)^k$, with coordinates $\bm q=(q_1,\dots,q_k)$. If we choose the basis $\tilde T_1,\dots, \tilde T_k$ in $H^2(F,\Z)$, as in Remark \ref{remperiodicity}, for each point $\bm q \in (\C^*)^k$, we have a well-defined product $\sqqf{\bm q}$ on $H^\bullet(E,\C)$, defined by
\beq\label{smallqprodq}
T_i\sqqf{\bm q} T_j=\sum_{d,\ell=0}^N\sum_\bt\bm q^\bt\langle T_i, T_j, T_d \rangle^{\rm Fib}_{0,3,\bt} \, \eta^{d\ell} T_\ell,\qquad \bm q^\bt:=\prod_{i=1}^k q_i^{\int_\bt \tilde T_i}.
\eeq
Notice that $\bm q\mapsto \sqqf{\bm q}$ is well-defined, thanks to the periodicity condition \eqref{periodicity}.
We thus obtain a trivial bundle $(\C^*)^k \times H^\bullet(E,\C) \to (\C^*)^k$, whose fibers carry a Frobenius super-algebra structure: the {\it small fiberwise} (or {\it vertical}) {\it quantum cohomology} over any point $\bm q \in (\C^*)^k$ is the Frobenius super-algebra
\[
\left(H^\bullet(E,\C),\,\sqqf{\bm q},\,\eta_E\right),
\]
which we denote by ${\rm QH}^{\rm Fib}_{\bm q}(E,B,F)$.

\begin{defn}\label{vqcp}
The {\it vertical quantum characteristic polynomial}  of $(E,B,F)$ at $\bm q\in(\C^*)^{k}$ is the characteristic polynomial $f_{(E,B,F)}(-;\bm q)\in\C[\zeta]$ of the $\C$-linear operator
\[c_1(E)\sqqf{\bm q}\colon H^\bullet(E,\C)\to H^\bullet(E,\C),\qquad\text{that is }f_{(E,B,F)}(\zeta;\bm q):=\det(\zeta\cdot{\rm Id}-c_1(E)\sqqf{\bm q}).
\]The {\it vertical quantum spectrum} of $(E,B,F)$ at $\bm q\in(\C^*)^{k}$ is the multi-set of zeroes of $f_{(E,B,F)}(\zeta;\bm q)$. When $B={\rm Spec}(\C)$, we will speak about (ordinary) {\it quantum characteristic polynomial}, simply denoted by $f_E(\zeta;\bm q)$, and {\it quantum spectrum}.
\end{defn}

\subsection{Functoriality properties of small vertical quantum cohomology} If $(E,B,F)$ is a locally trivial algebraic bundle, given an algebraic map $f\colon B'\to B$, we have the Cartesian diagram 
\beq\label{cartesian}
\xymatrix{
E' \ar[r]^{f'} \ar[d]_{\pi'} \ar@{}[dr]|{\lrcorner} &  E\ar[d]^\pi \\
B' \ar[r]_f & B
},\qquad E':=B'\times_BE.
\eeq
\begin{lem}
If $(E,B,F)$ is cohomologically decomposable, then also $(E',B',F)$ is cohomologically decomposable.
\end{lem}
\proof If $e_1,\dots,e_k\in H^\bullet(E,\C)$ restrict to bases on each fiber, the same holds true for $f'^*e_1,\dots, f'^*e_k\in H^\bullet(E',\C)$.
\endproof
For each $F$ Fano or homogeneous space, define the category $\text{{\sc Fib}}_F$ whose objects are cohomologically decomposable $F$-fiber bundles $(E,B,F)$  satisfying Assumptions G,G' and F', and whose morphisms are Cartesian diagrams \eqref{cartesian}.

Also, introduce the category $\text{\sc Alg}_F$ of super-algebras over the ring $\La_F$, with the natural morphisms.

The small vertical quantum cohomology defines an association of objects 
\[{\rm QH}^{\rm Fib}\colon {\rm Ob} (\text{{\sc Fib}}_F)\to {\rm Ob} (\text{\sc Alg}_F),\quad (E,B,F)\mapsto {\rm QH}^{\rm Fib}(E,B,F). \]
This turns out to be a functor. This remarkable fact was already understood by A. Astashkevich and V.\,Sadov in \cite{AS95}. Here we give more details about the proof.

\begin{thm}\label{thmfunct}
The small vertical quantum cohomology defines a contravariant functor ${\rm QH}^{\rm Fib}\colon \text{{\sc Fib}}_F\to \text{\sc Alg}_F$:
\[\xymatrix@R=0.5em@C=1em{
(E',B',F')\ar[dd]_{(f,f')} & \ar@{|->}[r] & &{\rm QH}^{\rm Fib}(E',B',F') \\
&\ar@{|->}[r]&&\\
(E,B,F)&\ar@{|->}[r] && {\rm QH}^{\rm Fib}(E,B,F) \ar[uu]_{f'^*}
}
\]where $f'^*\colon H^\bullet(E,\La_F)\to H^\bullet(E',\La_F)$ is the pullback.
\end{thm}

For the proof we need a preliminary result.

Consider a cohomologically decomposable bundle $(E,B,F)$, and let $e_1,\dots, e_k\in H^\bullet(E,\C)$ be classes restricting to cohomological bases at each fiber (identifiable with a fixed basis of $H^\bullet(F,\C)$), and $b_1,\dots,b_h$ a basis of $H^\bullet(B,\C)$, so that $\pi^*b_i\cup e_j$, with $i=1,\dots,h$ and $j=1,\dots,k$ define a basis of $H^\bullet(E,\C)$. By Theorem \ref{elemprod}, the $\sqf$-product is uniquely determined by the products $e_i\sqf e_j$.

\begin{lem}\label{fundlemma}
Any elementary product $e_i\sqf e_j$ is a linear combination of $e_1,\dots,e_k$ only.
\end{lem}
\proof
By definition, we have 
\[e_i\sqf e_j=\sum_{\ell,o=1}^h\sum_{m,p=1}^k\sum_\bt {\bf q}^\bt\langle e_i,e_j,\pi^*b_\ell\cup e_m\rangle^{\rm Fib}_{0,3,\bt}\eta^{\ell o,mp}\pi^*b_o\cup e_p.\]The Gromov--Witten invariant $\langle e_i,e_j,\pi^*b_\ell\cup e_m\rangle^{\rm Fib}_{0,3,\bt}$ is nonzero only if
\[\deg(e_i)+\deg(e_j)+\deg(\pi^*b_\ell\cup e_m)=2\left(\dim B+\dim F+\int_\bt c_1(F)\right).
\]Moreover, if $\tau\colon \overline{\mc M}^{\rm Fib}_{0,3}(E,B,F,\bt)\to B$ denotes the $\overline{\mc M}_{0,3}(F,\bt)$-fibration of the moduli space, we have (Fubini formula \eqref{Fubini})
\[\langle e_i,e_j,\pi^*b_\ell\cup e_m\rangle^{\rm Fib}_{0,3,\bt}=\int_B b_\ell\cup \tau_!({\rm ev}_1^*e_i\cup{\rm ev}_1^*e_j\cup{\rm ev}_1^*e_m),
\]and the integral over the fiber $\tau_!({\rm ev}_1^*e_i\cup{\rm ev}_1^*e_j\cup{\rm ev}_1^*e_m)$ is nonzero only if 
\[\deg(e_i)+\deg(e_j)+\deg(e_m)=2\left(\dim F+\int_\bt c_1(F)\right).
\]So $b_\ell$ must be a $(2\dim B)$-degree (top-degree) form. By the tensor decomposition of the Poincar\'e pairing $\eta_E$ as $\eta_B\otimes \eta_F$ (see Corollary \ref{deccorpairing}), we deduce that $b_o\in H^0(B,\C)$.
\endproof

\proof[Proof of Theorem \ref{thmfunct}]
The nontrivial statement to be proved is that morphisms are mapped to morphisms.

As above, consider classes $e_1, \dots, e_k \in H^\bullet(E, \C)$ which restrict to cohomological bases on each fiber (identifiable with a fixed basis of $H^\bullet(F,\C)$), a basis $b_1, \dots, b_h$ of $H^\bullet(B, \C)$, and a basis $b'_1, \dots, b'_h$ of $H^\bullet(B', \C)$. Without loss of generality, we may assume that $b_1 = 1$, $b'_1 = 1$, and that $b_h$ and $b'_h$ are their respective Poincar\'e duals (i.e., top-degree classes), normalized so that
\[
\int_B b_h = 1, \qquad \int_{B'} b'_h = 1.
\]
Let us denote by $\sqf$ and $\Hat{\sqf}$ the products on ${\rm QH^{Fib}}(E,B,F)$ and ${\rm QH^{Fib}}(E',B',F)$, respectively.
We need to prove 
\[
f'^*(e_i \sqf e_j) = f'^*e_i \,\Hat{\sqf}\, f'^*e_j.
\]

By Lemma \ref{fundlemma} and Corollary \ref{deccorpairing}, we have:
\[
f'^*(e_i \sqf e_j) = \sum_{m,p=1}^k \sum_\bt \mathbf{q}^\bt \, \langle e_i, e_j, \pi^*b_h\cup e_m \rangle^{\mathrm{Fib}}_{0,3,\bt} \, \tilde\eta^{mp} f'^*e_p, \qquad \tilde\eta_{mp} = \int_F \iota^*e_m \cup \iota^*e_p.
\]

Since the classes $f'^*e_1, \dots, f'^*e_k \in H^\bullet(E', \C)$ also restrict to cohomological bases on each fiber (identifiable with the same fixed basis of $H^\bullet(F,\C)$ as above), the same lemma and corollary give:
\[
f'^*e_i \,\Hat{\sqf}\, f'^*e_j = \sum_{m,p=1}^k \sum_\bt \mathbf{q}^\bt \, \langle f'^*e_i, f'^*e_j, \pi'^*b'_h\cup f'^*e_m \rangle^{\mathrm{Fib}}_{0,3,\bt} \, \tilde\eta^{mp} f'^*e_p.
\]

The crucial point is that the moduli spaces $\overline{\mc M}^{\mathrm{Fib}}_{g,n}(E', B', F, \bt)$ are equal to the fibered products $\overline{\mc M}^{\mathrm{Fib}}_{g,n}(E, B, F, \bt) \times_B B'$. This follows directly, for instance, from the descriptions given around formulas (2.5), (2.6), and (2.7) in \cite{BDOP25}.

For brevity, let us write $\mc M_n := \overline{\mc M}^{\mathrm{Fib}}_{g,n}(E, B, F, \bt)$ and $\mc M_n' := \overline{\mc M}^{\mathrm{Fib}}_{g,n}(E', B', F, \bt)$, and denote their projections to $B$ and $B'$ by $\tau$ and $\tau'$, respectively. We then have the commutative diagram:
\beq\label{cartmgn}
\xymatrix{
E' \ar[d]_{\pi'} \ar[r]^{f'} & E \ar[d]^{\pi} \\
B' \ar[r]^f & B \\
\mc M_n' \ar[u]^{\tau'} \ar[r]^{\bar{f}} \ar@/^3em/[uu]^{{\rm ev}_i'} & \mc M_n \ar[u]_\tau \ar@/_3em/[uu]_{{\rm ev}_i}
}
\eeq

By compatibility of integration along the fibers with base change (since fibers are preserved under pullback), we obtain:
\beq\label{beautiful}
\tau'_! \, {\rm ev}_i'^* f'^* = f^* \, \tau_! \, {\rm ev}_i^*.
\eeq
By projection formula, we have
\begin{multline}\label{magincint}\langle f'^*e_{j_1}, f'^*e_{j_2}, \pi'^*b'_{h}\cup f'^*e_{j_3} \rangle^{\mathrm{Fib}}_{0,3,\bt}=\int_{B'}\left[b'_h\cup \tau'_!\left(\prod_{i=1}^3{\rm ev}'^*_if'^*e_{j_i}\right)\right]\\=\int_{B'}\left[b_h'\cup f^*\tau_!\left(\prod_{i=1}^3{\rm ev}^*_ie_{j_i}\right)\right].
\end{multline}
Since the restrictions $\iota_b^*e_{j_i}=e_{j_i}|_{F_b}=e_{j_i}|_F\in H^\bullet(F,\C)$ are constant with respect to $b\in B$, the pushforward $\tau_!\left(\prod_{i=1}^3{\rm ev}^*_ie_{j_i}\right)\in H^\bullet(B,\C)$ is the cohomology class give by the constant function
\[b\mapsto \int_{\tau^{-1}(b)}\prod_{i=1}^3{\rm ev}^*_i(e_{j_i}|_{F_b})=\int_{\overline{\mc M}_{0,3}(F,\bt)}\prod_{i=1}^3{\rm ev}^*_i(e_{j_i}|_{F}),
\]and similarly for the pullback $f^*\tau_!\left(\prod_{i=1}^3{\rm ev}^*_ie_{j_i}\right)\in H^\bullet(B',\C)$. By applying Proposition \ref{decompint} to \eqref{magincint}, we conclude:
\begin{multline*}
\langle f'^*e_{j_1}, f'^*e_{j_2}, \pi'^*b'_{h}\cup f'^*e_{j_3} \rangle^{\mathrm{Fib}}_{0,3,\bt}=\left(\int_{B'}b_h'\right)\cdot\left(\int_{\overline{\mc M}_{0,3}(F,\bt)}\prod_{i=1}^3{\rm ev}^*_i(e_{j_i}|_{F})\right)\\
=\left(\int_{B}b_h\right)\cdot\left(\int_{\overline{\mc M}_{0,3}(F,\bt)}\prod_{i=1}^3{\rm ev}^*_i(e_{j_i}|_{F})\right)=\int_{B}\left[b_h\cup \tau_!\left(\prod_{i=1}^3{\rm ev}^*_ie_{j_i}\right)\right]=\langle e_{j_1}, e_{j_2}, \pi^*b_h\cup e_{j_3} \rangle^{\mathrm{Fib}}_{0,3,\bt}.
\end{multline*}
This completes the proof.
\endproof

\begin{rem}
Lemma \ref{fundlemma} plays a crucial role in the proof of Theorem \ref{thmfunct}. We do not expect any analog of Lemma \ref{fundlemma} for the big vertical $\bqf$-product, as its proof breaks for Gromov--Witten invariants with higher insertions. This appears to be a genuine obstruction to any functoriality property for the big vertical quantum cohomology on the category $\text{\sc Fib}_F$. In the next section, we will recast a functoriality property of ${\rm Q^{Fib}_{big}}$, at the price of restricting to a {\it wide}\footnote{A subcategory of a category $\eu C$ is called \emph{wide} if it has the same objects as $\eu C$, but possibly fewer morphisms.} subcategory of $\text{\sc Fib}_F$.
\qrem\end{rem}

The following result relates the small vertical quantum cohomology of a bundle with the ordinary small quantum cohomology of the fiber.
\begin{cor}\label{corfibra}
Let $(E,B,F)$ be cohomologically decomposable, and denote by $I$ the ideal in ${\rm QH}^{\rm Fib}(E,B,F)$ generated by $\pi^*\bigoplus_{i\geq 1}H^i(B,\La_F)$. We have the isomorphism of  algebras 
\[{\rm QH}(F)\cong {\rm QH}^{\rm Fib}(E,B,F)/I.
\]
\end{cor}
\proof
Any inclusion $\{b\}\hookrightarrow B$ induces the inclusion $\iota_b\colon F\to E$, which is a base change $(F,0,F)\to (E,B,F)$. By Theorem \ref{thmfunct}, we have a morphism of  algebras 
\[\iota^*_b\colon {\rm QH}^{\rm Fib}(E,B,F)\to {\rm QH}^{\rm Fib}(F,0,F)={\rm QH}(F),\]
whose kernel equals $I$ (see Theorem \ref{piinj}). 
\endproof

We now introduce the following:
\begin{itemize}
\item the {wide} subcategory $\text{\sc Fib}_F^{\rm eq.dim}$ of $\text{\sc Fib}_F$, whose morphisms are those Cartesian diagrams \eqref{cartesian} for which $\dim B = \dim B'$;
\item the category $\text{\sc FrobAlg}_F$ of Frobenius super-algebras over $\La_F$, whose morphisms are super-algebra homomorphisms that are conformal maps.
\end{itemize}

\begin{thm}
The small vertical quantum cohomology restricts to a functor \[{\rm QH^{Fib}}\colon \text{\sc Fib}_F^{\rm eq.dim}\to \text{\sc FrobAlg}_F.\]
Given a morphism $(f,f')$, the conformal factor of $f'^*$ equals the topological degree of $f$.
\end{thm}

\proof
The map $f$ has a well-defined topological degree $\deg(f)\in\Z$, the manifolds $B,B'$ being compact, oriented, and equidimensional. Given $\om\in H^\bullet(E,\C)$, we have
\[\int_{E'}f'^*\om=\int_{B'}\pi_*f'^*\om=\int_{B'}f^*\pi_*\om=\deg(f)\int_B\pi_*\om=\deg(f)\int_E\om,
\]where the second equality follows from the compatibility of integration along the fibers with base change. The result follows.
\endproof

\subsection{Functoriality property of big vertical quantum cohomology} Let us introduce a further wide subcategory $\text{\sc Fib}^{\rm coh}_F$ of $\text{\sc Fib}_F$ such that
\[\text{\sc Fib}^{\rm coh}_F\quad \subset \quad \text{\sc Fib}_F^{\rm eq.dim}\quad \subset\quad \text{\sc Fib}_F.
\]

The morphisms of $\text{\sc Fib}^{\rm coh}_F$ are Cartesian diagrams \eqref{cartesian} for which $f\colon B'\to B$ is a {\it cohomological equivalence}, that is it induces isomorphism in cohomology $f^*\colon H^\bullet(B,\C)\to H^\bullet(B',\C)$.

In particular, we both have $\dim B=\dim B'$, and the morphism $f'^*\colon H^\bullet(E,\C)\to H^\bullet(E',\C)$ is an isomorphism, by cohomological triviality.

Introduce then the category $\text{\sc FrobMan}_F$:
\begin{itemize}
\item its objects are Frobenius super-algebras over the ring $\C[\![H^\bullet(E,\C)^*]\!]\cong \C[\![\bm t]\!]$, where $\bm t=(t^i)_{i=0}^N$ are dual coordinates with respect to an arbitrarily fixed basis;
\item its morphisms are super-algebra homomorphisms that are conformal maps.
\end{itemize}

\begin{thm}
The big vertical quantum cohomology defines a contravariant functor 
\[{\rm QH^{Fib}_{big}}\colon \text{\sc Fib}^{\rm coh}_F\to \text{\sc FrobMan}_F.\]
\[\xymatrix@R=0.5em@C=1em{
(E',B',F')\ar[dd]_{(f,f')} & \ar@{|->}[r] & &{\rm QH}^{\rm Fib}_{\rm big}(E',B',F') \\
&\ar@{|->}[r]&&\\
(E,B,F)&\ar@{|->}[r] && {\rm QH}^{\rm Fib}_{\rm big}(E,B,F) \ar[uu]_{f'^*}
}
\]
Given a morphism $(f,f')$, the conformal factor of $f'^*$ equals the topological degree of $f$.
\end{thm}
\proof
Consider classes $e_1, \dots, e_k \in H^\bullet(E, \C)$ that restrict to cohomological bases on each fiber (identifiable with a fixed basis of $H^\bullet(F,\C)$), and a basis $b_1, \dots, b_h$ of $H^\bullet(B, \C)$. Without loss of generality, we may assume $b_1 = 1$ and $b_h$ is its Poincar\'e dual (i.e., a top-degree class), normalized by $\int_B b_h = 1$.

The classes $f'^*e_1,\dots,f'^*e_k \in H^\bullet(E',\C)$ restrict to bases at each fiber, and $f^*b_1,\dots,f^*b_h$ form a basis of $H^\bullet(B',\C)$. Consequently, the classes $\pi^*b_i \cup e_j$ (resp. $f'^*(\pi^*b_i \cup e_j) = \pi'^*f^*b_i \cup f'^*e_j$), with $i = 1,\dots,h$ and $j = 1,\dots,k$, form a basis of $H^\bullet(E,\C)$ (resp. $H^\bullet(E',\C)$), with dual coordinates $(t^{ij})_{i,j}$.

Since $\int_{B'} f^*b_h = \deg(f)\int_B b_h = \deg(f)$, it follows that $\eta_{B'} = \deg(f)\,\eta_B$.

Let $\bqf$ and $\Hat{\bqf}$ denote the products on ${\rm QH^{Fib}_{big}}(E,B,F)$ and ${\rm QH^{Fib}_{big}}(E',B',F)$, respectively. We aim to prove that $f'^*(e_i \bqf e_j) = f'^*e_i \,\Hat{\bqf}\, f'^*e_j$.

By definition, 
\[f'^*(e_i \bqf e_j)=\sum_{n\geq 0}\sum_{\substack{\al_1,\dots,\al_n=1\\ a_1,a_2=1}}^h\sum_{\substack{\gm,\gm_1,\dots,\gm_n=1\\c_1,c_2=1}}^k\sum_{\bt}\frac{\prod_{i=1}^h\prod_{j=1}^kt^{\al_i\gm_j}}{n!}K^{(i,j)}_{\bm\al,\bm\gm,a_1,c_1}\eta_E^{a_1a_2,c_1c_2}f'^*(\pi^*b_{a_2}\cup e_{c_2}),
\]
where $K^{(i,j)}_{\bm\al,\bm\gm,a_1,c_1} = \langle \pi^*b_{\al_1} \cup e_{\gm_1}, \dots, \pi^*b_{\al_n} \cup e_{\gm_n}, e_i, e_j, \pi^*b_{a_1} \cup e_{c_1} \rangle^{\rm Fib}_{0,n+3,\bt}$ and $\eta_E = \eta_B \otimes \eta_F$ with $(\eta_B)_{a_1a_2} = \int_B b_{a_1} b_{a_2}$, $(\eta_F)_{c_1c_2} = \int_F \iota^*e_{c_1} \iota^*e_{c_2}$.
Similarly, 
\[f'^*e_i \,\Hat{\bqf}\, f'^*e_j=\sum_{n\geq 0}\sum_{\substack{\al_1,\dots,\al_n=1\\ a_1,a_2=1}}^h\sum_{\substack{\gm,\gm_1,\dots,\gm_n=1\\c_1,c_2=1}}^k\sum_{\bt}\frac{\prod_{i=1}^h\prod_{j=1}^kt^{\al_i\gm_j}}{n!}\Hat{K}^{(i,j)}_{\bm\al,\bm\gm,a_1,c_1}\eta_{E'}^{a_1a_2,c_1c_2}\pi'^*f^*b_{a_2}\cup f'^*e_{c_2},
\]with $\Hat{K}^{(i,j)}_{\bm\al,\bm\gm,a_1,c_1} = \langle \pi'^*f^*b_{\al_1} \cup f'^*e_{\gm_1}, \dots, \pi'^*f^*b_{\al_n} \cup f'^*e_{\gm_n}, f'^*e_i, f'^*e_j, \pi'^*f^*b_{a_1} \cup f'^*e_{c_1} \rangle^{\rm Fib}_{0,n+3,\bt}$ and $\eta_{E'} = \eta_{B'} \otimes \eta_F$ with $(\eta_{B'})_{a_1a_2} = \int_{B'} f^*(b_{a_1} b_{a_2})$.

We claim that $\Hat{K}^{(i,j)}_{\bm\al,\bm\gm,a_1,c_1} = \deg(f)\, K^{(i,j)}_{\bm\al,\bm\gm,a_1,c_1}$, from which it follows that:
\[
\Hat{K}^{(i,j)}_{\bm\al,\bm\gm,a_1,c_1} \eta_{E'}^{a_1a_2, c_1c_2}
= \deg(f)\, K^{(i,j)}_{\bm\al,\bm\gm,a_1,c_1} \cdot \frac{1}{\deg(f)} \eta_B^{a_1a_2} \cdot \eta_F^{c_1c_2}
= K^{(i,j)}_{\bm\al,\bm\gm,a_1,c_1} \eta_E^{a_1a_2, c_1c_2}.
\]

To prove the claim, we proceed as in the proof of Theorem \ref{thmfunct}, using the commutative diagram \eqref{cartmgn} and equation \eqref{beautiful}. Denote by $\omega \in H^\bullet(B, \C)$ the class
\[
\omega = \prod_{i=1}^n b_{\al_i} \cdot b_{a_1} \cdot \tau_! \left[ \left( \prod_{\ell=1}^n \text{ev}_\ell^* e_{\gm_\ell} \right) \cdot \text{ev}_{n+1}^* e_i \cdot \text{ev}_{n+2}^* e_j \cdot \text{ev}_{n+3}^* e_{c_1} \right].
\]
Then, for a suitable sign $(-1)^\Delta$ coming from super-commutativity,
\[
\Hat{K}^{(i,j)}_{\bm\al,\bm\gm,a_1,c_1} = (-1)^\Delta \int_{B'} f^*\omega = \deg(f) \cdot (-1)^\Delta \int_B \omega = \deg(f)\, K^{(i,j)}_{\bm\al,\bm\gm,a_1,c_1}.
\]
This proves that $f'^*$ is a morphism of algebras. That it is conformal, with conformal factor $\deg(f)$, is clear.
\endproof

\subsection{Induction property, and partial classical limits}
Let $(E_1,E_2,F_1)$ and $(E_2,B,F_2)$ be two locally trivial algebraic bundles such that $(E_1,B,F)$ is also locally trivial:
\[\xymatrix{
E_1\ar[r]^{\pi_1}&E_2\ar[r]^{\pi_2}&B\\
F_1\ar@{_{(}->}[u]^{\iota_1}&F_2\ar@{_{(}->}[u]^{\iota_2}&
}\qquad \pi=\pi_2\circ \pi_1,\qquad F=\pi^{-1}({\rm pt})=\pi_1^{-1}(F_2).
\]
Notice that $F$ is the total space of a locally trivial $F_1$-bundle on $F_2$. We will assume that the bundle $(F,F_2,F_1)$ satisfies the Assumption G, so that
\beq\label{homoF} 
H_2(F,\Z)\cong H_2(F_1,\Z)\oplus H_2(F_2,\Z).
\eeq

\begin{lem}\label{lemcones}
We have ${\rm Eff}_1(F)={\rm Eff_1}(F_1)\oplus{\rm Eff_1}(F_2)$.
\end{lem}
\begin{proof}
This follows from the fact that \( F \to F_2 \) is a locally trivial fibration with fiber \( F_1 \), together with Assumption G, which gives a splitting \eqref{homoF}.
In this identification, any effective curve class in \( F \) is represented by a curve whose projection to \( F_2 \) is either a point or an effective curve, and whose fiberwise component lies in \( F_1 \). Since the bundle is locally trivial, the image of any such curve corresponds to a sum of effective classes in \( F_1 \) and \( F_2 \). Therefore, the monoid of effective classes splits as claimed.
\end{proof}
Let us assume that all the Assumptions G,F,G',F' are satisfied by the bundles $(E_1,E_2,F_1),$ $(E_2,B,F_2),(E_1,B,F)$. Hence, we have well-defined algebras:
\begin{align*}
{\rm QH^{Fib}}(E_1,B,F)&=\left(H^\bullet(E_1,\La_F),\,\sqqf{1}\,,\eta_{E_1}\right),\\
{\rm QH^{Fib}}(E_1,E_2,F_1)&=\left(H^\bullet(E_1,\La_{F_1}),\,\sqqf{2}\,,\eta_{E_1}\right),\\
{\rm QH^{Fib}}(E_2,B,F_2)&=\left(H^\bullet(E_2,\La_{F_2}),\,\sqqf{3}\,,\eta_{E_2}\right).
\end{align*}

By Lemma \ref{lemcones}, we have the ring isomorphism
\[\La_F\cong \La_{F_1}\otimes \La_{F2}=\C[{\bf q}_1^{\bt},\,{\bf q}_2^{\bt'}\colon \bt\in {\rm Eff}_1(F_1),\,\bt'\in{\rm Eff}_1(F_2)].\]
Consider the canonical projection
\[\psi\colon \La_F\longrightarrow\frac{\La_F}{\langle{\bf q}_2^{\bt'}\colon \bt'\in{\rm Eff}_1(F_2)\setminus\{0\}\rangle}\cong \La_{F_1},
\]
inducing a morphism of graded abelian groups\footnote{In general, a ring morphism $\phi\colon R_1 \to R_2$ does not induce a ring morphism $\phi_*\colon H^\bullet(X, R_1) \to H^\bullet(X, R_2)$, but only a morphism of graded abelian groups. For example, let $X = \mathbb{RP}^2$, $R_1 = \mathbb{Z}$ and $R_2 = \mathbb{Z}/2\mathbb{Z}$. Then:
\[
H^\bullet(X, \mathbb{Z}) \cong \mathbb{Z}[0] \oplus (\mathbb{Z}/2\mathbb{Z})[-2], \qquad H^\bullet(X, \mathbb{Z}/2\mathbb{Z}) \cong (\mathbb{Z}/2\mathbb{Z})[x]/(x^3),
\]
where $\deg(x) = 1$. In degree 2, the induced morphism is the identity. Let $\alpha \in H^2(X, \mathbb{Z})$ be the nontrivial torsion class, so that $\phi_*(\alpha) = x$. Then:
\[
\phi_*(\alpha \cup \alpha) = \phi_*(0) = 0 \quad \text{but} \quad \phi_*(\alpha) \cup \phi_*(\alpha) = x \cup x = x^2 \ne 0.
\]
Hence, $\phi_*$ is not compatible with the cup product and thus not a morphism of rings. Similar examples arise in the setting of complex algebraic geometry, e.g., Enriques surfaces, where torsion classes in integral cohomology yield the same failure of multiplicativity. One can prove (as a consequence of the Universal Coefficient Theorem) that if $\phi\colon R_1 \to R_2$ is a {\it flat} morphism of rings and the singular cohomology $H^\bullet(X,R_1)$ is {\it flat} (e.g., a free module) over $R_1$, then the natural change‐of‐coefficients map
$\phi_*$
is in fact a morphism of {graded rings}, i.e.\ it respects the cup product. In general, the natural surjection \( R_1 \to R_1/I \) is not flat, unless \( I = 0 \). As a consequence, the induced map in cohomology
$H^\bullet(X, R_1) \to H^\bullet(X, R_1/I)$
does not preserve the ring structure in general. This explains why change of coefficients via a quotient may fail to respect cup products.
}
\[\psi_*\colon H^\bullet(E_1,\La_F)\to H^\bullet(E_1,\La_{F_1}).
\]It turns out that $\psi_*$ preserves not only the $\cup$-product but even the quantum products. This was already described by A.\,Astashkevich and V.\,Sadov \cite{AS95}.

\begin{thm}
The moprhism $\psi_*$ defines a morphism of rings
\[\psi\colon {\rm QH^{Fib}}(E_1,B,F)\to {\rm QH^{Fib}}(E_1,E_2,F_1).
\]
\end{thm}
\proof
Given a $\C$-basis $(T_i)_{i=0}^N$ of $H^\bullet(E,\C)$, we need to prove that 
\[\psi_*(T_i\sqqf{1} T_j)=T_i\sqqf{2}T_j.\]
We have
\[\psi_*(T_i\sqqf{1} T_j)=\sum_{d,\ell=0}^N\sum_{\bt\in{\rm Eff}_1(F_1)} {\bf q}^\bt\langle T_i, T_j, T_d \rangle^{\rm Fib}_{0,3,\bt\oplus 0} \, \eta^{d\ell} T_\ell,
\]
and
\[T_i\sqqf{2}T_j=\sum_{d,\ell=0}^N\sum_{\bt\in{\rm Eff}_1(F_1)} {\bf q}^\bt\langle T_i, T_j, T_d \rangle^{\rm Fib}_{0,3,\bt} \, \eta^{d\ell},
\]where:
\begin{itemize} 
\item $\langle T_i, T_j, T_d \rangle^{\rm Fib}_{0,3,\bt\oplus 0}$ is an integral over the moduli space $\overline{\mc M}_{0,3}^{\rm Fib}(E_1,B,F,\bt\oplus0)$,\\
\item and $\langle T_i, T_j, T_d \rangle^{\rm Fib}_{0,3,\bt}$ is an integral over $\overline{\mc M}_{0,3}^{\rm Fib}(E_1,E_2,F_1,\bt)$.
\end{itemize}
But $\overline{\mc M}_{g,n}^{\rm Fib}(E_1,B,F,\bt\oplus0)=\overline{\mc M}_{g,n}^{\rm Fib}(E_1,E_2,F_1,\bt)$, as it follows from their geometrical definitions. The claim follows.
\endproof
\begin{cor}\label{pcl1}
We have an isomorphism of Frobenius super-algebras
\beq\label{isopcl1}
\frac{{\rm QH^{Fib}}(E_1,B,F)}{\langle{\bf q}_2^{\bt'}\colon \bt'\in{\rm Eff}_1(F_2)\setminus\{0\}\rangle}\cong {\rm QH^{Fib}}(E_1,E_2,F_1).
\eeq
\end{cor}
\proof
The ideal $\langle{\bf q}_2^{\bt'}\colon \bt'\in{\rm Eff}_1(F_2)\setminus\{0\}\rangle$ is the kernel of $\psi_*$. 
\endproof

\begin{cor}\label{pcl2}
Let $(E,B,F)$ be a locally trivial bundle for which the small quantum cohomology is well-defined. The (ordinary) small quantum cohomology of $E$ and the small vertical quantum cohomology of $(E,B,F)$ are related by the isomorphism
\beq\label{isopcl2}
\frac{{\rm QH}(E)}{\langle{\bf q}^{\bt}\colon \bt\in{\rm Eff}_1(B)\setminus\{0\}\rangle}\cong {\rm QH^{Fib}}(E,B,F).
\eeq
\end{cor}
\begin{proof}
This is Corollary \ref{pcl1} specialized to the case $B=\mathrm{Spec}(\C)$.
\end{proof}

Let us reinterpret this result in terms of families of Frobenius super-algebras parametrized by points of a torus, as described in Section \ref{secsqc}. 

Consider a locally trivial bundle $(E,B,F)$ for which the small vertical quantum cohomology is defined. As in Remark \ref{remperiodicity}, choose
\begin{itemize}
\item an integral basis $\tilde T_1,\dots, \tilde T_k$ of $H^{1,1}(F,\C)$, 
\item an integral basis $\tilde T_{k+1},\dots, \tilde T_{k+m}$ of $H^{1,1}(B,\C)$,
\end{itemize}
and construct an integral basis $T_1,\dots, T_{k+m}$ of $H^{1,1}(E,\C)$ such that
\[
\iota^*T_i = \tilde T_i,\quad i=1,\dots,k, \qquad \pi^*(\tilde T_j) = T_j,\quad j = k+1,\dots,k+m.
\]
For each point $\bm q=(q_1,\dots,q_{k+m})\in(\C^*)^{k+m}$, we have a well-defined small quantum cohomology ring ${\rm QH}_{\bm q}(E)$, via \eqref{smallqprodq}. 

The isomorphism \eqref{isopcl2} is equivalent to the statement that, in a suitable limit $\bm q \to \bar{\bm q}$ (with $\bar{\bm q}$ lying in a partial compactification of $(\C^*)^{k+m}$), the Frobenius super-algebra ${\rm QH}_{\bm q}(E)$ specializes to the vertical quantum cohomology ${\rm QH}^{\rm Fib}_{\bar{\bm q}}(E,B,F)$.

For example, assume -- just for simplicity -- that the basis $\tilde T_{k+1},\dots, \tilde T_{k+m}$ lies in the NEF cone of $B$ (as in Remark \ref{remperiodicity}). We can then consider the partial compactification $(\C^*)^{k+m} \subset \C^{k+m}$, and the isomorphism \eqref{isopcl2} implies that in the \emph{partially classical} limit
\[
(q_1,\dots,q_k, q_{k+1},\dots, q_{k+m}) \longrightarrow (q_1,\dots,q_k, 0,\dots, 0),
\]
the small quantum cohomology ${\rm QH}_{\bm q}(E)$ reduces to ${\rm QH}^{\rm Fib}_{(q_1,\dots,q_k)}(E,B,F)$. In a further limit -- for instance $(q_1,\dots,q_k)\to 0$, if also $\tilde T_1,\dots,\tilde T_k$ lie in the NEF cone -- we recover the classical cohomology algebra $H^\bullet(E,\C)$.

A similar description can be given for the isomorphism \eqref{isopcl1}.

\section{Flag bundles, and their vertical quantum spectra}\label{sec3}
\subsection{Partial flag varieties}
Let $N,n\in\Z_{>0}$ with $N\leq n$, and let $\bm\lambda = (\lambda_1,\dots,\lambda_N)\in\Z_{>0}^N$ be a composition of $n$, i.e., $|\bm\lambda| := \sum_i\lambda_i = n$. The associated partial flag variety $F_{\bm\lambda}$ parametrizes flags
\beq\label{flag1}
0 = V_0 \subset V_1 \subset \dots \subset V_N = \C^n,\qquad \dim_\C(V_i/V_{i-1}) = \lambda_i.
\eeq
The integer $N$ is the {\it length} of both the composition and the flag. Grassmannians are the special case $N=2$, with $G(k,n) := F_{(k,n-k)}$. The complex dimension of $F_{\bm\lambda}$ is $\dim_\C F_{\bm\lambda} = \sum_{i<j} \lambda_i \lambda_j = \frac{1}{2}\left(n^2 - \sum_{i=1}^N \lambda_i^2\right).$

\begin{rem}\label{numbercomposition}
The number of partial flag varieties of length $N$ chains in $\C^n$ equals $\binom{n-1}{N-1}$, the number of compositions of $n$ into $N$ positive integers.
\qrem\end{rem}

\begin{rem}
The variety \( F_{\bm\lambda} \) is smooth and projective: it embeds as a closed subvariety of the product \( \prod_{i=1}^N G(\lambda^{(i)}, n) \), where \( \lambda^{(i)} := \sum_{j=1}^i \lambda_j \). It is also a rational homogeneous space, isomorphic to the quotient \( GL_n / P \), where \( P \subset GL_n \) is the parabolic subgroup stabilizing a flag of type \( \bm\lambda \).
It follows that all its cohomology classes are algebraic: the cohomology is purely even and of type \( (p,p) \), so \( h^{p,q} = 0 \) for \( p \ne q \).
\qrem\end{rem}

Let $Q_i \to F_{\bm\lambda}$, $i=1,\dots,N$, be the canonical quotient bundles of rank $\lambda_i$, with fiber $V_i/V_{i-1}$ over the flag \eqref{flag1}. For each bundle $V$, let $c(V) = \sum_{j\geq 0} c_j(V) \, {\sf t}^j$ be its total Chern class, with formal parameter $\sf{t}$. The relation
\[
\bigoplus_{i=1}^N Q_i \cong \underline{\C^n} \quad\Rightarrow\quad \prod_{i=1}^N c(Q_i) = 1
\]
generates the ideal of relations in $H^\bullet(F_{\bm\lambda},\C)$.

Let $\bm\gamma_i = (\gamma_{i,1},\dots,\gamma_{i,\lambda_i})$ be the Chern roots of $Q_i$, and $\C[\bm\gamma]^{S_{\bm\lambda}}$ -- where $S_{\bm\la}=S_{\la_1}\times\dots\times S_{\la_N}$ -- the ring of block-symmetric polynomials in $\bm\gamma = (\bm\gamma_1,\dots,\bm\gamma_N)$.  Then
\beq\label{prescoh}
H^\bullet(F_{\bm\lambda},\C) \cong \frac{\C[\bm\gamma]^{S_{\bm\lambda}}}{I}, \quad \text{where } I = \left\langle \prod_{i=1}^N \prod_{j=1}^{\lambda_i}(1 + {\sf t}\, \gamma_{i,j}) = 1 \right\rangle.
\eeq
Equivalently, $I$ is generated by
\[
\left\{ \sum_{i_1+\dots+i_N=h} \prod_{j=1}^N e_{i_j}(\bm\gamma_j) \;\middle|\; h = 1,\dots,n \right\},
\]
with $e_j$ denoting the $j$-th elementary symmetric polynomial.

By a classical result of C.\,Ehresmann \cite{Ehr34}, the integral cohomology ring of the partial flag variety \( F_{\bm\lambda} \) is freely generated by the \emph{Schubert classes}. These are the (Poincar\'e duals of the) fundamental classes of certain subvarieties \( \Omega_\sigma \subset F_{\bm\lambda} \), known as \emph{Schubert varieties}, which are indexed by the minimal coset representatives in the quotient
$S_n / \left(S_{\lambda_1} \times \dots \times S_{\lambda_N}\right).$
In particular, the total Betti number of \( F_{\bm\lambda} \) coincides with the number of such coset representatives:
\[
\dim_\C H^\bullet(F_{\bm \lambda}, \C) = \binom{n}{\lambda_1, \dots, \lambda_N} = \frac{n!}{\lambda_1! \cdots \lambda_N!}.
\]
In the polynomial algebra \eqref{prescoh}, each Schubert class can be represented algebraically by a \emph{Schubert polynomial} \( \mathfrak{S}_\sigma(\bm\gamma) \); see Appendix~\ref{appSchub} for further details.

\begin{prop}\label{chernflagvariety}
There is a canonical isomorphism of vector bundles:
\beq\label{dec1}
T F_{\bm\lambda} \cong \bigoplus_{i<j} \Hom(Q_i, Q_j) = \bigoplus_{i<j} Q_i^* \otimes Q_j,
\eeq
and hence,
\[
c_1(F_{\bm\lambda}) = \sum_{i<j} c_1(Q_i^* \otimes Q_j) = \sum_{i=1}^N \left( \sum_{j<i} \lambda_j - \sum_{j>i} \lambda_j \right) c_1(Q_i). 
\]
\end{prop}
\proof
The partial flag variety $F_{\bm\lambda}$ is a homogeneous space $GL_n/P$, where $P$ is a parabolic subgroup. Its tangent bundle identifies with $\End(\C^n)/\mathfrak{p}$, the space of endomorphisms modulo those preserving the flag. This yields the decomposition \eqref{dec1}. Using $c_1(Q_i^* \otimes Q_j) = \lambda_i c_1(Q_j) - \lambda_j c_1(Q_i)$ and summing over $i<j$ gives the formula for $c_1(F_{\bm\lambda})$.
\endproof

\begin{rem}
The classes \( c_1(Q_1), \dots, c_1(Q_{N-1}) \) are nef and generate the nef cone of \( F_{\bm\lambda} \). By the relation \( \sum_i c_1(Q_i) = 0 \), the remaining class \( c_1(Q_N) \) is anti-nef. In particular, $F_{\bm\lambda}$ is a Fano variety. 
\qrem\end{rem}

\begin{rem}
On the partial flag variety $F_{\bm\lambda}$, the Schubert divisors are given by 
$D_{i}=c_1(\det(S_i)^{*})=-\sum_{j=1}^i c_1(Q_j)$, with $i=1,\dots, N-1$,
where $S_i$ and $Q_j$ denote the tautological subbundles and quotients. 
Since $-K_{F_{\bm\lambda}}=\sum_{i=1}^{N-1}(\lambda_i+\lambda_{i+1})D_{i}$ in $\mathrm{Pic}(F_{\bm\lambda})$, 
it follows that $F_{\bm\la}$ has Fano index 
$\operatorname{ind}(F_{\bm\lambda})=\gcd(\lambda_1+\lambda_2,\ \lambda_2+\lambda_3,\ \dots,\ \lambda_{N-1}+\lambda_N).$\qrem
\end{rem}

The small quantum cohomology ring \( \mathrm{QH}_{\bm{q}}(F_{\bm\lambda}) \) admits a presentation as a deformation of the classical cohomology ring, depending on quantum parameters \( \bm{q} = (q_1, \dots, q_{N-1}) \in (\mathbb{C}^*)^{N-1} \).
The quantum relations are encoded by a companion-type matrix \( A \) whose determinant governs the presentation of the quantum cohomology ring. Its entries depend linearly on the Chern roots \( \gamma_{i,j} \), and include quantum corrections via the variables \( q_i \).

Define the $n\times n$ matrix $A$ as follows:

\begin{itemize}
  \item \( A_{r, r-1} = -1 \) for \( r = 2, \dots, n \) (i.e., the subdiagonal is all \(-1\));
  \item For \( i = 1, \dots, N \):
  \begin{itemize}
    \item Let \( s_i = \la_1 + \cdots + \la_{i-1} \) (with \( s_1 = 0 \));
    \item For \( j = 1, \dots, \la_i \): \( A_{s_i + 1, s_i + j} = \gm_{i,j}\); %x^{(i)}_j \);
    \item If \( i < N \), then set:
    \[
      A_{s_i + 1,\, s_i + \la_i + \la_{i+1}} = -(-1)^{\la_{i+1}} q_i
    \]
  \end{itemize}
  \item All other entries of \( A \) are zero.
\end{itemize}

\begin{example}
For $\bm\la=(1,1,1)$, we have
\[
A = 
\begin{pmatrix}
\gm_{1,1} & q_1   & 0     \\
-1     & \gm_{2,1} & q_2  \\
0      & -1     & \gm_{3,1}
\end{pmatrix}.
\]For $\bm\la=(2,2)$, we have
\[
A = 
\begin{pmatrix}
\gm_{1,1} & \gm_{1,2} & -q_1    & 0     \\
-1     & 0      & 0      & 0     \\
0      & -1     & \gm_{2,1} & \gm_{2,2} \\
0      & 0      & -1     & 0
\end{pmatrix}.
\]For $\bm\la=(2,3,1)$, we have
\[
A = 
\begin{pmatrix}
\gm_{1,1} & \gm_{1,2} & 0      & 0      & q_1    & 0     \\
-1     & 0      & 0      & 0      & 0      & 0     \\
0      & -1     & \gm_{2,1} & \gm_{2,2} & \gm_{2,3} & q_2   \\
0      & 0      & -1     & 0      & 0      & 0     \\
0      & 0      & 0      & -1     & 0      & 0     \\
0      & 0      & 0      & 0      & -1     & \gm_{3,1}
\end{pmatrix}.\tag*{\qetr}
\]
\end{example}

\begin{thm}\label{thmpresqcohfl}\cite{AS95,Kim95,Kim96}
We have
\[{\rm QH}^\bullet_{\bm q}(F_{\bm\lambda})\cong \C[\bm\gm]^{S_{\bm\la}}[\bm q]/J,
\]where $J$ is the ideal generated by the coefficients of the polynomial $\det({\sf t}\cdot {\rm Id}+A)-{\sf t}^n$.\qed
\end{thm}
Such a presentation was independently obtained by A.\,Astashkevich and V.\,Sadov \cite{AS95}, and by B.\,Kim \cite{Kim95}, with a complete proof provided in~\cite{Kim96}; see also~\cite{Kim99} for further developments.
Earlier results covering the extreme cases of Grassmannians and complete flag varieties were obtained in \cite{ST97}, \cite{Wit95}, and \cite{CF95}, \cite{GK95}, respectively. See also \cite{CF99} for a unified description of the small quantum cohomology rings of all projective homogeneous spaces \( SL_n(\mathbb{C}) / P \), where \( P \) is a parabolic subgroup.

When the parameters \( \bm{q} \) are set to zero, the presentation of Theorem \ref{thmpresqcohfl} reduces to the classical one by A.\,Borel \cite{Bor53} for the cohomology ring \( H^\bullet(F_{\bm\lambda}, \mathbb{C}) \).

\subsection{Flag bundles}
Let \( X \) be a smooth projective variety, and let \( E \to X \) be a holomorphic vector bundle of rank \( n \).  
Fix a composition \( \bm \lambda = (\lambda_1, \dots, \lambda_N) \in \Z_{>0}^N \) with \( |\bm\lambda| = n \).  
We define the \emph{flag bundle} \( \eu{F}_{\bm\lambda}(E) \to X \) to be the fiber bundle over \( X \) whose fiber over a point \( p \in X \) is the partial flag variety \( F_{\bm\lambda}(E_p) \) parametrizing filtrations
\beq\label{flag2}
0 = V_0 \subset V_1 \subset \dots \subset V_N = E_p, \qquad \dim_\C(V_i / V_{i-1}) = \lambda_i.
\eeq
This generalizes the Grassmann bundle \( \eu{G}_k(E) \to X \), corresponding to the case \( N = 2 \) and \( \bm\lambda = (k, n-k) \).

The total space \( \eu{F}_{\bm\lambda}(E) \) is smooth and projective. Over \( \eu{F}_{\bm\lambda}(E) \), we have canonical rank-\( \lambda_i \) quotient bundles
$\eu{Q}_i \to \eu{F}_{\bm\lambda}(E),$ $ i = 1, \dots, N,$
whose fiber over a flag \eqref{flag2} is \( \eu{Q}_{i, p} = V_i / V_{i-1} \subset E_p \).

Let \( \pi\colon \eu{F}_{\bm\lambda}(E) \to X \) be the natural projection.
\begin{prop}\label{c1flbl}
We have
\[c_1(\eu F_{\bm\la}(E))=\pi^*c_1(X)+\sum_{i<j}c_1(\eu Q_i^*\otimes \eu Q_j)=\pi^*c_1(X)+\sum_{i=1}^N\left(\sum_{j<i}\lambda_j-\sum_{j>i}\lambda_j\right)c_1( \eu Q_i).
\]
\end{prop}
\proof
Let \( T_\pi \) denote the vertical tangent bundle (tangent to the fibers) of the projection \( \pi\colon \eu{F}_{\bm\lambda}(E) \to X \).  
We have a short exact sequence of vector bundles
$0 \to T_\pi \to T\eu{F}_{\bm\lambda}(E) \to \pi^* T X \to 0$,
and thus
$c_1\big(T \eu{F}_{\bm\lambda}(E)\big) = \pi^* c_1(T X) + c_1(T_\pi).$
The vertical bundle \( T_\pi \) restricts fiberwise to the tangent bundle of the flag variety \( F_{\bm\lambda} \), and Proposition~\ref{chernflagvariety} gives
$c_1(T_\pi) = \sum_{i<j} c_1(\eu{Q}_i^* \otimes \eu{Q}_j).$
\endproof

Let $\tilde{\bm\gamma}_i = (\gamma_{i,1},\dots,\gamma_{i,\lambda_i})$ be the Chern roots of $\eu Q_i$, and $\C[\tilde{\bm\gamma}]^{S_{\bm\lambda}}$  the ring of block-symmetric polynomials in $\tilde{\bm\gamma} = (\tilde{\bm\gamma}_1,\dots,\tilde{\bm\gamma}_N)$. 

\begin{thm}
The bundle $(\eu F_{\bm\la}(E),X,F_{\bm\la})$ is cohomologically decomposable. We have the ring isomorphism
\beq\label{cohisoflbl}
H^\bullet(\eu F_{\bm\la}(E),\C)=\frac{H^\bullet(X,\C)[\tilde{\bm\gm}]^{S_{\bm\la}}}{I},\quad \text{where } I = \left\langle \prod_{i=1}^N \prod_{j=1}^{\lambda_i}(1 + {\sf t}\, \tilde{\gamma}_{i,j}) = c(E) \right\rangle.
\eeq
\end{thm}
\proof The Schubert polynomials $\frak S_\si(\bm\gm)$, for $\si\in S_n$, form a basis of $H^\bullet(F_{\bm\la},\C)$. Their pullbacks $\frak S_\si(\tilde{\bm\gm})$ define classes in $H^\bullet(\eu F_{\bm\la}(E),\C)$ that restrict to a basis in the cohomology of each fiber $\pi^{-1}(b)$, $b\in X$. Hence, by the Leray--Hirsch theorem, the bundle is cohomologically decomposable.

Denote by $R$ the right-hand side of \eqref{cohisoflbl}. Define a ring morphism
\[
\phi\colon R \to H^\bullet(\eu F_{\bm\la}(E),\C),\quad \alpha \mapsto \pi^*\alpha,\quad e_j(\tilde{\bm\gm}_i) \mapsto c_j(\eu Q_i),
\]
for $\alpha\in H^\bullet(X,\C)$, $i=1,\dots,N$, and $j=1,\dots,\la_i$.
This is well-defined because of the universal relation 
$\bigoplus_{i=1}^N \eu Q_i = \pi^*E $, which implies $ \prod_{i=1}^N c(\eu Q_i) = c(\pi^* E).$
Note that this is the \emph{only} relation in the cohomology ring of the fiber, so $I$ captures all relations globally.
Since the Schubert basis pulls back fiberwise to a basis of $H^\bullet(\eu F_{\bm\la}(E),\C)$ over \( H^\bullet(X,\C) \), and since $\phi$ maps this basis to a basis, we conclude that $\phi$ is both surjective and injective.
\endproof

Since $(\eu F_{\bm\la}(E),X,F_{\bm\la})$ is cohomologically decomposable, and with Fano fiber, Assumptions G,F,G',F' hold (see Remarks \ref{remF}, \ref{remG}, \ref{remG'}, \ref{remF'}). 
Hence, we have a family of algebras ${\rm QH}_{\bm q}^{\rm Fib}(\eu F_{\bm\la}(E),X,F_{\bm\la})$ parametrized by points $\bm q\in(\C^*)^{N-1}$.

Although no general presentation is known for the small quantum cohomology of \( \eu F_{\bm\lambda}(E) \), the vertical part admits an explicit description. Let \( A \) be the matrix defined in the previous section.

\begin{thm}{\cite{AS95}}
The small vertical quantum ring admits the presentation
\[
\mathrm{QH}^{\mathrm{Fib}}_{\bm{q}}(\eu F_{\bm\lambda}(E), F_{\bm\lambda}, X)
= \frac{H^\bullet(X, \C)[\tilde{\bm\gamma}]^{S_{\bm\lambda}}[\bm{q}]}{J},
\]
where \( J \) is the ideal generated by the coefficients of
\[
\det(\mathsf{t} \cdot \mathrm{Id} + A) - \sum_{j=0}^n \mathsf{t}^{n-j} c_j(E).
\tag*{\qed}
\]
\end{thm}

\subsection{Vertical quantum spectrum of flag bundles} We compute the vertical quantum characteristic polynomial of the bundle \((\eu F_{\bm\lambda}(E), X, F_{\bm\lambda})\), see Definition \ref{vqcp}.

Fix \(\bm{q} \in (\mathbb{C}^*)^{N-1}\), and denote by
\[
f_{(\eu F_{\bm\lambda}(E), X, F_{\bm\lambda})}(\zeta; \bm{q})
\quad\text{and}\quad
f_{F_{\bm\lambda}}(\zeta; \bm{q})
\]
the vertical quantum characteristic polynomial of \((\eu F_{\bm\lambda}(E), X, F_{\bm\lambda})\) and the (ordinary) quantum characteristic polynomial of the fiber \(F_{\bm\lambda}\), respectively, both evaluated at \(\bm{q}\).

\begin{thm}\label{thmvqcp}
The vertical quantum characteristic polynomial of \((\eu F_{\bm\lambda}(E), X, F_{\bm\lambda})\) satisfies
\[
f_{(\eu F_{\bm\lambda}(E), X, F_{\bm\lambda})}(\zeta; \bm{q}) = \left[ f_{F_{\bm\lambda}}(\zeta; \bm{q}) \right]^{\dim H^\bullet(X, \mathbb{C})}.
\]
\end{thm}

\begin{proof}
Let \( e_1, \dots, e_K \in H^\bullet(\eu F_{\bm\lambda}(E), \mathbb{C}) \), with \( K = \frac{n!}{\lambda_1! \cdots \lambda_N!} \), be classes restricting fiberwise to a basis of \( H^\bullet(F_{\bm\lambda}, \mathbb{C}) \). Let \( b_1 = 1, b_2, \dots, b_h \) be a basis of \( H^\bullet(X, \mathbb{C}) \), with \( h = \dim H^\bullet(X, \mathbb{C}) \).

Consider the basis:
\begin{equation} \label{magicbasis}
e_1, \dots, e_K,\quad b_2 \cup e_1, \dots, b_2 \cup e_K,\quad \dots,\quad b_h \cup e_1, \dots, b_h \cup e_K
\end{equation}
of \( H^\bullet(\eu F_{\bm\lambda}(E), \mathbb{C}) \). Without loss of generality, we may assume
\[e_1=\sum_{i<j}c_1(\eu Q_i^*\otimes \eu Q_j).\]
By Proposition~\ref{c1flbl}, the operator \( c_1(\eu F_{\bm\lambda}(E)) \sqqf{\bm q} \) decomposes as
\[
c_1(\eu F_{\bm\lambda}(E)) \sqqf{\bm q} = A_1 + A_2, \quad \text{where } A_1 = \pi^* c_1(X) \sqqf{\bm q},\quad A_2 = e_1 \sqqf{\bm q}.
\]
By Theorem~\ref{thm1}(1), \( A_1 \) acts as cup product with \( \pi^* c_1(X) \), hence is nilpotent. Moreover, the product \( \sqqf{\bm q} \) is commutative, so \( [A_1, A_2] = 0 \). It follows that the characteristic polynomial of the full operator equals that of \( A_2 \):
\[
f_{(\eu F_{\bm\lambda}(E), X, F_{\bm\lambda})}(\zeta; \bm{q}) = \det(\zeta \cdot \mathrm{Id} - A_2).
\]
By Lemma~\ref{fundlemma}, \( A_2 \) preserves the subspace \( \mathrm{Span}_\mathbb{C}\{e_1, \dots, e_K\} \), and by Theorem~\ref{thm1}(1), we have
\[
A_2(\pi^* b_i \cup e_j) = \pi^* b_i \cup A_2 e_j.
\]
Therefore, the matrix of \( A_2 \) in the basis \eqref{magicbasis} has block-diagonal form:
\[
\underbrace{M \oplus M \oplus \cdots \oplus M}_{h\text{ times}},
\]
where \( M \) is the matrix of the operator \( e_1 \sqqf{\bm q} \) on \( \mathrm{Span}_\mathbb{C}\{e_1, \dots, e_K\} \).

The restriction \( \iota^* \colon H^\bullet(\eu F_{\bm\lambda}(E), \mathbb{C}) \to H^\bullet(F_{\bm\lambda}, \mathbb{C}) \) induces an isomorphism on the subspace $\mathrm{Span}_\mathbb{C}\{e_1, \dots, e_K\} $, and it intertwines \( e_1 \sqqf{\bm q} \) with \( c_1(F_{\bm\lambda}) \sqqf{\bm q} \), by Corollary~\ref{corfibra} and Proposition~\ref{chernflagvariety}. The claim follows.
\end{proof}

\medskip

As a consequence, every eigenvalue of the vertical quantum operator at any \( \bm{q} \in (\mathbb{C}^*)^{N-1} \) has algebraic multiplicity at least \( \dim_\mathbb{C} H^\bullet(X, \mathbb{C}) \). If at least one eigenvalue has strictly greater algebraic multiplicity, we say that the vertical quantum spectrum is \emph{exceeding}.

\begin{cor}\label{corexcee1}
The spectrum of \( (\eu F_{\bm\lambda}(E), X, F_{\bm\lambda}) \) is exceeding at \( \bm{q} \in (\mathbb{C}^*)^{N-1} \) if and only if the fiber \( F_{\bm\lambda} \) does not have simple quantum spectrum at \( \bm{q} \). \qed
\end{cor}

\begin{cor}\label{corvqsgrb}
Let \( p_1(n) \) denote the smallest prime divisor of \( n \in \mathbb{N}_{>1} \). The Grassmann bundle \( (\eu G_k(E), X, G(k, \operatorname{rk} E)) \) has exceeding vertical quantum spectrum at any \( q \in \mathbb{C}^* \) if and only if
\[
p_1(\operatorname{rk} E) \leq k \leq \operatorname{rk} E - p_1(\operatorname{rk} E).
\]
In particular, whether the spectrum is exceeding does not depend on the value of \( q \).
\end{cor}
\proof
The claim follows from Corollary \ref{corexcee1} and \cite[Thm.\,4.4]{Cot22}.
\endproof

\subsection{Limits of quantum spectra of flag varieties, and prime factorization}\label{secpfv} Let $\bm\la=(\la_1,\dots,\la_N)\in \Z^N_{>0}$ be a composition of $n$, with associated partial flag variety $F_{\bm\la}$. 

Fix the nef integral basis $c_1(Q_1),\dots, c_1(Q_{N-1})$ of $H^{1,1}(F_{\bm\la},\C)$: for each point $\bm q\in(\C^*)^{N-1}$, equation \eqref{smallqprodq} defines the (ordinary) small quantum cohomology algebra ${\rm QH}_{\bm q}(F_{\bm\la})$, together with the (ordinary) quantum characteristic polynomial $f_{F_{\bm\la}}(\zeta;\bm q)$ and associated spectrum.

The nature of the spectrum is, in general, highly depending on the point $\bm q$, as the following example shows.

\begin{example}\label{exF111}
Let $N=n=3$, $\bm \la=(1,1,1)$. In a suitable basis (see Appendix \ref{appdynop}) of $H^\bullet(F_{\bm\la},\C)$, the operator $c_1(F_{\bm\la})\sq_{\bm q}$ is represented by the matrix
\[\left(
\begin{array}{cccccc}
 0 & -2 & -2 & 0 & 0 & 0 \\
 -2 {q_2} & 0 & 0 & -4 & -2 & 0 \\
 -2 {q_1} & 0 & 0 & -2 & -4 & 0 \\
 0 & 0 & -2 {q_2} & 0 & 0 & -2 \\
 0 & -2 {q_1} & 0 & 0 & 0 & -2 \\
 -4 {q_1} {q_2} & 0 & 0 & -2 {q_1} & -2 {q_2} & 0 \\
\end{array}
\right).\]
The quantum characteristic polynomial equals 
\[f_{F_{\bm\la}}(\zeta;\bm q)=\zeta^6+\zeta^4(-12 q_1-12 {q_2})+\zeta^2(48 q_1^2 - 336 q_1 q_2 + 48 q_2^2)-64 q_1^3 - 192 q_1^2 q_2 - 192 q_1 q_2^2 - 64 q_2^3,
\]whose discriminant is
\[2^{36}3^{18}q_1^4  q_2^4(q_1 - q_2)^4    (q_1 + q_2)^3.
\]In the complement $(\C^*)^{2}\setminus\{q_1= \pm q_2\}$, the quantum spectrum is simple. At points $\bm q=(q,q)$, with $q\neq 0$, the spectrum consists of 4 eigenvalues, two of which have algebraic multiplicity 2. At points $\bm q=(q,-q)$, with $q\neq 0$, the spectrum consists of 5 eigenvalues, one of which with algebraic multiplicity 2.
\qetr
\end{example}

For any fixed $i=1,\dots, N-1$, we now consider the partially classical limit $q_j\to 0$, for $j\neq i$, of the algebra ${QH}_{\bm q}(F_{\bm\la})$, quantum characteristic polynomial, and associated spectrum.

In what follows, the limit $\lim_{q_j\to 0, j\neq i}f_{F_{\bm\la}}(\zeta;\bm q)$ will be called {\it $i$-th semiclassical characteristic polynomial of $F_{\bm\la}$}. Its multiset of zeroes will be called {\it $i$-th semiclassical spectrum of $F_{\bm\la}$}.

\begin{thm}\label{THMSEMICLSPEC}
For any $i=1,\dots, N-1$, any root $\zeta_o$ of the $i$-th semiclassical characteristic polynomial $\lim_{q_j\to 0, j\neq i}f_{F_{\bm\la}}(\zeta;\bm q)$ satisfy the inequality
\beq\label{minalgmult}
{\rm alg.\,mult.}(\zeta_o)\geq \frac{n!}{\la_1!\dots\la_{i-1}!(\la_i+\la_{i+1})!\la_{i+2}!\dots\la_N!}.
\eeq
The $i$-th semiclassical spectrum is of exceeding type (that is in \eqref{minalgmult} the strict inequality holds for at least one zero $\zeta_o$) if and only if
\beq\label{primerelation}
p_1(\la_i+\la_{i+1})\leq\,\, \la_i,\,\la_{i+1}\,\,\leq \la_i+\la_{i+1}-p_1(\la_i+\la_{i+1}),
\eeq where $p_1(n)$ denotes the smallest prime factor of $n\in\N_{>1}$.
\end{thm}

\proof
For each $i=1,\dots, N-1$, set $\bm\la_{/i}:=(\la_1,\dots,\la_{i-1},\la_i+\la_{i+1},\la_{i+2},\dots,\la_N)$. This is a composition of $n$ into $N-1$ positive parts. 

We have a natural map $\phi\colon F_{\bm\la}\to F_{\bm\la_{/i}}$, forgetting the $i$-th vector subspace of a flag: to each flag
\beq\label{longflag}
0\subset V_1\subset V_2\subset V_{i-1}\subset V_i\subset V_{i+1}\subset\dots\subset V_N=\C^n,
\eeq we associate the flag
\beq\label{shortflag} 0\subset V_1\subset V_2\subset V_{i-1}\subset V_{i+1}\subset V_{i+2}\subset\dots\subset V_N=\C^n.
\eeq

We claim that the map $\phi$ realizes $F_{\bm\la}$ as a Grassmann bundle over $F_{\bm\la_{/i}}$, with fiber $G(\la_i,\la_i+\la_{i+1})$.
Indeed, denote by $Q_1',\dots, Q_{N-1}'$ the canonical quotients bundles on $F_{\bm\la_{/i}}$. The fiber of $Q'_j$ over the point \eqref{shortflag} equals
\[ V_j/V_{j-1} \text{ for $j<i$},\qquad V_{i+1}/V_{i-1}\text{ for $j=i$},\qquad V_{j+1}/V_j\text{ for $j>i$.}
\]The datum of the vector space $V_i$ in \eqref{longflag} is equivalent to the datum of a point of the Grassmannian of $\la_i$-dimensional subspace in $V_{i+1}/V_{i-1}$. Hence $F_{\bm\la}$ can be identified with the total space of the Grassmann bundle $\eu G_{\la_i}(Q'_i)\to F_{\bm\la_{/i}}$.

From the discussion after Corollary \ref{pcl2}, it follows that the limit of the ordinary quantum characteristic polynomial $f_{F_{\bm\la}}(\zeta;\bm q)$, in the regime $q_j\to 0$ for $j\neq i$, is identified with the vertical quantum characteristic polynomial of $(F_{\bm\la},F_{\bm\la_{/i}}, G(\la_i,\la_{i}+\la_{i+1}))$ at the point $q_i\in\C^*$.

The claim then follows from Theorem \ref{thmvqcp} and Corollary \ref{corvqsgrb}.
\endproof

\subsection{Examples}
\begin{example}
Let $N=n=3$ and $\bm\la=(1,1,1)$. In Example \ref{exF111}, we already computed the quantum characteristic polynomial 
\[f_{F_{\bm\la}}(\zeta;\bm q)=\zeta^6+\zeta^4(-12 q_1-12 {q_2})+\zeta^2(48 q_1^2 - 336 q_1 q_2 + 48 q_2^2)-64 q_1^3 - 192 q_1^2 q_2 - 192 q_1 q_2^2 - 64 q_2^3.
\]In the regime $q_2\to 0$, we obtain the polynomial
\beq\label{1pcqcpF111}
f_{F_{\bm\la}}(\zeta; q_1,0)=(\zeta^2-4q_1)^3,\qquad q_1\neq 0.
\eeq This can be identified with $[f_{\Pb^1}(\zeta;q_1)]^D$, where
\[f_{\Pb^1}(\zeta;q)=\zeta^2-4q,\qquad D=\dim_\C H^\bullet(G(2,3),\C)=3.
\]
Notice that $\bm\la_{/1}=(2,1)$, and $F_{\bm\la}$ is realized as a $\Pb^1$-bundle over $F_{\bm\la_{/1}}=G(2,3)$. Any zero of \eqref{1pcqcpF111} has algebraic multiplicity 3, coherently with the fact that the condition \eqref{primerelation} is not satisfied.

The partially classical regime $q_1\to 0$ is similar.\qetr
\end{example}

\begin{example}
Let $n=4$, $N=3$, $\bm\la=(2,1,1)$. In a suitable basis of $H^\bullet(F_{\bm\la},\C)$ (see Appendix \ref{appdynop}), the operator $c_1(F_{\bm\la})\sq_{\bm q}$ is represented by the matrix
\[\left(
\begin{array}{cccccccccccc}
 0 & -2 & -3 & 0 & 0 & 0 & 0 & 0 & 0 & 0 & 0 & 0 \\
 -2 q_2 & 0 & 0 & -5 & -3 & 0 & 0 & 0 & 0 & 0 & 0 & 0 \\
 0 & 0 & 0 & -2 & -5 & 0 & -3 & 0 & 0 & 0 & 0 & 0 \\
 0 & 0 & -2 q_2 & 0 & 0 & -3 & 0 & -5 & 0 & 0 & 0 & 0 \\
 0 & 0 & 0 & 0 & 0 & -2 & 0 & 0 & -3 & 0 & 0 & 0 \\
 0 & 0 & 0 & 0 & -2 q_2 & 0 & 0 & 0 & 0 & -5 & -3 & 0 \\
 -3 q_1 & 0 & 0 & 0 & 0 & 0 & 0 & -2 & -5 & 0 & 0 & 0 \\
 0 & 0 & 0 & 0 & 0 & 0 & -2 q_2 & 0 & 0 & -3 & 0 & 0 \\
 0 & -3 q_1 & 0 & 0 & 0 & 0 & 0 & 0 & 0 & -2 & -5 & 0 \\
 -5 q_1 q_2 & 0 & 0 & 0 & 0 & 0 & 0 & 0 & -2 q_2 & 0 & 0 & -3 \\
 0 & 0 & 0 & -3 q_1 & 0 & 0 & 0 & 0 & 0 & 0 & 0 & -2 \\
 0 & 0 & -5 q_1 q_2 & 0 & 0 & 0 & 0 & -3 q_1 & 0 & 0 & -2 q_2 & 0 \\
\end{array}
\right).
\]The quantum characteristic polynomial equals 
\begin{multline*}
f_{F_{\bm\la}}(\zeta;\bm q)=531441 q_1^4+5563728 q_1^3 q_2 \zeta+78732 q_1^3 \zeta^3-12637312 q_1^2 q_2^3+6060960 q_1^2 q_2^2 \zeta^2\\-1005048 q_1^2 q_2 \zeta^4+4374 q_1^2 \zeta^6+945152 q_1 q_2^4 \zeta-191488 q_1 q_2^3 \zeta^3-79744 q_1 q_2^2 \zeta^5+16704 q_1 q_2 \zeta^7\\+108 q_1 \zeta^9+4096 q_2^6-6144 q_2^5 \zeta^2+3840 q_2^4 \zeta^4-1280 q_2^3 \zeta^6+240 q_2^2 \zeta^8-24 q_2 \zeta^{10}+\zeta^{12},
\end{multline*}
whose discriminant is
\begin{multline*}
-2^{84}\, q_{1}^{16}\, q_{2}^{9}\,
\bigl(3^{18} q_{1}^{2} - 2^{24} q_{2}^{3}\bigr)^{2}\,
\bigl(3^{12} q_{1}^{2} - 2^{20} q_{2}^{3}\bigr)^{3}\\
\Bigl(
2^{6} 3^{36} q_{1}^{6}
+ (3^{18}\cdot 41 \cdot 163 \cdot 277 \cdot 1024783)\, q_{1}^{4} q_{2}^{3}
+ (2^{29}\cdot 17659 \cdot 13255661)\, q_{1}^{2} q_{2}^{6}
+ 2^{56}\, q_{2}^{9}
\Bigr)^{2}.
\end{multline*}%If $q_1=0$, the quantum spectrum at $\bm q$ consists of two eigenvalues of multiplicity 6 each. If $q_2=0$, the quantum spectrum at $\bm q$ consists of 3 eigenvalues of multiplicity 4 each.
If $3^{18} q_{1}^{2} - 2^{24} q_{2}^{3}=0$, or $3^{12} q_{1}^{2} - 2^{20} q_{2}^{3}=0$, then the quantum spectrum at $\bm q$ has an eigenvalue of algebraic multiplicity 2.

If $2^{6} 3^{36} q_{1}^{6}
+ (3^{18}\cdot 41 \cdot 163 \cdot 277 \cdot 1024783)\, q_{1}^{4} q_{2}^{3}
+ (2^{29}\cdot 17659 \cdot 13255661)\, q_{1}^{2} q_{2}^{6}
+ 2^{56}\, q_{2}^{9}=0$, then the quantum spectrum has 2 eigenvalues of algebraic multiplicity 3.

The first semiclassical characteristic polynomial of $F_{\bm\la}$ is obtained by taking the limit of $f_{F_{\bm\la}}(\zeta;\bm q)$ in the regime $q_2\to 0$, which gives
\beq\label{1stsemcpol}
\lim_{q_2\to 0}f_{F_{\bm\la}}(\zeta;\bm q)=(27 q_1 + \zeta^3)^4.
\eeq This can be identified with $[f_{G(2,3)}(\zeta;q_1)]^D$, where
\[f_{G(2,3)}(\zeta;q)=27 q + \zeta^3,\quad D=\dim_\C H^\bullet(G(3,4),\C)=4.
\]Notice that $\bm\la_{/1}=(3,1)$, and $F_{\bm\la}$ is realized as a $G(2,3)$-bundle over $G(3,4)\cong\Pb^3$. Any zero of \eqref{1stsemcpol} has algebraic multiplicity 4, coherently with the fact that the inequalities \eqref{primerelation} are not satisfied.

Similarly, the second semiclassical characteristic polynomial of $F_{\bm\la}$ is obtained by taking the limit of $f_{F_{\bm\la}}(\zeta;\bm q)$ in the regime $q_1\to 0$, which gives
\beq\label{2ndsemcpol}
\lim_{q_1\to 0}f_{F_{\bm\la}}(\zeta;\bm q)=(\zeta^2-4 q_2)^6.
\eeq This is exactly $[f_{\Pb^1}(\zeta;q_2)]^D$, where
\[f_{\Pb^1}(\zeta;q)=\zeta^2-4q,\quad D=\dim_\C H^\bullet(G(2,4),\C)=6.
\]Notice, indeed, that $\bm\la_{/2}=(2,2)$, and $F_{\bm\la}$ can be realized as a $\Pb^1$-bundle over $G(2,4)$. Any zero of \eqref{2ndsemcpol} has algebraic multiplicity 6, coherently with the fact that the inequalities \eqref{primerelation} are not satisfied. \qetr
\end{example}

\begin{example}
Let $n=5$, $N=3$, $\bm\la=(2,2,1)$. The operator $c_1(F_{\bm\la})\sq_{\bm q}$ can be represented by a $30\times 30$ matrix with characteristic polynomial
\[
\begin{aligned}
f_{F_{\bm\la}}(\zeta;\bm q)=\; &\bigl(-2^{50}\cdot 5^{10}\,q_1^6 q_2^2
+3^{15}\cdot 5^{5}\cdot 83\cdot 191\cdot 311\cdot 1481\,q_1^3 q_2^6
+3^{30}\,q_2^{10}\bigr) \\
&\;+\bigl(2^{37}\cdot 3^{5}\cdot 5^{9}\cdot 11\cdot 19\,q_1^5 q_2^3
+3^{21}\cdot 5^{4}\cdot 7^{2}\cdot 197\cdot 233\,q_1^2 q_2^7\bigr)\,\zeta \\
&\;+\bigl(-2^{21}\cdot 3^{10}\cdot 5^{6}\cdot 1109\cdot 1499\,q_1^4 q_2^4
+2\cdot 3^{24}\cdot 5\cdot 45013\,q_1 q_2^8\bigr)\,\zeta^{2} \\
&\;+\bigl(-2^{13}\cdot 3^{12}\cdot 5^{5}\cdot 7\cdot 1123901\,q_1^3 q_2^5
+2\cdot 3^{27}\cdot 5\,q_2^9\bigr)\,\zeta^{3} \\
&\;+\bigl(2^{40}\cdot 3^{2}\cdot 5^{6}\cdot 269\,q_1^5 q_2^2
+3^{18}\cdot 5\cdot 7\cdot 43\cdot 3500327\,q_1^2 q_2^6\bigr)\,\zeta^{4} \\
&\;+\bigl(2^{27}\cdot 3^{7}\cdot 5^{4}\cdot 1600219\,q_1^4 q_2^3
+3^{21}\cdot 5\cdot 106033\,q_1 q_2^7\bigr)\,\zeta^{5} \\
&\;+\bigl(2^{12}\cdot 3^{9}\cdot 5\cdot 19\cdot 621833521\,q_1^3 q_2^4
+3^{26}\cdot 5\,q_2^8\bigr)\,\zeta^{6} \\
&\;+\bigl(-2^{48}\cdot 5^{4}\,q_1^5 q_2
+2\cdot 3^{18}\cdot 5\cdot 13\cdot 17\cdot 779543\,q_1^2 q_2^5\bigr)\,\zeta^{7} \\
&\;+\bigl(2^{31}\cdot 3^{4}\cdot 5\cdot 37\cdot 5659\,q_1^4 q_2^2
-3^{18}\cdot 5\cdot 19\cdot 97\cdot 379\,q_1 q_2^6\bigr)\,\zeta^{8} \\
&\;+\bigl(2^{17}\cdot 3^{6}\cdot 5\cdot 41\cdot 73\cdot 564271\,q_1^3 q_2^3
+2^{3}\cdot 3^{22}\cdot 5\,q_2^7\bigr)\,\zeta^{9} \\
&\;+\bigl(-2^{50}\,q_1^5
+3^{12}\cdot 5^{2}\cdot 1021\cdot 206411\,q_1^2 q_2^4\bigr)\,\zeta^{10} \\
&\;+\bigl(2^{37}\cdot 3^{2}\cdot 5\cdot 373\,q_1^4 q_2
-2\cdot 3^{15}\cdot 5\cdot 966547\,q_1 q_2^5\bigr)\,\zeta^{11} \\
&\;+\bigl(2^{22}\cdot 3^{3}\cdot 5\cdot 11\cdot 639571\,q_1^3 q_2^2
+2\cdot 3^{19}\cdot 5\cdot 7\,q_2^6\bigr)\,\zeta^{12} \\
&\;+\bigl(2^{7}\cdot 3^{9}\cdot 5^{2}\cdot 11\cdot 13^{2}\cdot 37\cdot 167\,q_1^2 q_2^3\bigr)\,\zeta^{13} \\
&\;+\bigl(2^{40}\cdot 5\,q_1^4
-2^{3}\cdot 3^{12}\cdot 5^{2}\cdot 53\cdot 839\,q_1 q_2^4\bigr)\,\zeta^{14} \\
&\;+\bigl(2^{27}\cdot 5\cdot 127\cdot 311\,q_1^3 q_2
+2^{2}\cdot 3^{17}\cdot 7\,q_2^5\bigr)\,\zeta^{15} \\
&\;+\bigl(2^{10}\cdot 3^{9}\cdot 5^{2}\cdot 13\cdot 37^{2}\,q_1^2 q_2^2\bigr)\,\zeta^{16} \\
&\;-\bigl(3^{9}\cdot 5\cdot 17\cdot 24907\,q_1 q_2^3\bigr)\,\zeta^{17} \\
&\;+\bigl(-2^{31}\cdot 5\,q_1^3
+2\cdot 3^{13}\cdot 5\cdot 7\,q_2^4\bigr)\,\zeta^{18} \\
&\;+\bigl(2^{17}\cdot 3^{3}\cdot 5\cdot 11\cdot 311\,q_1^2 q_2\bigr)\,\zeta^{19} \\
&\;+\bigl(3^{6}\cdot 5\cdot 7\cdot 17\cdot 2137\,q_1 q_2^2\bigr)\,\zeta^{20} \\
&\;+\bigl(2^{3}\cdot 3^{10}\cdot 5\,q_2^3\bigr)\,\zeta^{21} \\
&\;+\bigl(2^{21}\cdot 5\,q_1^2\bigr)\,\zeta^{22} \\
&\;+\bigl(2^{7}\cdot 3^{3}\cdot 5\cdot 911\,q_1 q_2\bigr)\,\zeta^{23} \\
&\;+\bigl(3^{8}\cdot 5\,q_2^2\bigr)\,\zeta^{24} \\
&\;-\bigl(2^{10}\cdot 5\,q_1\bigr)\,\zeta^{26}
+ \bigl(2\cdot 3^{3}\cdot 5\,q_2\bigr)\,\zeta^{27}
+ \zeta^{30}.
\end{aligned}
\]
In the regime $q_2\to 0$, we obtain the first semiclassical characteristic polynomial
\[\lim_{q_2\to 0}f_{F_{\bm\la}}(\zeta;\bm q)=\zeta^{10} (\zeta^4-1024 q_1)^5,
\]which equals $[f_{G(2,4)}(\zeta;q_1)]^D$, where
\[f_{G(2,4)}(\zeta;q)=\zeta^2(\zeta^4-1024 q),\quad D=\dim_\C H^\bullet(\Pb^4,\C)=5.
\]Notice that $\bm\la_{/1}=(4,1)$, so that $F_{\bm\la}$ can be realized as a $G(2,4)$-bundle over $G(4,5)\cong\Pb^4$. For the computation of $f_{G(2,4)}(\zeta;q)$, see also \cite[Sec.\,6]{CDG20}. Notice that the first semiclassical spectrum is exceeding: it consists of one eigenvalue of algebraic multiplicity 10, and 4 of multiplicity 5. The inequalities \eqref{primerelation} are indeed satisfied.

In the regime $q_1\to 0$, we obtain the second semiclassical characteristic polynomial
\beq\label{2nd2ndsemcpol}
\lim_{q_1\to 0}f_{F_{\bm\la}}(\zeta;\bm q)=(\zeta^3+27q_2)^{10},
\eeq which equals $[f_{G(2,3)}(\zeta;q_2)]^D$, where
\[f_{G(2,3)}(\zeta;q)=27 q + \zeta^3,\quad D=\dim_\C H^\bullet(G(2,5),\C)=10.
\]Notice that $\bm\la_{/2}=(2,3)$, and $F_{\bm\la}$ is realized as a $G(2,3)$-bundle over $G(2,5)$. Any zero of \eqref{2nd2ndsemcpol} has algebraic multiplicity 10, coherently with the fact that the inequalities \eqref{primerelation} are not satisfied.\qetr
\end{example}

\begin{rem}

The spectrum of $c_1(F_{\bm\la}) \sq_{\bm q}$ at the special point $\bm q={\bf 1}=(1,\dots,1)$ is of independent interest. 
For a general Fano variety $X$, the so-called \emph{Conjecture $\mc O$} of \cite{GGI16} predicts the existence of a positive real eigenvalue $\delta_0$ equal to the spectral radius of $c_1(X)\sq_{\bm q}$, and that any other eigenvalue $\delta$ with $|\delta|=\delta_0$ satisfies $\delta=\delta_0\xi$, where $\xi$ is an $r$-th root of unity ($r$ being the index of $X$). 
For $X=F_{\bm\la}$, this property was established in \cite{CL17}.

It is natural to ask in which cases coalescences of eigenvalues of $c_1(F_{\bm\la})\sq_{\bm q}$ at $\bm q=\mathbf{1}$ occur. 
This relates to the discriminants discussed above. 
Below we list low-dimensional examples of $F_{\bm\la}$ with simple and non-simple spectra. 
It would be interesting to determine whether such simplicity can be characterized by an arithmetic condition on $(n,N,\bm\la)$, in analogy with Theorem~\ref{THMSEMICLSPEC}. 
For $N=2$, such a characterization is indeed known by \cite{Cot22}.

For $N=2$, $\bm\la=(k,n-k)$, the quantum spectrum of $F_{\bm\la}$ at $\bm q=\bf 1$ is not simple iff $p_1(n)\leq k\leq n-p_1(n)$, see \cite{Cot22}.

For $n=3$, $N=3$, $\bm\la=(1,1,1)$, the quantum spectrum of $F_{\bm\la}$ at $\bm q=\bf 1$ is not simple.

For $n=4$, $N=3$, $\bm\la=(1,2,1)$, the quantum spectrum of $F_{\bm\la}$ at $\bm q=\bf 1$ is not simple. For $\bm\la=(2,1,1),(1,1,2)$ it is simple.

For $n=5$, $N=3$, $\bm\la=(3,1,1),(1,3,1),(1,1,3)$, the quantum spectrum of $F_{\bm\la}$ at $\bm q=\bf 1$ is not simple. For $\bm\la=(2,2,1),(2,1,2),(1,2,2)$ it is simple.

For $n=6$, $N=3$, $\bm\la=(1,4,1),(2,2,2)$, the quantum spectrum of $F_{\bm\la}$ at $\bm q=\bf 1$ is not simple. For $\bm\la=(4,1,1),(1,1,4),(3,2,1),(3,1,2),(2,3,1),(2,1,3),(1,3,2),(1,2,3)$ it is simple.

For $n=7$, $N=3$, $\bm\la=(1,5,1),(3,3,1),(3,1,3),(1,3,3),(2,3,2)$, the quantum spectrum of $F_{\bm\la}$ at $\bm q=\bf 1$ is not simple. For $\bm\la=(5,1,1),(1,1,5),(4,2,1),(4,1,2),(2,4,1),$ $(2,1,4),(1,4,2),(1,2,4),(3,2,2),(2,2,3)$ it is simple.

For $n=8$, $N=3$, $\bm\la=(1,6,1),(5,1,2),(2,1,5),(4,2,2),(2,4,2),(2,2,4),(3,2,3)$, the quantum spectrum of $F_{\bm\la}$ at $\bm q=\bf 1$ is not simple. For $\bm\la=(6,1,1),(1,1,6),(5,2,1),$ $(2,5,1),(1,5,2),(1,2,5),(4,3,1),(4,1,3),(3,1,4),(3,4,1),(1,3,4),(1,4,3),(3,3,2),(2,3,3)$ it is simple.\qrem
\end{rem}

\section{\texorpdfstring{The double sequence \( \lcyr(n,N)\), and its generating functions}
                  {The double sequence lcyr(n,N), and its generating functions}}\label{sec4}
\subsection{The double sequence $\lcyr(n,N)$}

For any positive integers \( i \), \( N \), and \( n \) such that \( i < N \leq n  \), let \(\text{\textcyr{l}}(n, N, i)\) denote the number of length-\( N \) partial flag varieties in \(\mathbb{C}^n\) that do not admit an \( i \)-th semiclassical spectrum of exceeding type.

For example, we know
\[\text{\textcyr{l}}(n, 2, 1)=\begin{cases}
2(p_1(n)-1),&\text{ if $n$ is composite},\\
n-1,&\text{if $n$ is prime.}
\end{cases}
\]
\begin{prop}
The sequence $\lcyr(n,N,i)$ does not depend on $i$, that is $\lcyr(n,N,i)=\lcyr(n,N,j)$ for any $i,j< N$. Moreover, if we denote by $\lcyr(n,N)$ this common value, we have
\beq\label{eqlcyr}
\lcyr(n,N)=\sum_{h=2}^{n-N+2}\binom{n-h-1}{N-3}\lcyr(h,2).
\eeq
In particular, for fixed $N$, we have $\lcyr(n,N)=O(n^{N-1})$ as $n\to\infty$.
\end{prop}
\proof
The number $\lcyr(n,N,i)$ equals the number of tuples $(\la_1,\dots,\la_N)\in\Z^N_{>0}$ such that 
\[\la_1+\dots+\la_N=n,\quad \min\{\la_i,\la_{i+1}\}<p_1(\la_i+\la_{i+1}),\quad \la_i+\la_{i+1}-p_1(\la_i+\la_{i+1})<\max\{\la_i,\la_{i+1}\}.
\]Denote by $\mathcal{C}_{n,N}^{(i)}$ this set of compositions of $n$. We claim that, for all \( i, j \in \{1, \dots, N-1\} \), the sets \( \mathcal{C}_{n,N}^{(i)} \) and \( \mathcal{C}_{n,N}^{(j)} \) have the same cardinality. Define the cyclic shift operator \( \sigma : \mathcal{C}_{n,N} \to \mathcal{C}_{n,N} \) on the set of compositions of $n$ into $N$ positive parts by
\[
\sigma(\lambda_1, \lambda_2, \dots, \lambda_N) := (\lambda_2, \lambda_3, \dots, \lambda_N, \lambda_1).
\]
This map is bijective and preserves the total sum \( \sum_{k=1}^N \lambda_k = n \). Moreover, it maps the pair \((\lambda_i, \lambda_{i+1})\) to \((\lambda_{i-1}, \lambda_i)\), with indices modulo \( N \). Thus, \(\sigma\) maps \(\mathcal{C}_{n,N}^{(i)}\) bijectively to \(\mathcal{C}_{n,N}^{(i-1)}\), with indices taken modulo $N$.

By iterating \( \sigma \), we obtain a bijection
$\sigma^{i-j} : \mathcal{C}_{n,N}^{(i)} \to \mathcal{C}_{n,N}^{(j)}$
for any \( i, j \in \{1, \dots, N\} \). Hence
\[
\lcyr(n,N,i)=\#\mathcal{C}_{n,N}^{(i)} = \#\mathcal{C}_{n,N}^{(j)}=\lcyr(n,N,j),
\]
as desired.

To compute the number $\lcyr(n,N)$, consider a composition $(\lambda_1,\lambda_2,\lambda_3,\dots,\lambda_N)$ in $\mathcal{C}_{n,N}^{(1)}$. Set \( h := \lambda_1 + \lambda_2 \); then \((\lambda_3,\dots,\lambda_N)\) is a composition of \( n - h \) into \( N-2 \) positive integers, and the pair \( (\lambda_1, \lambda_2) \) belongs to \(\mathcal{C}_{h,2}^{(1)}\), so:
\[
(\lambda_1, \dots, \lambda_N) \in \mathcal{C}_{n,N}^{(1)} 
\quad \Longleftrightarrow \quad
(\lambda_1,\lambda_2) \in \mathcal{C}_{h,2}^{(1)} 
\text{ and } (\lambda_3,\dots,\lambda_N) \in \mathcal{C}_{n-h,N-2}.
\]
Conversely, for each composition \((\mu_3,\dots,\mu_N) \in \mathcal{C}_{n-h,N-2}\) and each \((\mu_1,\mu_2) \in \mathcal{C}_{h,2}^{(1)}\), we can form a composition \((\mu_1,\mu_2,\mu_3,\dots,\mu_N) \in \mathcal{C}_{n,N}^{(1)}\). 

The admissible values of \( h = \lambda_1 + \lambda_2 \) range from 2 (minimum: \( \lambda_1 = \lambda_2 = 1 \)) to \( n - (N-2) \) (maximum: \( \lambda_3 = \cdots = \lambda_N = 1 \)).
Therefore,
\[
\#\mathcal{C}_{n,N}^{(1)} = \sum_{h=2}^{n-N+2} \lcyr(h,2) \cdot \#\mathcal{C}_{n-h,N-2}.
\]
By Remark~\ref{numbercomposition}, we know that \(\#\mathcal{C}_{n-h,N-2} = \binom{n-h-1}{N-3}\), and \eqref{eqlcyr} is proved.

From the obvious (optimal) estimate $\lcyr(n,2)=O(n)$, we obtain
\begin{multline*}
\lcyr(n,N)=\sum_{k=N-3}^{n-3}\binom{k}{N-3}\lcyr(n-k-1,2)\leq C \sum_{k=N-3}^{n-3}\binom{k}{N-3}(n-k-1)\\\leq Cn \sum_{k=N-3}^{n-3}k^{N-3} =O(n\cdot n^{N-2})=O(n^{N-1}).\qedhere
\end{multline*}

\begin{rem}
Since $\lcyr(n,2)=O(n)$, from \eqref{eqlcyr} it follows that for fixed $N$ we have $\lcyr(n,N)=O(n^{N-1})$. For $N\geq 3$, we will obtain a better estimate in Section \ref{refLcyrN}, see equation \eqref{estimate2}.
\qrem\end{rem}

In the light of the previous proposition, the number $\lcyr(n,N)$ can be defined as 
\[\lcyr(n,N):=\#\left\{\parbox{8cm}{
\centering
$F_{\bm\la}$, with $\bm\la\in\Z^N_{>0}$ such that $\sum_{a=1}^N\la_a=n$, admitting {\it at least} one non-exceeding semiclassical spectrum 
}
\right\}.
\]

In the following table we collect the values $\lcyr(n,N)$ for $2\leq N\leq n\leq 18$.
\begin{center}
\small
\begin{tabular}{c|*{17}{c}}
$N \backslash n$ & 2 & 3 & 4 & 5 & 6 & 7 & 8 & 9 & 10 & 11 & 12 & 13 & 14 & 15 & 16 & 17 & 18 \\
\hline
2  & 1 & 2 & 2 & 4 & 2 & 6 & 2 & 4 & 2  & 10  & 2   & 12  & 2   & 4   & 2   & 16  & 2 \\
3  &   & 1 & 3 & 5 & 9 & 11 & 17 & 19 & 23 & 25  & 35  & 37  & 49  & 51  & 55  & 57  & 73 \\
4  &   &   & 1 & 4 & 9 & 18 & 29 & 46 & 65 & 88  & 113 & 148 & 185 & 234 & 285 & 340 & 397 \\
5  &   &   &   & 1 & 5 & 14 & 32 & 61 & 107 & 172 & 260 & 373 & 521 & 706 & 940 & 1225 & 1565 \\
6  &   &   &   &   & 1 & 6  & 20 & 52 & 113 & 220 & 392 & 652 & 1025 & 1546 & 2252 & 3192 & 4417 \\
7  &   &   &   &   &   & 1 & 7 & 27 & 79  & 192 & 412 & 804 & 1456 & 2481 & 4027 & 6279 & 9471 \\
8  &   &   &   &   &   &   & 1 & 8 & 35  & 114 & 306 & 718 & 1522 & 2978 & 5459 & 9486 & 15765 \\
9  &   &   &   &   &   &   &   & 1 & 9   & 44  & 158 & 464 & 1182 & 2704 & 5682 & 11141 & 20627 \\
10 &   &   &   &   &   &   &   &   & 1   & 10  & 54  & 212 & 676  & 1858 & 4562 & 10244 & 21385 \\
11 &   &   &   &   &   &   &   &   &     & 1   & 11  & 65  & 277  & 953  & 2811 & 7373  & 17617 \\
12 &   &   &   &   &   &   &   &   &     &     & 1   & 12  & 77   & 354  & 1307 & 4118  & 11491 \\
13 &   &   &   &   &   &   &   &   &     &     &     & 1   & 13   & 90   & 444  & 1751  & 5869 \\
14 &   &   &   &   &   &   &   &   &     &     &     &     & 1    & 14   & 104  & 548   & 2299 \\
15 &   &   &   &   &   &   &   &   &     &     &     &     &      & 1    & 15   & 119   & 667 \\
16 &   &   &   &   &   &   &   &   &     &     &     &     &      &      & 1    & 16    & 135 \\
17 &   &   &   &   &   &   &   &   &     &     &     &     &      &      &      & 1     & 17 \\
18 &   &   &   &   &   &   &   &   &     &     &     &     &      &      &      &       & 1 \\
\end{tabular}
\end{center}

Let us collect these numbers in several generating functions, of both ordinary and Dirichlet type:
\begin{align}
\Lcyrit_N(z)&=\sum_{n=N}^\infty\lcyr(n,N)z^n,& \Lcyr_N(s)&=\sum_{n=N}^\infty\frac{\lcyr(n,N)}{n^s},&  N\geq 2,\\
\Lcyrit(z_1,z_2)&=\sum_{2\leq N\leq n}\lcyr(n,N)z_1^nz_2^N,& 
\Lcyr(s_1,s_2)&=\sum_{2\leq N\leq n}\frac{\lcyr(n,N)}{n^{s_1}N^{s_2}}.
\end{align}
\subsection{Properties of ordinary generating functions $\Lcyrit_N(z)$, 
and Pascal rules}\label{OGFPascal}
\begin{thm}\label{thmHaL}
We have
\beq\label{HaL} \Lcyrit_N(z)=\Lcyrit_2(z)\left(\frac{z}{1-z}\right)^{N-2},\quad N\geq 2,\qquad   \Lcyrit(z_1,z_2)=\Lcyrit_2(z_1)\frac{z_2^2(1-z_1)}{1-z_1(1+z_2)}.
\eeq
\end{thm}
\proof From the relation $\binom{\ell}{k}+\binom{\ell}{k-1}=\binom{\ell+1}{k}$, one easily obtain that 
\beq\label{genfunbin} f_k(z)=\sum_{\ell=k}^\infty\binom{\ell}{k}z^\ell=\frac{z^k}{(1-z)^{k+1}}.
\eeq
Equation \eqref{eqlcyr} implies 
\[\frac{1}{z}\Lcyrit_N(z)=\Lcyrit_2(z)f_{N-3}(z)=\Lcyrit_2(z)\frac{z^{N-3}}{(1-z)^{N-2}}.
\]This proves the first claim.
To prove the second identity \eqref{HaL}, observe:
\begin{multline*}
\Lcyrit(z_1, z_2) = \sum_{N=2}^\infty \Lcyrit_N(z_1) z_2^N
= \Lcyrit_2(z_1) \cdot z_2^2\,\sum_{N=2}^\infty \left( \frac{z_1 z_2}{1 - z_1} \right)^{N-2}  \\
= \Lcyrit_2(z_1) \cdot z_2^2 \sum_{n=0}^\infty \left( \frac{z_1 z_2}{1 - z_1} \right)^n
= \Lcyrit_2(z_1) \cdot \frac{z_2^2 (1 - z_1)}{1 - z_1(1 + z_2)}.\tag*{\qedhere}
\end{multline*}

From this result, we obtain a ``Pascal rule'' and several other derived identities for the double sequence $\lcyr(n,N)$.
\begin{cor}\label{cor:pascal}
For $n\geq N\geq 2$, we have  
\beq\label{pascal} \lcyr(n,N)+\lcyr(n,N+1)=\lcyr(n+1,N+1).
\eeq
Moreover, we have
\beq\label{eq:sum}
\sum_{k=N}^n\lcyr(k,N)=\lcyr(n+1,N+1),\quad N\geq 2.
\eeq
\end{cor}
\proof
The generating functions $\Lcyrit_N(z),\Lcyrit_{N+1}(z)$ satisfy the identities
\[\Lcyrit_{N+1}(z)=\Lcyrit_{N}(z)\cdot\frac{z}{1-z}\qquad\Longrightarrow\qquad\Lcyrit_{N}(z)+\Lcyrit_{N+1}(z)=\frac{\Lcyrit_{N+1}(z)}{z}.
\]By taking the coefficients of $z^n$ of both sides, we get equation \eqref{pascal}. 
Finally, identity \eqref{eq:sum} is a telescoping sum:
\begin{multline*}
\sum_{k=N}^n\lcyr(k,N)=\sum_{k=N}^n\left(\lcyr(k+1,N+1)-\lcyr(k,N+1)\right)=\lcyr(n+1,N+1)-\cancel{\lcyr(n,N+1)}\\+\cancel{\lcyr(n,N+1)}-\cancel{\lcyr(n-1,N+1)}+\dots-\cancel{\lcyr(N+1,N+1)}+\cancel{\lcyr(N+1,N+1)}.
\end{multline*}
\endproof

\begin{cor}\label{cor:pascal-consequences}
For every integer \(k\ge 0\) the following identities hold:
\begin{align}
\Delta_n \lcyr(n,N)&:=\lcyr(n+1,N)-\lcyr(n,N)=\lcyr(n,N-1),\label{eq:diff1}&&N\geq 3,\\
\Delta_n^{\,k} \lcyr(n,N)&=\lcyr(n,N-k),\label{eq:diffk}&&N\geq 2+k,\\
\lcyr(n+k,N)&=\sum_{j=0}^{k}\binom{k}{j}\,\lcyr(n,N-j),\label{eq:forward-binomial}&&N\geq 2+k,\\
\lcyr(n+k,N+k)&=\sum_{j=0}^{k}\binom{k}{j}\,\lcyr(n,N+j),\label{eq:diagonal}&&N\geq 2,\\
\lcyr(n,N)&=\sum_{j=0}^{k}(-1)^j\binom{k}{j}\,\lcyr(n+k-j,N+k),\label{eq:alternating}&&N\geq 2.
\end{align}
\end{cor}

\begin{proof}
Identity \((\ref{eq:diff1})\) is obtained by replacing \(N\) with \(N-1\) in \eqref{pascal}.  
Applying \((\ref{eq:diff1})\) repeatedly yields \((\ref{eq:diffk})\) by induction on \(k\).

Eq.\ \((\ref{eq:forward-binomial})\) is the Newton (binomial) expansion associated to \(\Delta_n^{\,k}\): writing
\[
\lcyr(n+k,N)=\sum_{i=0}^{k}\binom{k}{i}\Delta_n^{\,i}\lcyr(n,N)
\]
and using \((\ref{eq:diffk})\) gives \((\ref{eq:forward-binomial})\). Shifting \(N\mapsto N+k\) in \((\ref{eq:forward-binomial})\) yields the diagonal identity \((\ref{eq:diagonal})\).

Identity \((\ref{eq:alternating})\) is the binomial inversion of \((\ref{eq:diagonal})\): solve \eqref{pascal} for \(\lcyr(n,N)\) and iterate, or equivalently apply the standard identity for inverting finite differences,
\[
f(n)=\sum_{j=0}^{k}(-1)^j\binom{k}{j}\,F(n+k-j)
\]
with \(F(n)=\lcyr(n,N+k)\) to obtain \((\ref{eq:alternating})\). 
This completes the proof.
\end{proof}

We record the following nontrivial cancellation identity, which will be useful later.
\begin{cor}\label{magiccancelcor}
For any $k\geq 1$ and $N\geq 2$, we have
\[
\sum_{l=1}^k\binom{k}{l}\lcyr(N+k,N+k-l)
+\sum_{j=1}^k(-1)^j\binom{k}{j}\lcyr(N+2k-j,N+k-j)=0.
\]
\end{cor}

\begin{proof}
Set $A:=\lcyr(N+k,N+k)$. Applying \eqref{eq:forward-binomial} with $n=N+k$ and $N\mapsto N+k$ yields
\[
\sum_{l=0}^k\binom{k}{l}\lcyr(N+k,N+k-l)=\lcyr(N+2k,N+k),
\]
hence
\beq\label{eqcancel1}
\sum_{l=1}^k\binom{k}{l}\lcyr(N+k,N+k-l)=\lcyr(N+2k,N+k)-A.
\eeq
Applying \eqref{eq:alternating} with $(n,N)=(N+k,N+k)$ gives
\[
\sum_{j=0}^k(-1)^j\binom{k}{j}\lcyr(N+2k-j,N+k-j)=A,
\]
so
\beq\label{eqcancel2}
\sum_{j=1}^k(-1)^j\binom{k}{j}\lcyr(N+2k-j,N+k-j)=A-\lcyr(N+2k,N+k).
\eeq
Summing the two equalities \eqref{eqcancel1}, \eqref{eqcancel2} yields $0$, as required.
\end{proof}

The {\it Eulerian polynomials} $\eu E_k\in\Z[r]$, $k\in\N$, are recursively defined by
\beq\label{recQ}
\eu E_0(r)=1,\qquad \eu E_{k+1}(r)=(k+1)r\eu E_k(r)+r(1-r)\eu E'_k(r),\quad k\geq 0.
\eeq
Their coefficients are called {\it Eulerian numbers}, see \cite{Com74,Pet15}. If we expand $\eu E_n(r)=\sum_{k=0}^n\eu E(n,k)r^k$, the Eulerian number $\mc E(n,k)$ is the number of permutations of the numbers $1$ to $n$ with $k$ ascents (i.e. in which exactly $k$ elements are greater than the previous element).

\begin{rem}
The Eulerian polynomials were first introduced by L.\,Euler in his 1749 manuscript \emph{Remarques sur un beau rapport entre les séries des puissances tant directes que réciproques}, which was posthumously published in 1768 
\cite{Eul68}. In this work, Euler investigates a remarkable connection between power series of direct and reciprocal powers, and gives a method for evaluating the Riemann zeta function at negative integers. His approach anticipates aspects of Abel summation and involves manipulating divergent series formally -- well before the rigorous development of such techniques. \qrem
\end{rem}

\begin{example}
The first Eulerian polynomials are
\begin{align*}
\eu E_0(r) &= 1, \\
\eu E_1(r) &= r, \\
\eu E_2(r) &= r + r^2, \\
\eu E_3(r) &= r + 4r^2 + r^3, \\
\eu E_4(r) &= r + 11r^2 + 11r^3 + r^4, \\
\eu E_5(r) &= r + 26r^2 + 66r^3 + 26r^4 + r^5, \\
\eu E_6(r) &= r + 57r^2 + 302r^3 + 302r^4 + 57r^5 + r^6, \\
\eu E_7(r) &= r + 120r^2 + 1191r^3 + 2416r^4 + 1191r^5 + 120r^6 + r^7, \\
\eu E_8(r) &= r + 247r^2 + 4293r^3 + 15619r^4 + 15619r^5 + 4293r^6 + 247r^7 + r^8, \\
\eu E_9(r) &= r + 502r^2 + 14608r^3 + 88234r^4 + 156190r^5 + 88234r^6 + 14608r^7 + 502r^8 + r^9. \tag*{\qetr}
\end{align*}
\end{example}

\begin{lem}[Euler--Frobenius]\label{lemQ}
The polynomials $\eu E_k(r)$ satisfy the following properties:

\begin{enumerate}

 \item%[\textbf{(iii)}] \textbf{Differential operator representation.}  
  Let $\thi_r := r \, \frac{d}{dr}$. Then $\eu E_k(r) = (1-r)^{k+1}\thi_r^k \left( \frac{1}{1 - r} \right)$.
  
   \item%[\textbf{(ii)}] \textbf{Evaluation at $r=1$.}  
  We have $\eu E_k(1) = k!$ for all $k \geq 0$.
  
  \item%[\textbf{(i)}] \textbf{Explicit formula via Stirling numbers.}  
  We have 
\[\eu E_k(r) = \sum_{n=1}^{k} \left( \sum_{j=1}^{n} 
\left\{ \begin{matrix} k \\ j \end{matrix} \right\} 
\cdot j! \cdot \binom{k - j}{n - j} \cdot (-1)^{n - j} 
\right) r^n,\qquad k\geq 1,
\]
  where $\genfrac{\{}{\}}{0pt}{}{k}{j}$ denotes the Stirling numbers of the second kind. %, i.e., the number of ways to partition a $k$-element set into $j$ non-empty subsets.

\end{enumerate}
\end{lem}
\proof Points (1) and (3) can be easily proved by induction. Point (2) follows from the recurrence relation \eqref{recQ}. 
\endproof
\endproof

Let $\mathbb D$ be the open unit disk, and $\der\mathbb D$ its boundary. A {\it Stolz region} $\Om_M(\zeta)$ with vertex $\zeta\in\der\mathbb D$ is a region of the form $\Om_M(\zeta)=\{z\in\mathbb D\colon |z-\zeta|<M(1-|z|)\}$ for some $M\in\R>0$. A set $\Om\subseteq\mathbb D$ is {\it nontangential at $\zeta$} if it can be contained in a Stolz region $\Om_M(\zeta)$.

\begin{prop}\label{growLcyrn}
The power series $\Lcyrit_N(z)$, with $N\geq 3$, converges on the unit disk $\mathbb D$, and it has the same singularities as $\Lcyrit_2(z)$ on the boundary $\der\Dl$. On the whole disk $\mathbb D$, we have the estimate \[|\Lcyrit_N(z)|=O\left( \frac{\eu E_{N-1}(|z|)}{(1-|z|)^{_N}}\right)=O((1-|z|)^{-N}).\]
If $\tau\in\der\Dl$ is a singularity for $\Lcyrit_N(z)$, we have $\Lcyrit_N(z)=O((z-\tau)^{-N})$ as $z\to\tau$ in any nontangential set at $\tau$.
\end{prop}
\proof
From the estimate $\lcyr(n,N)=O(n^{N-1})$, we deduce that $\Lcyr_N(z)$ has radius of convergence $1$. From the first of equations \eqref{HaL}, we see that $\Lcyr_N(z)$ and $\Lcyr_2(z)$ have the same singularities on $\der\mathbb D$. Moreover, by Lemma \ref{lemQ}, for $z\in\mathbb D$ we have 
\[\left|\Lcyr_N(z)\right|\leq C\sum_{n=1}^\infty n^{N-1}z^n=C\cdot \thi^{N-1}_r\left.\left(\frac{1}{1-r}\right)\right|_{r=|z|}=C\cdot \frac{\eu E_{N-1}(|z|)}{(1-|z|)^{N}}\leq \frac{C\cdot(N-1)!}{(1-|z|)^N}.
\]
Finally, if $\tau\in\der\mathbb D$ and $z\in\Om_M(\tau)$, we have $|\tau-z|<M(1-|z|)$, so that
\[\left|\Lcyr_N(z)\right|\leq \frac{C\cdot(N-1)!}{(1-|z|)^N}\leq \frac{C\cdot M\cdot(N-1)!}{|z-\tau|^N}
\]This completes the proof. 
\endproof

\begin{rem}
We expect the functions $\Lcyrit_N(z)$ to admit the whole $\der\mathbb D$ as a natural boundary. This is supported by numerical experiments, standing on a theorem of E.\,Fabry and E.\,Lindel\"of. See Appendix.
\qrem\end{rem}

\begin{rem}
The bound $\lcyr(n,N) = O(n^{N-1})$ led to corresponding estimates for the growth of $\Lcyrit_N(z)$ on the unit disk. One might be tempted to reverse the argument and ask whether such coefficient bounds can, in turn, be recovered from suitable growth estimates of $\Lcyrit_N(z)$.

If $f(z)=\sum_{n=0}^\infty a_nz^n$ is convergent on $\mathbb D$, and $f(z)=O((1-|z|)^{-k})$ on the whole $\mathbb D$ for some $k\geq 0$, then $a_n=O(n^k)$. Indeed, let $\gm_{n,k}:=\{|z|=\frac{n}{n+k}\}$, and notice that the function $[0,1]\ni r\mapsto r^{-n}(1-r)^{-k}$ attains its maximum at $r=\frac{n}{n+k}$. We have
\[a_n=\frac{1}{2\pi\sqrt{-1}}\oint_{\gm_{n,k}}\frac{f(z)}{z^{n+1}}{\rm d}z\quad\Longrightarrow\quad |a_n|\leq \left(\frac{n}{n+k}\right)^{2-n}\left(\frac{n+k}{k}\right)^{k}\sim e^k k^{-k}n^k.
\]
This estimate is sharp in general. Without additional assumptions on \( f \), one cannot improve the bound to \( a_n = O(n^{k - \varepsilon}) \) for any \( \varepsilon > 0 \). For example\footnote{This example is due to C. Remling; see the MathOverflow discussion at \url{https://mathoverflow.net/questions/409287}.}, consider the series
$f(z)=\sum_{n=1}^\infty n^{-2}(n^n)^kz^{n^n}$. 
For any $0<d<1$, the function $[0,1]\ni r\mapsto r^kd^r$ takes maximum value $(k/e)^{k}(-\log d)^{-k}\leq (k/e)^{k}(1-d)^{-k}$. Hence,
\[
|f(z)| \leq \left( \sum_{n=1}^\infty \frac{1}{n^2} \right) \left( \frac{k}{e} \right)^k (1 - |z|)^{-k},\qquad |z|<1.
\]
However, the coefficient \( a_n \) is not \( O(n^{k - \varepsilon}) \) for any \( \varepsilon > 0 \). Notice that the function $f(z)$ admits $\der\mathbb D$ as a natural boundary, by the Fabry gap theorem.

If the stronger estimate \( f(z) = O((1 - z)^{-k}) \) holds for \( z \in \mathbb{D} \), then by a transfer theorem of \cite[Thm.~4]{FO90}, \cite[Rem.~VI.10]{FS09}, one obtains the improved bound \( a_n = O(n^{k - 1}) \). In our setting, the global bound \( \Lcyrit_N(z) = O((1 - z)^{-N}) \) cannot hold throughout the disk, but only in Stolz regions, due to the expected presence of a natural boundary. Hence, the transfer theorem cannot be directly applied. 
\qrem\end{rem}

\subsection{Properties of $\Lcyr_2(s)$
} The function $\Lcyr_2(s)$ has been extensively studied in \cite{Cot22}: it was proved that the Dirichlet series $\Lcyr_2(s)$ is absolutely convergent in the half-plane ${\rm Re}(s)>2$, where it can be represented by the infinite series
\[\Lcyr_2(s)=\sum_{\substack{p\text{ prime}}}\frac{p-1}{p^s}\left(\frac{2\zeta(s)}{\text{\textcyr{p}}(s,p-1)}-1\right),
\]involving the Riemann $\zeta$-function and the truncated Euler products 
\[\text{\textcyr{p}}(s,k):=\prod_{\substack{p\text{ prime}\\ p\leq k}}\left(1-\frac{1}{p^s}\right)^{-1},\quad k\in \R_{>0},\quad s\in\C^*.
\]
The point $s=2$ is a singularity of $\Lcyr_2(s)$, a consequence of a theorem of E.\,Landau. More precisely, it is a logarithmic singularity: the local expansion at $s=2$ is 
\beq\label{asymL2}\Lcyr_2(s)=\log\frac{1}{s-2}+O(1),%C+O(s-2),
\qquad s\to 2,\quad {\rm Re}(s)>2.
\eeq 
Moreover, by analytic continuation, $\Lcyr_2(s)$ can be extended to the universal cover of the punctured half-plane $\{{\rm Re}(s)>\bar\si\}\setminus Z$, where 
\beq\label{barsi}\bar\si:=\limsup_{n}\frac{1}{\log n}\cdot\log\left(\sum_{\substack{k\leq n\\ k\text{ composite}}}\lcyr(k,2)\right),\qquad 1\leq \bar\si\leq \frac{3}{2},
\eeq
\[Z=\left\{s=\frac{\rho}{k}+1\colon \parbox{6cm}{
\centering
$\rho$\text{ zero or pole of }$\zeta(s)$,\\ $k$\text{ squarefree positive integer}
}
\right\}.
\]In particular, the following statements are equivalent:
\begin{enumerate}
\item (RH) all non-trivial zeros of $\zeta(s)$ satisfy ${\rm Re}(s)=1/2$;
\item the derivative $\Lcyr'_2(s)$ extends to a meromorphic function on ${\rm Re}(s)>3/2$, with a single pole of order one at $s=2$.
\end{enumerate}
See \cite{Cot22} and references therein for more details. 

We limit here to add a further identity, relating $\Lcyr_2(s)$ and the arithmetic functions $d,\om_0,\om_1\colon\N\to\C$ defined by
\[\begin{aligned}
d(n)&:=\#\text{ (proper and improper) divisors of $n$,}\\
\om_0(n)&:=\#\text{ distinct prime factors of $n$},\\
 \om_1(n)&:=\text{ sum of distinct prime factors of $n$,}
\end{aligned}
\]with $d(1)=1,\om_0(1)=\om_1(1)=0$.
\begin{prop}\label{identityLcyr2}
The following identity holds:
\[\Lcyr_2(s)\zeta(s)=2\left(\sum_{n=1}^\infty\frac{d(n)}{n^s}\right)\sum_{\substack{p\text{ prime}}}\frac{p-1}{p^s\cdot\text{\textnormal{\textcyr{p}}}(s,p-1)}+\sum_{n=1}^\infty\frac{\om_0(n)}{n^s}-\sum_{n=1}^\infty\frac{\om_1(n)}{n^s}.
\]
\end{prop}
\proof
Denote by $\zeta_P(s)=\sum_{p \text{ prime}}p^{-s}$ the prime zeta function. We have
\[\Lcyr_2(s)=2\zeta(s)\sum_{\substack{p\text{ prime}}}\frac{p-1}{p^s\cdot\text{\textnormal{\textcyr{p}}}(s,p-1)}-\zeta_P(s-1)+\zeta_P(s).
\]It is easy to see that $\zeta(s)^2=\sum_{n=1}^\infty d(n)n^{-s}$, $\zeta_P(s-1)\zeta(s)=\sum_{n=1}^\infty{\om_1(n)}{n^{-s}}$, and $\zeta_P(s)\zeta(s)=\sum_{n=1}^\infty{\om_0(n)}{n^{-s}}$.
\endproof

\subsection{Properties of $\Lcyr_N(s)$, 
and asymptotic estimates}\label{refLcyrN} The Hurwitz zeta function is defined, for ${\rm Re}(s) > 1$ and $a \notin \mathbb{Z}_{\le 0}$, by  
\[
\zeta(s,a) = \sum_{n=0}^\infty \frac{1}{(n+a)^s}.
\]  
It admits a meromorphic continuation to the whole complex $s$--plane with a simple pole at $s=1$ of residue $1$, and satisfies $\zeta(s,1)=\zeta(s)$, the Riemann zeta function.

\begin{lem}\label{lemHurwz}
For ${\rm Re}(s) > k$, we have 
\[\int_0^\infty \frac{e^{-k x} x^{s-1}}{(1 - e^{-x})^k} \, {\rm d}x = \Gamma(s) \cdot \zeta(s - k + 1, k).\]
\end{lem}
\proof Denote by $I_k(s)$ the integral on the l.h.s..
From the expansion $(1 - e^{-x})^{-k} = \sum_{n=0}^\infty \binom{n + k - 1}{k - 1} e^{-n x}$, we obtain
\[
I_k(s) = \sum_{n=0}^\infty \binom{n + k - 1}{k - 1} \int_0^\infty x^{s - 1} e^{-(n + k)x} \, dx.
\]
The inner integral evaluates to \( \Gamma(s)/(n + k)^s \), yielding
\[
I_k(s) = \Gamma(s) \sum_{n=0}^\infty \binom{n + k - 1}{k - 1} (n + k)^{-s}=\Gamma(s) \cdot \zeta(s - k + 1, k).\qedhere
\]

\begin{thm}\label{thmint1}
Let $N\geq 3$. For $c>2$ and ${\rm Re}(s)>N+c-2$, we have
\beq\label{convint1}
\Lcyr_N(s)=\frac{1}{2\pi\sqrt{-1}}\frac{1}{\Gm(s)}\int_{\La_c}\Lcyr_2(\tau)\Gm(\tau)\Gm(s-\tau)\zeta(s-\tau-N+3,N-2){\rm d}\tau,
\eeq where $\La_c=\{c+\sqrt{-1}t\colon t\in\R\}$.
\end{thm}
\proof 
By the Riemann reduction formula, for ${\rm Re}(s)>N$ we have
\[\Lcyr_N(s)\Gm(s)=\int_0^\infty\Lcyrit_N(e^{-x})x^{s-1}{\rm d}x.
\]From the factorization $\Lcyrit_N(z)=\Lcyrit_2(z)\left(\frac{z}{1-z}\right)^{N-2}$, we can invoke the convolution property of Mellin transform (Theorem \ref{mconv}), to obtain
\beq\label{Lnconv}
\Lcyr_N(s)\Gm(s)=(f_1*_cf_2)(s)= \frac{1}{2\pi\sqrt{-1}}\int_{\La_c}f_1(\tau)f_2(s-\tau){\rm d}\tau,
\eeq where $f_1,f_2$ are the Mellin transforms
\begin{align*}
f_1(s)&=\int_0^\infty\Lcyrit_2(e^{-x})x^{s-1}{\rm d}x=\Lcyr_2(s)\Gm(s),&&\text{(by Riemann reduction})\\
f_2(s)&=\int_0^\infty\frac{e^{-(N-2)x}}{(1-e^{-x})^{N-2}}x^{s-1}{\rm d}x=\zeta(s-N+3,N-2)\Gm(s) &&\text{(by Lemma \ref{lemHurwz}).}
\end{align*}
To justify \eqref{Lnconv} in virtue of Theorem \ref{mconv}, we need to show that
\[\int_0^\infty \left| \Lcyrit_2(e^{-x})x^c\right|^p\frac{{\rm d}x}{x}<\infty,\qquad \int_0^\infty\left| \frac{e^{-(N-2)x}}{(1-e^{-x})^{N-2}} x^{s-c}\right|^q\frac{{\rm d}x}{x}<\infty,
\]for at least one pair of positive numbers $(p,q)$ satisfying
\[\frac{p-1}{p}+\frac{q-1}{q}\geq 1.
\]
By Proposition \ref{growLcyrn}, we have
\[\int_0^\infty \left| \Lcyrit_2(e^{-x})x^c\right|^p\frac{{\rm d}x}{x}\leq \int_0^\infty \left( \frac{e^{-x}}{(1 - e^{-x})^2} \right)^p x^{pc - 1} \, {\rm d}x.
\]
To determine for which values of \( p > 0 \) and \( c \in \mathbb{R} \) the second integral
converges, we analyze the behavior of the integrand near \( x = 0 \). As \( x \to 0^+ \), we have
\[
\frac{e^{-x}}{(1 - e^{-x})^2} \sim \frac{1}{x^2},
\]
so the integrand behaves like \( x^{-2p} \cdot x^{pc - 1} = x^{pc - 2p - 1} \). This is integrable near zero if and only if $pc - 2p - 1 > -1$, that is $c>2$ (and no condition on $p$).

Finally, we determine for which \( \si, c \in \mathbb{R} \) and \( q > 0 \) the integral
\[
\int_0^\infty \frac{e^{-q(N - 2)x}}{(1 - e^{-x})^{q(N - 2)}} \, x^{q \si - q c - 1} \, dx
\]
converges. Near \( x = 0 \), we have \( 1 - e^{-x} \sim x \), so the integrand behaves like
\[
x^{-q(N - 2)} \cdot x^{q \si - q c - 1} = x^{q(\si - c - (N - 2)) - 1}.
\]
This is integrable near zero if and only if \( q(\si - c - (N - 2)) > 0 \), that is, $\si > c+N-2$.
At infinity, the exponential term \( e^{-q(N - 2)x} \) ensures convergence for all real \( \si \). Therefore, the integral is finite if and only if \( \si > c + (N - 2) \) (and no condition on $q$).

This proves the claim.
\endproof

\begin{rem}
In the particular case $N=3$, we have
\[\Lcyr_3(s)=\frac{1}{2\pi\sqrt{-1}}\frac{1}{\Gm(s)}\int_{\La_c}\Lcyr_2(\tau)\Gm(\tau)\Gm(s-\tau)\zeta(s-\tau){\rm d}\tau,\quad c>2,\quad {\rm Re}(s)>c+1,
\]where $\zeta$ is the Riemann zeta function. This can be generalized to arbitrary $N\geq 3$, see Theorem \ref{iterint} below. \qrem
\end{rem}

\begin{rem}Consider the integral transform
\[f(z)\mapsto \frac{1}{2\pi\sqrt{-1}}\frac{1}{\Gm(s)}\int_{c-\sqrt{-1}\infty}^{c+\sqrt{-1}\infty}f(\tau)\Gm(\tau)\Gm(s-\tau)\zeta(s-\tau+\al,\bt){\rm d}\tau,\quad \al,\bt\in\N^*,
\]defined on a suitable space of holomorphic functions with good vertical decay conditions. The structure of this integral transform, involving products of Gamma and Hurwitz zeta functions, is reminiscent of archimedean local factors appearing in the theory of automorphic $L$-functions. It could be of interest to clarify whether this analogy can be made precise in a suitable representation-theoretic or analytic framework; see for instance Tate's thesis \cite{Tat50}, the construction of standard $L$-functions in \cite{GJ72}, and related discussions in \cite{Bum97}.\qrem
\end{rem}

\begin{thm}\label{iterint}
Let $N\geq 2$. For $c>1$ and ${\rm Re}(s)>N+c$, we have
\beq\label{receqLcyrN}
\Lcyr_{N+1}(s)=\frac{1}{2\pi\sqrt{-1}}\frac{1}{\Gm(s)}\int_{\La_c}\Lcyr_{N}(s-\tau)\Gm(s-\tau)\Gm(\tau)\zeta(\tau){\rm d}\tau,
\eeq where $\La_c=\{c+\sqrt{-1}t\colon t\in\R\}$.
\end{thm}

\proof
The proof is similar to that of Theorem \ref{thmint1}. From the identity $\Lcyrit_{N+1}(z)=\Lcyrit_N(z)\cdot z/(1-z)$, the Riemann reduction formula, and convolution properties of Mellin transform, we obtain
\beq\label{Lconv2} \Lcyr_{N+1}(s)\Gm(s)=(f_1*_cf_2)(s)= \frac{1}{2\pi\sqrt{-1}}\int_{\La_c}f_1(\tau)f_2(s-\tau){\rm d}\tau,
\eeq where $f_1,f_2$ are the Mellin transforms
\begin{align*}
f_1(s)&=\int_0^\infty\frac{e^{-x}}{1-e^{-x}}x^{s-1}{\rm d}x=\zeta(s)\Gm(s) &&\text{(by Lemma \ref{lemHurwz})}\\
f_2(s)&=\int_0^\infty\Lcyrit_2(e^{-x})x^{s-1}{\rm d}x=\Lcyr_2(s)\Gm(s),&&\text{(by Riemann reduction}).
\end{align*}
To justify \eqref{Lconv2} in virtue of Theorem \ref{mconv}, we need to show that
\[\int_0^\infty\left| \frac{e^{-x}}{1-e^{-x}} x^{c}\right|^q\frac{{\rm d}x}{x}<\infty,\qquad \int_0^\infty \left| \Lcyrit_N(e^{-x})x^{s-c}\right|^p\frac{{\rm d}x}{x}<\infty,
\]for at least one pair of positive numbers $(p,q)$ satisfying $\frac{p-1}{p}+\frac{q-1}{q}\geq 1$. For 
\[
I_1=\int_0^\infty\left|\frac{e^{-x}}{1-e^{-x}}\,x^{c}\right|^q\frac{dx}{x},
\]
as $x\to 0$ one has $(1-e^{-x})^{-1}\sim x^{-1}$, giving $x^{q(c-1)-1}$; integrability near $0$ requires $q(c-1)>0$. As $x\to\infty$ the decay $e^{-qx}$ ensures convergence for $q>0$. Thus $I_1<\infty$ iff $q>0$ and $c>1$.  

For the second integral, by Proposition \ref{growLcyrn}, it suffices to study the convergence of 
\[
I_2=\int_0^\infty \left|\frac{\eu E_{N-1}(e^{-x})}{(1-e^{-x})^{N}}\,x^{\si-c}\right|^p\frac{dx}{x}.
\]
Near $x=0$ one has $\eu E_{N-1}(1)=(N-1)!$ and $(1-e^{-x})^{-N}\sim x^{-N}$, giving $x^{p(\si-c-N)-1}$, so $\si >c+N$. For $x\to\infty$, $\eu E_{N-1}(e^{-x})\sim e^{-x}$ if $N\ge 2$ (automatic convergence). %, while $E_0\equiv 1$ for $k=1$ yields the extra condition $s<c$, incompatible with $s>c+1$. 
Hence $I_2<\infty$ iff $k\ge 2$, $p>0$, and $\si >c+N$. 
\endproof

\begin{thm}\label{mthLcyrN}
For $N\geq 2$, the function $\Lcyr_N(s)$ is holomorphic at all points of the line ${\rm Re}(s)=N$ except at $s=N$. Moreover, in a neighborhood of $s=N$ and for ${\rm Re}(s)>N$, we have
\beq\label{asymLN} \Lcyr_N(s)\sim\frac{1}{(N-1)!}\log\left(\frac{1}{s-N}\right),\qquad s\to N.
\eeq
By analytic continuation, $\Lcyr_N(s)$ can be extended to the universal cover of the punctured half-plane
\[\{s\in\C\colon {\rm Re}(s)>\bar{\si}+N-2\}\setminus Z_N,
\]where $\bar\si\in[1;\frac{3}{2}]$ is defined in \eqref{barsi}, and 
\[Z_N=\left\{s=\frac{\rho}{k}+N-1\colon \parbox{6cm}{
\centering
$\rho$\text{ zero or pole of }$\zeta(s)$,\\ $k$\text{ squarefree positive integer}
}
\right\}.
\]
\end{thm}
\proof We argue by induction on $N$. For $N=2$, the result is already known, see \cite{Cot22}.

Inductive step. For $c>1$ and ${\rm Re}(s)>c+N$, we have
\[\Lcyr_{N+1}(s)=\frac{1}{2\pi\sqrt{-1}}\frac{1}{\Gm(s)}\int_{\La_c}\Lcyr_{N}(s-\tau)\Gm(s-\tau)\Gm(\tau)\zeta(\tau){\rm d}\tau.
\]
The integrand has a simple pole at $\tau =1$, coming from the Riemann $\zeta$-function. Shifting the integration contour to the left, we have
\[\Lcyr_{N+1}(s)=\frac{\Lcyr_N(s-1)\Gm(s-1)\Gm(1)}{\Gm(s)}+\frac{1}{2\pi\sqrt{-1}}\frac{1}{\Gm(s)}\int_{\La_{c'}}\Lcyr_{N}(s-\tau)\Gm(s-\tau)\Gm(\tau)\zeta(\tau){\rm d}\tau,
\]where $0<c'<1$ and ${\rm Re}(s)>N+c'$. As $s\to N+1$ with ${\rm Re}(s)>N+1$, the first term has a logarithmic singularity, coming form the singularity of $\Lcyr_N(s)$ at $s=N$, whereas the second term remains regular. Moreover, the function $\Lcyr_{N+1}(s)$ inherits the logarithmic singularities of $\Lcyr_N(s)$ in the half-plane $\{s\in\C\colon {\rm Re}(s)>\bar{\si}+N-2\}$, shifted by 1 to the right. The claim follows.
\endproof

\begin{rem}
The asymptotic behaviour of \(\Lcyr_N(s)\) can also be derived directly from the convolution formula \eqref{convint1} linking \(\Lcyr_2\) and \(\Lcyr_N\), rather than from the recursive relation between \(\Lcyr_N\) and \(\Lcyr_{N+1}\). However, in this case the analysis is technically more involved: as \(s \to N^+\), the pole of the Hurwitz zeta factor \(\zeta(s-t-N+3,\,N-2)\) at \(t = s-N+2\) collides with the branch point of \(\Lcyr_2(t)\) at \(t = 2\). This pole--branch cut interaction requires an explicit treatment of the Hankel contour contribution, in addition to the residue, in order to recover the correct coefficient in the asymptotic expansion.\qrem
\end{rem}

\begin{cor}\label{coRH}
The following statements are equivalent:
\begin{itemize}
\item (RH) all non-trivial zeroes of the Riemann zeta functions $\zeta(s)$ satisfy ${\rm Re}(s)=\frac{1}{2}$;
\item for any $N\geq 2$, the derivative $\Lcyr_N'(s)$ extends, by analytic continuation, to a meromorphic function in the half-plane ${\rm Re}(s)>N-\frac{1}{2}$ with a single pole of order one at $s=N$.\qed
\end{itemize}
\end{cor}

\begin{cor}
We have
\begin{align}
\label{estimate1}
\sum_{n\leq x}\lcyr(n,N)&\sim\frac{1}{N!}\frac{x^N}{\log x},\quad x\to +\infty,& & N\geq 2,\\
\label{estimate2}
\lcyr(n,N)&\sim\frac{1}{(N-1)!}\frac{n^{N-1}}{\log n},& &N\geq 3.
\end{align}
\end{cor}
\proof
Equation \eqref{estimate1} follows from the asymptotic estimate \eqref{asymLN} and from the Ikehara--Delange Tauberian theorem. See \cite[Thm.\,IV]{Del54}\cite[pag.\,350]{Ten15}. Finally, the Pascal rule \eqref{eq:sum} implies equation \eqref{estimate2}.
\endproof

\begin{cor}\label{cordensity}
For each $N\geq 3$, the set of $N$-legnth partial flag varieties admitting at least one non-exceeding semiclassical spectrum is of density zero, that is
\[\lim_{n\to\infty} \frac{\#\left\{\parbox{8cm}{
\centering
$F_{\bm\la}$, with $\bm\la\in\Z^N_{>0}$ such that $\sum_{a=1}^N\la_a\leq n$, admitting {\it at least} one non-exceeding semiclassical spectrum 
}
\right\}}{\#\{F_{\bm\la}\text{, with $\bm\la\in\Z^N_{>0}$ such that $\sum_{a=1}^N\la_a\leq n$}\}}=0.
\]
\end{cor}

\proof
The number of partial flag varieties parametrizing $N$-length flags of subspaces of $\C^k$, with $k\leq n$, equals
\[\sum_{k=N}^n\binom{k-1}{N-1}=\binom{n}{N}\sim \frac{n^N}{N!},\quad n\to\infty.
\]Hence, we have
\[\frac{\sum_{k=N}^n\lcyr(k,N)}{\sum_{k=N}^n\binom{k-1}{N-1}}\sim \frac{1}{\cancel{N!}}\cdot\frac{\cancel{n^N}}{\log n}\cdot\frac{\cancel{N!}}{\cancel{n^N}}=\frac{1}{\log n}\to 0.\tag*{\qed}
\]

\begin{rem}
We conclude our analysis of generating functions with a remark on the double Dirichlet series 
\begin{equation}\label{dDs}
\Lcyr(s_1,s_2) = \sum_{1 < N \leq n < \infty} \frac{\lcyr(n,N)}{n^{s_1} N^{s_2}}.
\end{equation}
By Theorem~\ref{mthLcyrN}, for each fixed \(N \geq 2\), the Dirichlet series
\[
\sum_{n=N}^\infty \frac{\lcyr(n,N)}{n^{s_1}}
\]
has abscissa of absolute convergence \(\sigma_a = N\). Therefore, for any fixed \((s_1^o,s_2^o) \in \mathbb{C}^2\) there exists \(N > \mathrm{Re}(s_1^o)\) such that the above series in \(n\) diverges absolutely. It follows that the double series \eqref{dDs} cannot converge absolutely at \((s_1^o,s_2^o)\). Since the choice of \((s_1^o,s_2^o)\) was arbitrary, we conclude that $\Lcyr(s_1,s_2)$ is nowhere absolutely convergent in \(\mathbb{C}^2\), and must be regarded as a purely formal Dirichlet series.\qrem
\end{rem}

\subsection{The extremal case $N=1$
} Partial flag varieties $F_{(n)}$, $n\geq 1$, of length $N=1$ are all isomorphic to a point. Their quantum cohomology coincides with their classical cohomology algebra $H^\bullet(F_{(n)},\C)\cong \C$. 

Consequently, the operator of (quantum) multiplication by $c_1(F_{(n)})=0$ is identified with the trivial morphism $0\colon\C\to\C$. Its spectrum is the singleton $\{0\}$.

Being fairly natural to consider the singleton $\{0\}$ as non-exceeding, we set 
\[\lcyr(n,1)=1,\quad n\geq 1.
\]
In this way, we are naturally led to introduce the generating functions
\[\Lcyrit_1(z)=\sum_{n=1}^\infty \lcyr(n,1)z^n=\frac{z}{1-z},\qquad \Lcyr_1(s)=\sum_{n=1}^\infty \frac{\lcyr(n,1)}{n^s}=\zeta(s).
\]
In terms of these functions, equations \eqref{HaL},  take the form
\[\Lcyrit_N(z)=\Lcyrit_2(z)\Lcyrit_1(z)^{N-2},\qquad \Lcyrit_{N+1}(z)=\Lcyrit_N(z)\Lcyrit_1(z),\qquad N\geq 2,
\]
\[\sum_{1\leq N\leq n}\lcyr(n,N)z_1^nz_2^N=\Lcyrit_1(z_1)z_2+\Lcyrit_2(z_1)\frac{z_2^2(1-z_1)}{1-z_1(1+z_2)}.
\]Similarly, if we set $\Hat\Lcyr_N(s):=\Gm(s)\Lcyr_N(s)$ for $N\geq 1$, equation \eqref{receqLcyrN} takes the form
\[\Hat\Lcyr_{N+1}(s)=\left(\Hat\Lcyr_N*\Hat\Lcyr_1\right)(s),\quad N\geq 2,
\]where $*$ is the convolution product along a vertical line contained in the common domain of holomorphy of $\Lcyr_1(s)$ and $\Lcyr_N(s)$.

\subsection{Eventual polynomiality of $N\mapsto\lcyr(N+k,N)$}\label{polcyrsec} From the identification of $\lcyr(n,N)$ as counting numbers of suitable compositions of $n$, it is easy to prove that
\[\lcyr(N,N)=1,\qquad \lcyr(N+1,N)=N,\quad N\geq 1.
\]Consider now the sequence $a(N)=\lcyr(N+2,N)$: the first values of\,\footnote{Here $\Dl=\Dl_N$ is the forward difference operator $\Dl a(N)=a(N+1)-a(N)$.} $\Dl^ka(N)$ with $k\geq 0$  are
\begin{align*}
a(N)&&1  &&2&&  5&&  9&&  14&&  20&&  27&&  35&&  44&&  54&&  65&&  77&&  90&&\dots\\
\Delta a(N)&& && 1&&  3&&  4&&  5&&  6&&  7&&  8&&  9&&   10&&   11&&   12&&   13&&\dots\\
\Delta^2a(N)&& && &&2  &&1&&  1&&   1&&   1&&   1&&   1&&   1&&    1&&   1&&   1&&\dots\\
\Delta^3a(N)&& && && && -1&&  0&&  0&&   0&&   0&&   0&&   0&&   0 &&   0 &&   0&&\dots\\
\Delta^4a(N)&& && && && &&  1&&  0&&   0&&   0&&   0&&   0&&   0&&   0 &&   0&&\dots
\end{align*}  
If we assume that in fact $\Delta^2a(N)=1$ for all $N\geq 2$, then Newton’s forward interpolation formula gives 
\begin{align*}
a(N)\overset{?}{=}&\,a(2)\binom{N-2}{0}+\Delta a(2)\binom{N-2}{1}+\Delta^2 a(2)\binom{N-2}{2}\\
=&\, 2+3(N-2)+\binom{N-2}{2}
=\binom{N+1}{2}-1
=\, \frac{N^2+N-2}{2},\quad N\geq 2.
\end{align*}  
A similar analysis for the sequence $a(N)=\lcyr(N+3,N)$ leads to the formula 
\[\lcyr(N+3,N)\overset{?}{=}\frac{N^3+3N^2-4N+12}{6},\quad N\geq 2,
\]provided that $\Dl^3a(N)=1$ for all $N\geq 2$.

In both cases the data strongly suggest that the sequences $\lcyr(N+2,N)$ and $\lcyr(N+3,N)$ exhibit \emph{eventual polynomiality}\footnote{ By eventual polynomiality we mean that a sequence coincides with a polynomial function for all sufficiently large values of~$N$.}, with quadratic and cubic closed forms respectively. 

This is indeed a general true fact, which holds for $\lcyr(N+k,N)$ for any $k\geq 0$. This is a simple consequence of Pascal rules and its derived identities found in Section \ref{OGFPascal}.

\begin{thm}\label{thmpolylcyr}
For any $k\geq 0$, the sequence $N\mapsto \lcyr(N+k,N)$ is eventually polynomial. Namely, there exists a polynomial $P_k(n)\in\Q[n]$, of degree $\deg P_k=k$, such that $\lcyr(N+k,N)=P_k(N)$ for any $N\geq 2$. Moreover, $P_k(n)$ admits the explicit binomial form 
\beq\label{polyPlcyr} P_k(n)=\sum_{j=0}^{k}\lcyr(k+2,2+j)\binom{n-2}{j}.
\eeq
\end{thm}

\proof
Set $a(N)=\lcyr(N+k,N)$. We claim that $\Dl^k a(N)=1$ for any $N\geq 2$, and the result will follow from Newton's forward interpolation formula. We have
\begin{multline*}\Dl^ka(N)=\sum_{j=0}^k(-1)^j\binom{k}{j}a(N+k-j)=\sum_{j=0}^k(-1)^j\binom{k}{j}\lcyr(N+2k-j,N+k-j)\\
=\lcyr(N+2k,N+k)+\sum_{j=1}^k(-1)^j\binom{k}{j}\lcyr(N+2k-j,N+k-j).
\end{multline*}
Applying \eqref{eq:forward-binomial} with $n=N+k$ and $N\mapsto N+k$ yields
\begin{multline*}
\Dl^ka(N)=\sum_{\ell=0}^k\binom{k}{\ell}\lcyr(N+k,N+k-\ell)+\sum_{j=1}^k(-1)^j\binom{k}{j}\lcyr(N+2k-j,N+k-j)\\
=\lcyr(N+k,N+k)=1,
\end{multline*}
where we invoke Corollary \ref{magiccancelcor} to simplify the sum. Hence, the Newton's forward interpolation formula implies the existence of $P_k(n)$, explicitly given by
\[P_k(n)=\sum_{j=0}^k\Dl^ja(2)\binom{n-2}{j},\quad \Dl^0a(2)=a(2)=\lcyr(k+2,2).
\]We claim that $\Dl^ja(2)=\lcyr(k+2,j+2)$ for any $j\geq 0$, with the clear convention $\lcyr(n,N)=0$ of $N>n$. From this, equation \eqref{polyPlcyr} follows. 

Let us prove then the more general identity $\Dl^ja(m)=\lcyr(m+k,m+j)$ for any $j\geq 0$ and $m\geq 2$, by induction on $j$. If $j=0$, we have $\Dl^0a(m)=a(m)=\lcyr(m+k,m)$, true by definition of $a$. Assume the identity holds for $j$: we have
\begin{multline*}
\Dl^{j+1}a(m)=\Dl^ja(m+1)-\Dl^ja(m)\\
=\lcyr(m+k+1,m+j+1)-\lcyr(m+k,m+j)=\lcyr(m+k,m+j+1),
\end{multline*}by Pascal rule \eqref{pascal}. This completes the proof.
\endproof

\begin{example}
The first polynomials $P_k(n)$ in binomial forms are
\begin{align*}
P_0(n)&=\binom{n-2}{0},\\
P_1(n)&=\binom{n-2}{1}+2\binom{n-2}{0},\\
P_2(n)&=\binom{n-2}{2}+3\binom{n-2}{1}+2\binom{n-2}{0},\\
P_3(n)&=\binom{n-2}{3}+4\binom{n-2}{2}+5\binom{n-2}{1}+4\binom{n-2}{0},\\
P_4(n)&=\binom{n-2}{4}+5\binom{n-2}{3}+9\binom{n-2}{2}+9\binom{n-2}{1}+2\binom{n-2}{0},\\
P_5(n)&=\binom{n-2}{5}+6\binom{n-2}{4}+14\binom{n-2}{3}+18\binom{n-2}{2}+11\binom{n-2}{1}+6\binom{n-2}{0},\\
&\dots
\end{align*}
\end{example}
In polynomial form, we have
\begin{multline*}
P_0(n)=1,\quad P_1(n)=n,\quad P_2(n)=\frac{n^2+n-2}{2},\quad P_3(n)=\frac{n^3+3n^2-4n+12}{6},\\
P_4(n)=\frac{k^4+6k^3-k^2+42k-96}{24},\quad P_5(n)=\frac{n^5+ 10 n^4 + 15 n^3+110 n^2- 376 n+720}{120}.\tag*{\qetr}
\end{multline*}

\section{\texorpdfstring{The double sequence \( \tlcyr(n,N)\), walks on graphs, and generating functions}
                  {The double sequence tlcyr(n,N), walks on graphs, and generating functions}}\label{sec5}

\subsection{The double sequence $\tlcyr(n,N)$} Alongside the double sequence $\lcyr(n,N)$, we consider a second sequence of potential interest. 

Define the quantity $\tilde\lcyr(n,N)$, for $2\leq N\leq n$, as follows:
\[
\tilde\lcyr(n,N):=\#\left\{\parbox{8cm}{
\centering
$F_{\bm\la}$, with $\bm\la\in\Z^N_{>0}$, $\sum_{a=1}^N\la_a=n$, admitting \emph{only} non-exceeding semiclassical spectra
}
\right\}.
\]

In other words, $\tilde\lcyr(n,N)$ counts the number of compositions $\bm\la$ of $n$ into $N$ positive integers such that, for \emph{every} $i = 1, \dots, N-1$, the following inequalities hold:
\beq\label{eq:allnotexc}
\min\{\la_i,\la_{i+1}\}<p_1(\la_i+\la_{i+1}),\quad \la_i+\la_{i+1}-p_1(\la_i+\la_{i+1})<\max\{\la_i,\la_{i+1}\}.
\eeq

It is immediate to observe that
\[
\tilde\lcyr(n,N)\leq \lcyr(n,N), \qquad \tilde\lcyr(n,2) = \lcyr(n,2), \qquad \tilde\lcyr(n,n) = \lcyr(n,n) = 1.
\]
In the following table we collect the values $\tlcyr(n,N)$ for $2\leq N\leq n\leq 18$.

\begin{center}
\small
\begin{tabular}{c|*{17}{c}}
$N \backslash n$ & 2 & 3 & 4 & 5 & 6 & 7 & 8 & 9 & 10 & 11 & 12 & 13 & 14 & 15 & 16 & 17 & 18 \\
\hline
2  & 1 & 2 & 2 & 4 & 2 & 6 & 2 & 4 & 2  & 10  & 2   & 12  & 2   & 4   & 2   & 16  & 2 \\
3  &   & 1 & 3 & 4 & 9 & 7 & 16 & 11 & 16 & 12  & 30  & 19  & 40  & 26  & 30  & 23  & 55 \\
4  &   &   & 1 & 4 & 7 & 16 & 19 & 34 & 39 & 46  & 53  & 74  & 87  & 110 & 135 & 120 & 159 \\
5  &   &   &   & 1 & 5 & 11 & 26 & 41 & 68 & 102 & 120 & 171 & 195 & 287 & 315 & 473 & 434 \\
6  &   &   &   &   & 1 & 6  & 16 & 40 & 76 & 130 & 222 & 290 & 442 & 530 & 786 & 924 & 1358 \\
7  &   &   &   &   &   & 1 & 7 & 22 & 59  & 128 & 236 & 434 & 642 & 1009 & 1355 & 1960 & 2568 \\
8  &   &   &   &   &   &   & 1 & 8 & 29  & 84  & 202 & 406 & 791 & 1306 & 2129 & 3162 & 4608 \\
9  &   &   &   &   &   &   &   & 1 & 9   & 37  & 116 & 304 & 665 & 1369 & 2475 & 4233 & 6799 \\
10 &   &   &   &   &   &   &   &   & 1   & 10  & 46  & 156 & 441 & 1044 & 2272 & 4430 & 8001 \\
11 &   &   &   &   &   &   &   &   &     & 1   & 11  & 56  & 205 & 621  & 1581 & 3638 & 7571 \\
12 &   &   &   &   &   &   &   &   &     &     & 1   & 12  & 67  & 264  & 853  & 2322 & 5646 \\
13 &   &   &   &   &   &   &   &   &     &     &     & 1   & 13  & 79   & 334  & 1147 & 3322 \\
14 &   &   &   &   &   &   &   &   &     &     &     &     & 1   & 14   & 92   & 416  & 1514 \\
15 &   &   &   &   &   &   &   &   &     &     &     &     &     & 1    & 15   & 106  & 511 \\
16 &   &   &   &   &   &   &   &   &     &     &     &     &     &      & 1    & 16   & 121 \\
17 &   &   &   &   &   &   &   &   &     &     &     &     &     &      &      & 1    & 17 \\
18 &   &   &   &   &   &   &   &   &     &     &     &     &     &      &      &      & 1 \\
\end{tabular}
\end{center}

Similarly to what was done before, we can collect the values of the double sequence $\tilde\lcyr(n,N)$ into generating functions, either of ordinary or Dirichlet type:
\[
\tLcyrit_N(z)=\sum_{n=N}^\infty \tlcyr(n,N)\, z^n, \qquad
\tLcyr_N(s)=\sum_{n=N}^\infty \frac{\tlcyr(n,N)}{n^s}.
\]

Similarly to the case of the double sequence $\lcyr(n,N)$, we naturally extend the definition for $N=1$, by setting
\[\tlcyr(n,1)=1,\quad n\geq 1.
\]

The properties of the double sequence $\tlcyr(n,N)$ appear to be more elusive than those of $\lcyr(n,N)$; for example, no obvious Pascal-type rule seems to hold. As a consequence, the analytical study of the generating functions is more challenging.

We propose the following objectives here:
\begin{itemize}
\item to describe the numbers $\tlcyr(n,N)$ and the generating functions $\tLcyrit_N(z)$ in terms of walks on suitable graphs (Theorem \ref{thmgenfuncLtilde}, Corollary \ref{corltildegraph});
\item to deduce an eventual polynomial property for the sequence $N\mapsto\tlcyr(N+k,N)$ for any $k\geq 0$ (Theorem \ref{thmpolytlcyr}).
\end{itemize}

\subsection{Graphs, transfer matrices, generating functions}\label{secgraphGm} Fix $\sf m\in\N_{>0}$. Let $\Gm_{\sf m}$ be the oriented graph with vertices $\{1,\dots,\sf m\}$, and arrows as follows: there is an arrow $(i,j)$ if and only if 
\[\min\{i,j\}<p_1(i+j),\quad \text{and}\quad i+j-p_1(i+j)<\max\{i,j\}.
\]

\begin{lem}\label{symGmm}
The graph $\Gm_{\sf m}$ is symmetric: if there is an arrow $(i,j)$, then there is also the arrow $(j,i)$. Moreover, the only loop $(i,i)$ occurs for $i=1$.\qed
\end{lem}

Let $M^{[\sf m]}$ be the $\sf m\times \sf m$ adjacency matrix of $\Gm_{\sf m}$, so that $M^{[\sf m]}_{ij}=1$ if there is the arrow $(i,j)$, and $M^{[\sf m]}_{ij}=0$ otherwise. By Lemma \ref{symGmm}, $(M^{[\sf m]})^T=M^{[\sf m]}$, and $M^{[\sf m]}_{ii}\neq 0$ if and only if $i=1$.

We assign now a monomial weight to each vertex: the vertex $j$ is assigned the weight $z^j$. %In this way a path $i_1 \to i_2 \to \dots \to i_N$ has weight $z^{i_1 + \dots + i_N}$. 
This information is encoded in the {\it transfer matrix} $T^{[\sf m]}(z)$, defined by
\beq\label{eq:trmatr}
T^{[\sf m]}(z)_{ij}=M^{[\sf m]}_{ij}z^j,\quad i,j=1,\dots,\sf m.
\eeq

\begin{example}
For $\sf m=5$ and $\sf m=10$ respectively, we have 
\[M^{[5]}=\left(
\begin{array}{ccccc}
 1 & 1 & 1 & 1 & 1 \\
 1 & 0 & 1 & 0 & 1 \\
 1 & 1 & 0 & 1 & 0 \\
 1 & 0 & 1 & 0 & 0 \\
 1 & 1 & 0 & 0 & 0 \\
\end{array}
\right),\qquad T^{[10]}(z)=\left(
\begin{array}{cccccccccc}
 z & z^2 & z^3 & z^4 & z^5 & z^6 & z^7 & z^8 & z^9 & z^{10} \\
 z & 0 & z^3 & 0 & z^5 & 0 & z^7 & 0 & z^9 & 0 \\
 z & z^2 & 0 & z^4 & 0 & 0 & 0 & z^8 & 0 & z^{10} \\
 z & 0 & z^3 & 0 & 0 & 0 & z^7 & 0 & z^9 & 0 \\
 z & z^2 & 0 & 0 & 0 & z^6 & 0 & z^8 & 0 & 0 \\
 z & 0 & 0 & 0 & z^5 & 0 & z^7 & 0 & 0 & 0 \\
 z & z^2 & 0 & z^4 & 0 & z^6 & 0 & 0 & 0 & z^{10} \\
 z & 0 & z^3 & 0 & z^5 & 0 & 0 & 0 & z^9 & 0 \\
 z & z^2 & 0 & z^4 & 0 & 0 & 0 & z^8 & 0 & z^{10} \\
 z & 0 & z^3 & 0 & 0 & 0 & z^7 & 0 & z^9 & 0 \\
\end{array}
\right).
\]\qetr
\end{example}

Introduce the column vectors $v_{\sf m}(z)=(z,z^2,\dots,z^{\sf m})^T$, and ${\bf 1}_{\sf m}=(1,\dots,1)^T$.

For each $N\geq 1$, define the generating function 
\beq\label{euLN} \eu L^{[\sf m]}_{N}(z):=v_{\sf m}(z)^T\cdot T^{[\sf m]}(z)^{N-1}\cdot\bf 1_{\sf m}.
\eeq

For each walk\footnote{A walk is a finite sequence of composable arrows which joins a sequence of vertices.} $\si$ in $\Gm_{\sf m}$, define the monomial ${\sf w}(\si):=\prod_{i\text{ node of }\si}z^i$, and denotes its length by $|\si|:=\#\{\text{arrows in }\si\}=\#\{\text{nodes of $\si$}\}-1$. 

\begin{example}Consider the walk $\si=1\to 4\to 3\to 1\to 1$ in $\Gm_{4}$. We have ${\sf w}(\si)=z^{1+4+3+1+1}=z^{10}$ and $|\si|=4$. \qetr
\end{example}

\begin{prop}
The function $\eu {L}^{[{\sf m}]}_N(z)$ is the polynomial in $z$ given by
\beq\label{euLN1}
\mathcal{L}^{[{\sf m}]}_N(z) = \sum_{i_1,\dots,i_N=1}^{\sf m} z^{i_1+\dots+i_N} \prod_{k=1}^{N-1} M^{[{\sf m}]}_{i_k,i_{k+1}}.
\eeq
It represents the sum of monomials ${\sf w}(\si)$ over all walks $\si$ in the graph $\Gamma_{\sf m}$, with length $|\si|=N-1$. That is
\beq\label{euLN2}
\eu {L}^{[{\sf m}]}_N(z)=\sum_{\substack{\si\text{ in }\Gm_{\sf m}\\ |\si|=N-1}}{\sf w}(\si).
\eeq
\end{prop}

\proof
Equation \eqref{euLN1} is the expansion of the defining equation \eqref{euLN}.
Each term in  \eqref{euLN1} corresponds to a walk $\si=(i_1,\dots,i_N)$ in the graph $\Gamma_{\sf m}$, with monomial ${\sf w}(\si)=z^{i_1+\dots+i_N}$. The adjacency matrix ensures that a term in \eqref{euLN1} contributes if and only if each pair $(i_k,i_{k+1})$ satisfies the inequalities defining $\Gamma_{\sf m}$. 
\endproof

\begin{rem}
Alternatively, we can assign ``costs'' to arrows. If we assign the cost $z^j$ to the arrow $(i,j)$, and concatenation of arrows corresponds to multiplication of costs, then the walk $i_1 \to i_2 \to \dots \to i_N$ has cost $z^{i_2 + \dots + i_N}$. Then the function $\eu {L}^{[{\sf m}]}_N(z)$ equals the sum of costs over walks in $\Gm_{\sf m}$, with length $N$ and starting from 1.\qrem
\end{rem}

\begin{example}
For $N=1$, we have $\eu L^{[\sf m]}_N(z)=z+z^2+\dots+z^{\sf m}$. \qetr
\end{example}

\begin{example}\label{L3}
For $\sf m=3$, we have
\begin{align*}
\eu L^{[\sf m]}_1(z)&=z+z^2+z^3,\\
\eu L^{[\sf m]}_2(z)&=z^2+2z^3+2 z^4+2 z^5,\\
\eu L^{[\sf m]}_3(z)&=z^3+3 z^4+4 z^5+6 z^6+2 z^7+z^8,\\
\eu L^{[\sf m]}_4(z)&=z^4+4 z^5+7 z^6+12 z^7+9 z^8+6 z^9+2 z^{10}.\tag*{\qetr}
\end{align*}
\end{example}

\begin{prop}\label{limeuLN}
Let $N \ge 1$ be fixed. Then the limit
$
\eu{L}_N(z) :=\lim_{{\sf m}\to\infty} \eu {L}^{[{\sf m}]}_N(z)
$
exists as a formal power series in $z$.
\end{prop}

\begin{proof}
Recall that $
\eu {L}^{[{\sf m}]}_N(z)$
represents the sum of monomials ${\sf w}(\si)$ over all walks of length $N$ in the graph $\Gamma_{\sf m}$, where the monomial of a walk $i_1 \to i_2 \to \dots \to i_N$ is $z^{i_1 + \dots + i_N}$. 

Consider a fixed power $z^k$. Any walk contributing to $z^k$ must involve only vertices $i_j \le k$.  

For ${\sf m} \ge k$, all edges connecting vertices up to $k$ are already present in $\Gamma_{\sf m}$ and do not change as $\sf m$ increases, since the adjacency conditions depend only on the vertices themselves, not on the maximum $\sf m$. Therefore, the coefficient of $z^k$ stabilizes for sufficiently large $\sf m$.  

We can thus define the limit
$\eu{L}_N(z) = \lim_{{\sf m}\to\infty} \eu{L}^{[{\sf m}]}_N(z) = \sum_{k \ge N} a_k z^k,$
where $a_k$ counts the number of walks $\si$ of length $N$ and with ${\sf w}(\si)=z^k$ in the infinite graph $\Gamma_\infty$.  

Hence, each coefficient stabilizes for sufficiently large $\sf m$, and the limit exists as a formal power series. 
\end{proof}

\begin{thm}\label{thmgenfuncLtilde}
For any $N\geq 1$, we have 
$
\tLcyrit_N(z)=\eu L_N(z).
$
\end{thm}

\begin{proof}
Recall the expansions \eqref{euLN1}, \eqref{euLN2} for $\mathcal{L}^{[{\sf m}]}_N(z)$.
Taking the limit ${\sf m}\to\infty$, all paths of finite vertices exist, and the coefficients of $z^n=z^{i_1+\cdots+i_N}$ stabilize. Denote
$\lambda_k := i_k,$ with $k=1,\dots, N$.
Then the contributing sequences $(\lambda_1,\dots,\lambda_N)$ are exactly the compositions of $n=\lambda_1+\dots+\lambda_N$ into $N$ positive parts satisfying
\[
\min\{\lambda_i,\lambda_{i+1}\} < p_1(\lambda_i+\lambda_{i+1}), \quad
\lambda_i+\lambda_{i+1}-p_1(\lambda_i+\lambda_{i+1}) < \max\{\lambda_i,\lambda_{i+1}\}, \quad i=1,\dots,N-1.
\]
Therefore, the limit series coincides with the generating function $\tLcyrit_N(z)$ of the numbers $\tlcyr(n,N)$.
\end{proof}

\begin{cor}\label{corltildegraph}
For any ${\sf m}\geq n-N+1$, we have 
\beq\label{eqcorltildegraph}
\tlcyr(n,N)=\sum_{\substack{\bm\la\in\Z^N_{>0}\\\la_1+\dots+\la_N=n}}\prod_{i=1}^{N-1}M^{[\sf m]}_{\la_i,\la_{i+1}}.
\eeq
\end{cor}
\proof
Notice that any composition $\bm\la\in\Z^N_{>0}$ of $n$ has parts $\leq n-(N-1)$ (the extremal case in which $N-1$ parts equal 1).
\endproof

\begin{cor}\label{congrgenfun0}
For any $N,{\sf m}\geq 1$, we have 
\[\tLcyrit_N(z)\equiv \eu L^{[\sf m]}_N(z)\mod(z^{{\sf m}+N}).\tag*{\qed}
\]
\end{cor}

\subsection{Rationality of the generating functions $\eu L^{[\sf m]}(z,t)$}\label{secratioLm} Introduce a family of generating functions 
\[\eu L^{[\sf m]}(z,t)=\sum_{N=1}^\infty\eu L^{[{\sf m}]}_N(z)t^N,\quad {\sf m}\geq 1.
\]
This generating function provides a good approximation of the double generating function 
\[\tLcyrit(z,t)=\sum_{N=1}^\infty\sum_{n=N}^\infty\tlcyr(n,N)z^nt^N=zt+z^2t+z^2t^2+z^3t+2z^3t^2+z^3t^3+\dots .
\]
\begin{prop}
For any $\sf m\geq 1$, we have 
\[\tLcyrit(z,t)\equiv \eu L^{[\sf m]}(z,t)\mod I,
\]where $I$ is the ideal of $\Z[\![z,t]\!]$ generated by $(z^{{\sf m}+h}t^h)_{h\in\N_{>0}}$.
\end{prop}
\proof
It follows form Corollary \ref{congrgenfun0}.
\endproof

Remarkably, it turns out that for any ${\sf m}\geq 1$ the function $\eu L^{[\sf m]}(z,t)$ is rational.

\begin{thm}\label{thmratLm}
For any ${\sf m}\geq 1$, we have $\eu L^{[\sf m]}(z,t)\in\Q(z,t)$.
\end{thm}
\proof
Let $p(\zeta,z)=\zeta^{\sf m}+\sum_{i=0}^{\sf m-1}a_i(z)\zeta^i$ be the characteristic polynomial of $T^{[\sf m]}(z)$. By Cayley--Hamilton Theorem, we deduce
\[\eu L^{[\sf m]}_{N+\sf m}(z)+\sum_{i=0}^{\sf m-1}a_i(z)\eu L^{[\sf m]}_{N+i}(z)=0,\quad \text{ for any }N\geq 1.
\]Multiply both sides by $t^{N+\sf m}$, and sum over all $N\geq 1$. We obtain
\[\left(\eu L^{[\sf m]}(z,t)-\sum_{i=1}^{\sf m}\eu L^{[\sf m]}_i(z)t^i\right)+\sum_{i=0}^{\sf m-1}a_i(z)t^{{\sf m}-i}\left(\eu L^{[\sf m]}(z,t)-\sum_{j=1}^{i}\eu L^{[\sf m]}_j(z)t^j\right)=0.
\]We conclude that
\begin{multline*}
\eu L^{[\sf m]}(z,t)\left(1+\sum_{i=0}^{\sf m-1}a_i(z)t^{{\sf m}-i}\right)=\sum_{i=1}^{\sf m}\eu L^{[\sf m]}_i(z)t^i+\sum_{i=0}^{\sf m-1}\sum_{j=1}^{i}a_i(z)\eu L^{[\sf m]}_j(z)t^{{\sf m}-i+j}\\
\Longrightarrow \eu L^{[\sf m]}(z,t)=\frac{\sum_{i=1}^{\sf m}\eu L^{[\sf m]}_i(z)t^i+\sum_{i=0}^{\sf m-1}\sum_{j=1}^{i}a_i(z)\eu L^{[\sf m]}_j(z)t^{{\sf m}-i+j}}{1+\sum_{i=0}^{\sf m-1}a_i(z)t^{{\sf m}-i}}.
\end{multline*}
This proves the claim.
\endproof

\begin{example}
Let $\sf m=3$. We have
\[T^{[3]}(z)=\left(
\begin{array}{ccc}
 z & z^2 & z^3 \\
 z & 0 & z^3 \\
 z & z^2 & 0 \\
\end{array}
\right),\quad\text{with char.\,\,pol. }p(\zeta,z)=\zeta^3-z\zeta^2-(z^3+z^4+z^5)\zeta-z^6.
\]By Cayley--Hamilton theorem, we deduce 
\begin{multline*}\left(\eu L^{[3]}(z,t)-t\eu L^{[3]}_1(z)-t^2\eu L^{[3]}_2(z)-t^3\eu L^{[3]}_3(z)\right)-zt\left(\eu L^{[3]}(z,t)-t\eu L^{[3]}_1(z)-t^2\eu L^{[3]}_2(z))\right)\\-t^2(z^3+z^4+z^5)\left(\eu L^{[3]}(z,t)-t\eu L^{[3]}_1(z)\right)-t^3z^6\eu L^{[3]}(z,t)=0.
\end{multline*}
Hence, from the explicit expressions of $\eu L^{[3]}_i(z)$, $i=1,2,3$, in Example \ref{L3}, we have
\[\eu L^{[3]}(z,t)=-\frac{t z \left(t^2 z^5+2 t z^4+t z^3+t z^2+z^2+z+1\right)}{t^3 z^6+t^2 z^5+t^2 z^4+t^2 z^3+t z-1}.
\]Notice that
\begin{multline*}\eu L^{[3]}(z,t)\equiv 
t z + (t + t^2) z^2 + (t + 2 t^2 + t^3) z^3 +\\ (2 t^2 + 3 t^3 + 
    t^4) z^4 + (\cancel{2 t^2} + 4 t^3 + 4 t^4 + t^5) z^5 \\+ (\cancel{6 t^3} + 7 t^4 + 
    5 t^5 + t^6) z^6 + (\cancel{2 t^3} + \cancel{12 t^4} + 11 t^5 + 6 t^6 + t^7) z^7+\dots\\
    \equiv \tLcyrit(z,t) \mod\langle z^{3+h}t^h\colon h\in\N_{>0}\rangle.
\end{multline*}
\qetr
\end{example}

\subsection{Eventual polynomiality of $N\mapsto \tlcyr(N+k,N)$}\label{poltlcyrsec}
Starting from the definition of $\tlcyr(n,N)$ as counting suitable compositions of $n$, one easily checks that  
\[
\tlcyr(N,N)=1,\qquad \tlcyr(N+1,N)=N,\quad N\geq 1.
\]  

Consider now the sequence $\tlcyr(N+2,N)$, $N\geq 1$, whose initial values are  
\[
1, 2, 4, 7, 11, 16, 22, 29, 37, 46, 56, 67, 79, 92, 106, 121,\dots
\]  
At first glance, these are precisely the {\it central polygonal numbers}, namely  
\[
c(N)=\frac{N^2-N+2}{2}=1+\binom{N}{2}=\binom{N-1}{0}+\binom{N-1}{1}+\binom{N-1}{2},\quad N\geq 1
\]  
It is therefore natural to conjecture that  
\[
\tlcyr(N+2,N)\overset{?}{=}c(N),\quad N\geq 1.
\]  

Similarly, let us consider the sequence $a(N)=\tlcyr(N+3,N)$ and examine its successive differences\footnote{The difference operator $\Delta$ acts on a sequence $a(N)$ by $\Delta a(N)=a(N+1)-a(N)$, for $N\geq 1$.}:  
\begin{align*}
a(N)&&1  &&4&&  9&&  16&&  26&&  40&&  59&&  84&&  116&&  156&&  205&&  264&&  334&&\dots\\
\Delta a(N)&& && 3&&  5&&  7&&  10&&  14&&  19&&  25&&  32&&   40&&   49&&   59&&   70&&\dots\\
\Delta^2a(N)&& && &&2  &&2&&  3&&   4&&   5&&   6&&   7&&   8&&    9&&   10&&   11&&\dots\\
\Delta^3a(N)&& && && && 0&&  1&&  1&&   1&&   1&&   1&&   1&&   1 &&   1 &&   1&&\dots\\
\Delta^4a(N)&& && && && &&  0&&  0&&   0&&   0&&   0&&   0&&   0&&   0 &&   0&&\dots
\end{align*}  

If we assume that in fact $\Delta^3a(N)=1$ for all $N\geq 2$, then Newton’s forward interpolation formula gives  
\begin{align*}
a(N)\overset{?}{=}&\,a(2)\binom{N-2}{0}+\Delta a(2)\binom{N-2}{1}+\Delta^2 a(2)\binom{N-2}{2}+\Delta^3 a(2)\binom{N-2}{3}\\
=&\, 4+5(N-2)+2\binom{N-2}{2}+\binom{N-2}{3}\\
=&\, \frac{N^3-3N^2+26N-24}{6},\quad N\geq 2.
\end{align*}  

Thus we are led to the conjectural identities  
\begin{align}
\label{idconj1}&\tlcyr(N+2,N)\overset{?}{=}\frac{N^2-N+2}{2},\quad N\geq 1, \\
\label{idconj2}&\tlcyr(N+3,N)\overset{?}{=}\frac{N^3-3N^2+26N-24}{6},\quad N\geq 2.
\end{align}  
In both cases, the data strongly suggest that the sequences $\tlcyr(N+2,N)$ and $\tlcyr(N+3,N)$ exhibit eventual polynomiality, with quadratic and cubic closed forms, respectively. However, a direct proof of these conjectural identities starting from the definition of $\tlcyr(n,N)$ appears to be quite difficult.

More generally, we may conjecture eventual polynomiality for any sequence $N \mapsto \tlcyr(N+k,N)$, $k\geq 0$, in close analogy with the sequence $N \mapsto \lcyr(N+k,N)$, as shown in Section~\ref{polcyrsec}. Remarkably, this turns out to be true, and in this section we provide a proof of the following result.

\begin{thm}\label{thmpolytlcyr}
The sequence $N\mapsto\tlcyr(N+k,N)$ is eventually polynomial for any $k\geq 0$. There exists $N_0\in\N_{>0}$ and a polynomial $\Tilde P_k\in\Q[n]$, with $\deg \Tilde P_k\leq k$, such that $\tlcyr(N+k,N)=\Tilde P_k(N)$ for any $N\geq N_0$.
\end{thm}

\begin{rem}
Notice that we already have $\tlcyr(N+k, N) = O(N^k)$ for fixed $k \geq 1$ and large $N$, which implies that $\deg \Tilde{P}_k \leq k$. This follows from the inequality $\tlcyr(n, N) \leq \lcyr(n, N)$ together with the polynomiality property stated in Theorem \ref{thmpolylcyr}. Theorem \ref{thmpolytlcyr}, however, is a considerably stronger result.\qrem
\end{rem}

For the proof of Theorem \ref{thmpolytlcyr}, we need some preliminary results.

\begin{lem}\label{lemtchn1}
Let $Q_1(z,t)$ and $Q_2(z,t)$ be polynomials over a field\,\footnote{No restriction on the characteristic ${\rm char}(K)$.} $K$. 
Define
\[
P_1(z,t)=z t\,Q_1(z,t),\qquad P_2(z,t)=1- z t+Q_2(z,t),
\]
and the rational function
\[
F(z,t)=\frac{P_1(z,t)}{P_2(z,t)}.
\]
Assume moreover that $Q_2(0,0)=0$, and fix an integer $m\geq 0$. 
Then the coefficient
\[
H_m(t):=[u^m]\,F\!\big(u,\tfrac{t}{u}\big)
\]
is a well-defined rational function of $t$. More precisely there exist an integer $s\ge1$ and a polynomial $R_m(t)\in K[t]$ such that
\[
H_m(t)=\frac{R_m(t)}{\big(1- t\big)^{s}}.
\]
Moreover, the extraction of $[u^m]$ can be written explicitly as the finite sum
\[
[u^m]\,F\!\big(u,\tfrac{t}{u}\big)
=\frac{t}{1- t}\sum_{r=0}^{r_{\max}} (-1)^r(1- t)^{-r}\;
[u^m]\!\big( Q_1(u,t/u)\, (Q_2(u,t/u))^r\big),
\]
where $r_{\max}\ge0$ is finite and depends only on $Q_1,Q_2$ and $m$.
\end{lem}

\proof
We proceed in several steps.

\textbf{Step 1.}
Substitute $z=u$ and $t\mapsto t/u$ in $F$. Using $P_1(z,t)=zt\,Q_1(z,t)$ we get the simplification
\[
F\!\big(u,\tfrac{t}{u}\big)
=\frac{u\cdot (t/u)\, Q_1\!\big(u,\tfrac{t}{u}\big)}{1- u\cdot(t/u)+Q_2\!\big(u,\tfrac{t}{u}\big)}
=\frac{t\;Q_1\!\big(u,\tfrac{t}{u}\big)}{1- t + Q_2\!\big(u,\tfrac{t}{u}\big)}.
\]
Set for brevity
\[
A(u,t):=Q_1\!\big(u,\tfrac{t}{u}\big),\qquad B(u,t):=Q_2\!\big(u,\tfrac{t}{u}\big).
\]
Note that $A(u,t)$ and $B(u,t)$ are finite sums of monomials of the form $c\,u^{\alpha}t^{\beta}$ with integer exponents $\alpha$ (possibly negative) and $\beta\ge0$. The hypothesis $Q_2(0,0)=0$ implies $B(u,t)=O(t)$ as a formal power series in $t$ (no constant term in $t$).

\textbf{Step 2.}
Isolate the factor $1- t$ in the denominator and expand in a (formal) geometric series in $B$:
\[
\frac{1}{1- t + B(u,t)}
=\frac{1}{1- t}\cdot
\frac{1}{1+\dfrac{B(u,t)}{1- t}}
=\frac{1}{1- t}\sum_{r\ge 0} \Big(-\frac{B(u,t)}{1- t}\Big)^r,
\]
the series converges as a formal power series in $t$ because $B(u,t)=O(t)$ and $(1- t)$ is an invertible unit in $K[\![t]\!]$. Hence
\[
F\!\big(u,\tfrac{t}{u}\big)
=\frac{t}{1- t}\sum_{r\ge0} (-1)^r(1- t)^{-r}\; A(u,t)\,B(u,t)^r.
\]
Extracting the coefficient $[u^m]$ termwise yields
\beq\label{eqlem1}
H_m(t)=[u^m]F\!\big(u,\tfrac{t}{u}\big)
=\frac{t}{1- t}\sum_{r\ge0} (-1)^r(1- t)^{-r}\; [u^m]\big(A(u,t)\,B(u,t)^r\big).
\eeq
\textbf{Step 3.}
For fixed $r$, the product $A(u,t)\,B(u,t)^r$ is a finite sum of monomials $c\, u^{\alpha} t^{\beta}$, hence $[u^m]\big(A B^r\big)$ is an element of $K[t]$. Therefore each summand in \eqref{eqlem1} is of the form
\[
(1- t)^{-1-r}\cdot(\text{polynomial in }t),
\]
so it is a rational function in $t$ whose denominator is a power of $(1- t)$.

\textbf{Step 4.}
It remains to show the sum in \eqref{eqlem1} is finite, i.e. that for all sufficiently large $r$ the coefficient $[u^m]\big(A B^r\big)$ vanishes. 
Write
\[
A(u,t)=\sum_{(i,j)\in I} a_{i,j}\, u^{i-j} t^j,\quad
B(u,t)=\sum_{(p,q)\in J} b_{p,q}\, u^{p-q} t^q,
\]
where $I,J$ are finite index sets. Set
\[
\alpha_{\min}:=\min_{(i,j)\in I} (i-j),\quad
\alpha_{\max}:=\max_{(i,j)\in I} (i-j),
\]
\[
\beta_{\min}:=\min_{(p,q)\in J} (p-q),\quad
\beta_{\max}:=\max_{(p,q)\in J} (p-q).
\]
Monomials in $A\,B^r$ have $u$-exponent in 
$[\alpha_{\min}+r\beta_{\min},\,\alpha_{\max}+r\beta_{\max}]$.
Thus $[u^m](A\,B^r)\neq0$ iff
\beq\label{eqlem2}
\alpha_{\min}+ r\,\beta_{\min} \le m \le \alpha_{\max}+ r\,\beta_{\max}. 
\eeq
For fixed $\alpha_{\min/\max},\beta_{\min/\max},m$ this admits only finitely many $r\ge0$, so there is $r_{\max}$ with $[u^m](A\,B^r)=0$ for all $r>r_{\max}$.

\textbf{Step 5.}
The sum in \eqref{eqlem1} is finite, so $H_m(t)$ is a rational function with denominator a power of $(1-t)$:
\[
H_m(t)=\frac{R_m(t)}{(1-t)^s},\quad s\le 1+r_{\max}.
\]
\endproof
\begin{rem}
The maximal $r_{\max}$ follows from \eqref{eqlem2}:
\begin{itemize}
\item if $\beta_{\min}>0$: $r \le U_1 := \lfloor\frac{m-\alpha_{\min}}{\beta_{\min}}\rfloor$.
\item if $\beta_{\max}>0$: $r \ge L_2 := \lceil\frac{m-\alpha_{\max}}{\beta_{\max}}\rceil$.
\item if $\beta_{\min}<0$: $r \ge L_1 := \lceil\frac{m-\alpha_{\min}}{\beta_{\min}}\rceil$.
\item if $\beta_{\max}<0$: $r \le U_2 := \lfloor\frac{m-\alpha_{\max}}{\beta_{\max}}\rfloor$.
\end{itemize}
Hence
\[
r_{\max}=\max\{\, r\in \mathbb{Z}_{\ge0} : r\ge L:=\max(0,L_1,L_2),\ \ r\le U:=\min(U_1,U_2)\,\},
\]
if $[L,U]\neq\emptyset$, otherwise no $r$ contributes.\qrem
\end{rem}

\begin{lem}\label{lemtchn2}
Let \(s\in\mathbb{N}_{>0}\) and let \(R_m(t)=\sum_{j=0}^{m} p_j t^j\) be a polynomial. Suppose
\[
F(t)=\sum_{n\ge 0} a_n t^n=\frac{R_m(t)}{(1-t)^s}.
\]
Then for every \(n\ge 0\) the coefficients are given by
\beq\label{eq:coeff_min}
{\;a_n=\sum_{j=0}^{\min(m,n)} p_j \binom{n-j+s-1}{s-1}\;},
\eeq
where \(\binom{t}{r}\) is the usual binomial coefficient for integers \(t\ge r\ge 0\).
In particular, for \(n\ge m\) one has
\[
a_n = Q(n)
\]
with
\beq\label{eq:Q}
Q(n) := \sum_{j=0}^m p_j \binom{n-j+s-1}{s-1},
\eeq
and \(Q(n)\) is a polynomial in \(n\) of degree at most \(s-1\). Therefore, the sequence \((a_n)_{n\ge 0}\) is \emph{eventually polynomial}.
\end{lem}

\begin{proof}
From the standard binomial expansion
\[
\frac{1}{(1-t)^s} = \sum_{r\ge 0} \binom{r+s-1}{s-1} t^r,
\]
we have
\[
\frac{t^j}{(1-t)^s} = \sum_{r\ge 0} \binom{r+s-1}{s-1} t^{r+j}
= \sum_{n\ge j} \binom{n-j+s-1}{s-1} t^n.
\]
Multiplying by \(R_m(t)\) and collecting coefficients gives
\[
a_n = \sum_{j=0}^{\min(m,n)} p_j \binom{n-j+s-1}{s-1}.
\]
For \(n\ge m\) the truncation disappears and the formula becomes
\[
a_n = \sum_{j=0}^{m} p_j \binom{n-j+s-1}{s-1}.
\]
Each binomial term is a polynomial in \(n\) of degree \(s-1\), hence \(a_n\) agrees with a fixed polynomial \(Q(n)\) of degree at most \(s-1\) for all \(n\ge m\).
\end{proof}

\begin{rem}
The equality $a_n = Q(n)$ may hold for some $n < m$ depending on the specific coefficients of $R_m$. The threshold $n \ge m$ is a general guarantee, but the polynomiality can start earlier.\qrem
\end{rem}

\begin{example}
Let
\[
F(t) = \frac{t - t^3 + t^5}{(1-t)^4}, \quad s=4,\quad m=5,\quad p_1=1,\ p_3=-1,\ p_5=1.
\]
From \eqref{eq:coeff_min}:
\[
a_n = \binom{n+2}{3} - \binom{n}{3} + \binom{n-2}{3}
\]
(combinatorial convention). This gives:
\[
\begin{array}{c|ccccccc}
n & 0 & 1 & 2 & 3 & 4 & 5 & 6 \\
\hline
a_n & 0 & 1 & 4 & 9 & 16 & 26 & 40
\end{array}
\]
From \eqref{eq:Q} (polynomial convention):
\[
Q(n) = \binom{n+2}{3} - \binom{n}{3} + \binom{n-2}{3} 
= \frac{n^3 - 3n^2 + 26n - 24}{6}.
\]
Here $Q(n) = a_n$ already for $n\ge 2$, earlier than the general bound $n\ge m=5$.\qetr
\end{example}

\proof[Proof of Theorem \ref{thmpolytlcyr}]
Consider the generating function
\[\eu L^{[\sf m]}(z,t)=\sum_{N=1}^\infty\sum_{n=N}^\infty l^{[\sf m]}_{n,N}z^nt^N,\qquad l^{[\sf m]}_{n,N}=\sum_{\substack{\bm\la\in\Z^N_{>0}\\\la_1+\dots+\la_N=n}}\prod_{i=1}^{N-1}M^{[\sf m]}_{\la_i,\la_{i+1}}.
\]Given $k\geq 0$, we have
\[[u^k]\eu L^{[\sf m]}(u,t/u)=\sum_{N=1}^\infty l^{[{\sf m}]}_{N+k,N}t^N,
\]and for ${\sf m}\geq k+1$, we have
\[[u^k]\eu L^{[\sf m]}(u,t/u)=\sum_{N=1}^\infty l^{[{\sf m}]}_{N+k,N}t^N=\sum_{N=1}^\infty\tlcyr(N+k,N)t^N,
\]by Corollary \ref{corltildegraph}. Moreover, by Theorem \ref{thmratLm}, $\eu L^{[\sf m]}(z,t)$ is a rational function of the form
\[\eu L^{[\sf m]}(z,t)=\frac{P(z,t)}{1+\sum_{i=1}^{\sf m}a_i(z)t^{{\sf m}-i}},
\]where $P(z,t)=ztQ(z,t)$ (because, by definition, $\eu L^{[\sf m]}(z,t)$ is a multiple of $zt$), and $a_{{\sf m}-i}(z)=(-1)^i{\rm Tr}\bigwedge^iT^{[\sf m]}(z)$ are the coefficients of the characteristic polynomial of the transfer matrix $T^{[\sf m]}(z)$. In particular, we have $a_{{\sf m}-1}(z)=-z$, by Lemma \ref{symGmm}.

Hence, Lemma \ref{lemtchn1} and Lemma \ref{lemtchn2} apply, and the claim follows. 
\endproof

\begin{example}
For ${\sf m}=4$, we have
\[\eu L^{[4]}(z,t)=\frac{-t^3 z^8-t^3 z^6-2 t^2 z^7-3 t^2 z^5-t^2 z^4-t^2 z^3-t z^4-t z^3-t z^2-t z}{t^3 z^8+t^3 z^6+t^2 z^7+2 t^2 z^5+t^2 z^4+t^2 z^3+t z-1},
\]so that
\[u\eu L^{[4]}(u,t/u)=\frac{t}{1-t}+\frac{t}{(1-t)^2}u+\frac{t-t^2+t^3}{(1-t)^3}u^2+\frac{t-t^3+t^5}{(1-t)^4}u^3+O(u^4).
\]We deduce
\[\sum_{N=1}^\infty\tlcyr(N+k,N)t^N=\begin{cases}
\frac{t}{1-t},\quad k=0,\\
\\
\frac{t}{(1-t)^2},\quad k=1,\\
\\
\frac{t-t^2+t^3}{(1-t)^3},\quad k=2,\\
\\
\frac{t-t^3+t^5}{(1-t)^4},\quad k=3,
\end{cases}
\]
together with the identities $\tlcyr(N+k,N)=\Tilde P_k(N)$, $k=0,1,2,3$, where 
\begin{align*}
\Tilde P_0(n)&=1,\\
\Tilde P_1(n)&=n,\\
\Tilde P_2(n)&=\binom{n+1}{2}-\binom{n}{2}+\binom{n-1}{2}=\frac{n^2-n+2}{2},\\
\Tilde P_3(n)&=\binom{n+2}{3} - \binom{n}{3} + \binom{n-2}{3} 
= \frac{n^3 - 3n^2 + 26n - 24}{6}.
\end{align*}
In particular, this confirms the conjectural identities \eqref{idconj1}, \eqref{idconj2}.\qetr
\end{example}

As a corollary of Theorem \ref{thmpolytlcyr}, we deduce a very restricted Pascal rule for $\tlcyr(n,N)$.

\begin{cor}\label{weakPascaltlcyr}
For $n\geq 2$, we have
\[\tlcyr(n,n-1)+\tlcyr(n+1,n-1)=\tlcyr(n+2,n).
\]
\end{cor}
\proof
We have the identity
\[n-1+\frac{(n-1)^2-(n-1)+2}{2}=\frac{n^2-n+2}{2}.\qedhere
\]

\section{Prime flag varieties, the double sequence $\ell(n,N)$, and sums of primes}\label{sec6}

\subsection{Flag varieties of prime type} A partial flag variety $F_{\bm\la}$, with $\bm\la=(\la_1,\dots,\la_N)\in\Z^N_{>0}$, is said to be {\it prime}, or of {\it prime type}, if and only if 
\[\la_i+\la_{i+1}\text{ is a prime number for all }i=1,\dots, N-1.
\]

\begin{example}\label{exampriflag1}
The Grassmannian $G(k,n)$ is prime if and only if $n$ is a prime number. Any complete flag variety $F_{(1,1,\dots,1)}$ is prime. \qetr
\end{example}

\begin{prop}\label{primeflagprop}
If $F_{\bm\la}$ is prime, then any of its semiclassical spectra is non-exceeding.
\end{prop}
\proof
The inequalities \eqref{eq:allnotexc} are satisfied for any $i=1,\dots, N-1$.
\endproof

\subsection{The double sequence $\ell(n,N)$} For any $2\leq N\leq n$, denote by $\ell(n,N)$ the number of prime flag varieties $F_{\bm\la}$, with $\bm\la\in\Z^N_{>0}$ and $|\bm\la|=n$.

It is easy to prove that
\[ \ell(N,N)=1,\qquad \ell(N+1,N)=N,\qquad \ell(n,N)\leq \tlcyr(n,N),
\]
where the last inequality follows from Proposition \ref{primeflagprop}.

The first values of the double sequence $\ell(n,N)$, for $2\leq N\leq n\leq 18$, are listed in the following table:

\begin{center}
\small
\begin{tabular}{c|*{17}{c}}
$N \backslash n$ & 2 & 3 & 4 & 5 & 6 & 7 & 8 & 9 & 10 & 11 & 12 & 13 & 14 & 15 & 16 & 17 & 18 \\
\hline
2  & 1 & 2 & 0 & 4 & 0 & 6 & 0 & 0 & 0  & 10  & 0   & 12  & 0   & 0   & 0   & 16  & 0 \\
3  &   & 1 & 3 & 1 & 5 & 3 & 8 & 6 & 3  & 3  & 10  & 8  & 16  & 14  & 10  & 10  & 19 \\
4  &   &   & 1 & 4 & 3 & 6 & 10 & 8 & 23 & {0}  & 22  & 10 & 33  & 12  & 56  & {0}   & 84 \\
5  &   &   &   & 1 & 5 & 6 & 8 & 19 & 14 & 42 & 15  & 40 & 33  & 64  & 44  & 100 & 48 \\
6  &   &   &   &   & 1 & 6 & 10 & 12 & 30 & 30 & 63  & 62 & 54  & 116 & 84  & 172 & 132 \\
7  &   &   &   &   &   & 1 & 7 & 15 & 19 & 44 & 59  & 94 & 144 & 99  & 249 & 177 & 373 \\
8  &   &   &   &   &   &   & 1 & 8 & 21 & 30 & 63  & 104 & 148 & 266 & 225 & 432 & 465 \\
9  &   &   &   &   &   &   &   & 1 & 9 & 28 & 46  & 90 & 169 & 242 & 443 & 488 & 719 \\
10 &   &   &   &   &   &   &   &   & 1 & 10 & 36 & 68 & 129 & 260 & 397 & 706 & 953 \\
11 &   &   &   &   &   &   &   &   &   & 1 & 11 & 45 & 97 & 185 & 386 & 639 & 1107 \\
12 &   &   &   &   &   &   &   &   &   &   & 1 & 12 & 55 & 134 & 264 & 560 & 1001 \\
13 &   &   &   &   &   &   &   &   &   &   &   & 1 & 13 & 66 & 180 & 373 & 800 \\
14 &   &   &   &   &   &   &   &   &   &   &   &   & 1 & 14 & 78 & 236 & 520 \\
15 &   &   &   &   &   &   &   &   &   &   &   &   &   & 1 & 15 & 91 & 303 \\
16 &   &   &   &   &   &   &   &   &   &   &   &   &   &   & 1 & 16 & 105 \\
17 &   &   &   &   &   &   &   &   &   &   &   &   &   &   &   & 1 & 17 \\
18 &   &   &   &   &   &   &   &   &   &   &   &   &   &   &   &   & 1 \\
\end{tabular}
\end{center}

We extend the definition of $\ell(n,N)$ to $N=1$, by setting
\[\ell(n,1)=1,\quad n\geq 1.
\]

\subsection{Vanishing of $\ell(n,N)$, sums of primes, Goldbach conjecture}
The existence of prime flag varieties depends on the values of $(n,N)$: remarkably, for some values of $(n,N)$, the sequence $\ell(n,N)$ may vanish. For example, it can be easily seen that 
\[\ell(n,2)=0\text{ whenever $n$ is not prime}, 
\]see Example \ref{exampriflag1} and Theorem \ref{thmsum1} below. Other instances (found by direct check) are 
\[\ell(11,4)=0,\quad \ell(17,4)=0,\quad \ell(23,4)=0,\quad \ell(29,4)=0,\quad \ell(35,4)=0.
\]
We propose now to investigate the nature of this phenomenon. It turns out that this is strictly tied with delicate open problems on additive number theory, namely {\it Goldbach--type problems}. These focus on the study of sumsets $A+A$ or $A+A+A$ for special sets $A\subset\N$, e.g.\,\,the set of primes. See \cite{Hua65,TV10} for beautiful introductions to the topic. The prototypical example of such problems is the \emph{strong (or binary) Goldbach conjecture} (formulated in 1742), which asserts that every even integer greater than $2$ is the sum of two primes; it remains open.

Denote 
\begin{itemize}
\item by $\eu P$ the set of prime numbers,
\item by $\eu P_k$ the set of sums of $k\geq 1$ prime numbers,
\item by $\Om:=\complement_{\N} \eu P_2$ the set of natural numbers not expressible as a sum of two primes, that is
\begin{multline*}\Om=\{n\in\N\colon n\neq p+q,\text{ for all primes }p,q\}\\=\{0,1,2,3,11,17,23,27,29,35,37,41,47,51,53,57,59,65,67,71,\dots\},
\end{multline*}
\item by $\eu T$ the set of smaller members of twin prime pairs, that it
\[\eu T=\{p\in\eu P\colon p+2\in\eu P\}.
\]
\end{itemize}

\begin{prop}
The following hold.
\begin{enumerate}
  %\item $\{0,1,2,3\}\subset S$.
  \item For every odd $n\ge 3$,
  \[
    n\in \Om \iff n-2\text{ is composite}.
  \]
  In particular $\eu T\subset \Om$.
  \item If the (strong) Goldbach conjecture holds, then $\Om\cap 2\N=\{0,2\}$.
  \item $\Om$ is infinite.
\end{enumerate}
\end{prop}

\begin{proof}
(1) Let $n\ge3$ be odd. If $n-2\in\eu P$ then $n=2+(n-2)$ is a sum of two primes, hence $n\notin \Om$. Conversely, if $n=p+q$ with $p,q$ primes then one of $p,q$ must be $2$ (the sum of two odd primes is even), hence $\{p,q\}=\{2,n-2\}$ and thus $n-2\in\eu P$. This proves the stated equivalence. %The ``in particular'' clause follows since an odd composite $n$ has $n-2$ composite, hence $n\in \Om$.

(2) The strong (binary) Goldbach conjecture asserts that every even integer $>2$ is a sum of two primes; under this assertion every even $n>2$ is not in $\Om$, so the only even elements of $\Om$ are $0$ and $2$.

(3) Since there are infinitely many odd composite integers, (1) implies $\Om$ is infinite. This completes the proof.
\end{proof}

The only nontrivial logical dependence above is the conditional statement in (2) which ties the structure of $\Om$ on the even integers to the (open) strong Goldbach conjecture.
The weak (ternary) Goldbach conjecture — namely that every odd integer greater than $5$ is the sum of three primes — has been proven by H.A.\,Helfgott in \cite{Hel13,Hel19}. 

\begin{thm}[Goldbach's week conjecture, Helfgott's Theorem, \cite{Hel13,Hel19}]\label{wGold}
Every odd integer greater than 5 can be written as the sum of three primes. Every odd integer greater than 7 can be written as the sum of three odd primes.\qed
\end{thm}

\begin{rem}
G.H.\,Hardy and J.E.\,Littlewood \cite{HL23} proved, under the Generalized Riemann Hypothesis (GRH), that the ternary conjecture holds for all sufficiently large odd integers. I.M.\,Vinogradov \cite{Vin37} established this unconditionally using the circle method, with later refinements and explicit (though enormous) bounds due to K.G.\,Borozdkin \cite{Bor56}. In 1997, J.-M.\,Deshouillers, G.\,Effinger, H.\,te Riele, and D.\,Zinoviev showed that GRH implies the ternary conjecture for all odd integers, subject to verification\footnote{Without GRH, the corresponding verification bound was known by 2002 to increase to $10^{1347}$.} up to $10^{20}$ \cite{DEtRZ97}. 
Between 2012 and 2013, in \cite{Hel12a,Hel12b,Hel13} H.A.\,Helfgott obtained sufficiently strong estimates on both major and minor arcs to prove the ternary Goldbach conjecture \emph{unconditionally}, relying in part on D.J.\,Platt’s verification of the Riemann Hypothesis for Dirichlet $L$-functions for moduli $q \leq 400{,}000$ up to height about $10^{8}/q$ \cite{Pla16}. See the book \cite{Hel19} for more details. \qrem
\end{rem}

\begin{thm}\label{thmsum1}Assume that one of the following holds:
\begin{enumerate}
\item $N=2$ and $n\notin \eu P$;
\item $N=4$ and $n\in\Om$.
\end{enumerate}
Then $\ell(n,N)=0$.
\end{thm}
\proof
If $N=2$, the number $\ell(n,N)$ equals the number of compositions $\la_1+\la_2=n$ with $\la_1+\la_2\in\eu P$. Hence $\ell(n,2)\neq 0$ iff $n\in\eu P$. In this case, one has $\ell(n,2)=n-1$.

If $N=4$, we have to count the compositions $\la_1+\la_2+\la_3+\la_4=n$ such that $\la_1+\la_2,\la_2+\la_3,\la_3+\la_4\in\eu P$. In particular, if one of such compositions exists, then $n$ can be expressed as a sum of two primes.
\endproof

Experiments suggest the validity of the inverse implication.

\begin{conj}\label{conjecture}
If $\ell(n,N)=0$, then either $N=2$ and $n$ is composite, or $N=4$ and $n\in\Om$.
\end{conj}

In what follows, we prove the following theorem supporting Conjecture \ref{conjecture}.

\begin{thm}\label{thm246}
If $N\neq 2,4,6$, then $\ell(n,N)\neq 0$ for any $n\geq N$.
\end{thm}

We need some preliminary results.

\begin{prop}\label{propbertr}
For every integer $n\ge 3$ there exist primes $p,q\le n$ such that $p+q>n$.
\end{prop}

\begin{proof}Bertrand's postulate states that for every integer $m\ge 2$ there exists a prime $r$ with
$m < r < 2m$.
Take $n\ge 3$ and apply Bertrand's postulate with $m=\frac{n}{2}$ (if $n$ is odd take $m=\frac{n-1}{2}$). There is a prime $p$ with $n/2<p\le n$. Then $p\le n$ and taking $q=p$ we get
$p+q=2p>n$,
as required.
\end{proof}

\begin{cor}\label{corell3}
For $n\geq 3$, we have $\ell(n,3)\neq 0$.
\end{cor}
\proof
By Proposition \ref{propbertr}, there exist $p,q\leq n$ primes with $p+q>n$. 
Set $\la_2=p+q-n$, $\la_1=p-\la_2$, and $\la_3=q-\la_2$. We have $\la_2<p,q$: for example, by absurd, if $\la_2\geq p$, then $p+q=n+\la_2\geq n+p$, so that $q\geq n$. Hence, we have
\[\la_1,\la_2,\la_3>0,\qquad \la_1+\la_2+\la_3=n,\qquad \la_1+\la_2,\la_2+\la_3\in\eu P.
\]This proves that $\ell(n,3)\neq 0$.
\endproof

\begin{thm}\label{thmell5}
Every integer $n\ge 5$ admits a composition into five positive parts with 
all four adjacent sums prime. That is  $\ell(n,5)\neq 0$ for $n\geq 5$. 
\end{thm}

\begin{proof}We have to show the existence of positive integers 
$\la_1,\la_2,\la_3,\la_4,\la_5$ such that
$\la_1+\la_2+\la_3+\la_4+\la_5=n$
and each of the four sums
$\la_1+\la_2,\, \la_2+\la_3,\, \la_3+\la_4,\, \la_4+\la_5$
is prime.

\emph{Small cases.}  
For $n=5$ take $(1,1,1,1,1)$, giving sums $(2,2,2,2)$.  
For $n=6$ take $(1,1,1,1,2)$, giving $(2,2,2,3)$.  
For $n=7$ take $(1,1,2,1,2)$, giving $(2,3,3,3)$.  
Thus the statement holds for $5\le n\le 7$.

\medskip
\emph{General case $n\ge 8$.}  
We use Helfgott’s Theorem \ref{wGold}.

\smallskip
\textbf{Case 1: $n$ even.}  
Then $n+1$ is odd $\ge 9$, so there exist primes $p_1,p_2,p_3$ with
\[
p_1+p_2+p_3=n+1.
\]
Define
\[
\la_1=p_1-1,\quad \la_2=1,\quad \la_3=p_2-1,\quad \la_4=1,\quad \la_5=p_3-1.
\]
Since each $p_i\ge 2$, all $\la_i\ge 1$. Their sum is
$\sum_{i=1}^5\la_i=(p_1+p_2+p_3)-3+2=n$. 
The adjacent sums are
\[
\la_1+\la_2=p_1,\quad \la_2+\la_3=p_2,\quad \la_3+\la_4=p_2,\quad \la_4+\la_5=p_3,
\]
all prime.

\smallskip
\textbf{Case 2: $n$ odd.}  
Then $n+2$ is odd $\ge 10$, so there exist odd primes $p_1,p_2,p_3$ with
\[
p_1+p_2+p_3=n+2.
\]
Define
\[
\la_1=p_1-2,\quad \la_2=2,\quad \la_3=p_2-2,\quad \la_4=2,\quad \la_5=p_3-2.
\]
Since $p_i\geq 3$, then each $\la_i\ge 1$. Their sum is 
$\sum_{i=1}^5\la_i=(p_1+p_2+p_3)-6+4=n$, and the adjacent sums are
\[
\la_1+\la_2=p_1,\quad \la_2+\la_3=p_2,\quad \la_3+\la_4=p_2,\quad \la_4+\la_5=p_3.
\]
all prime.
\end{proof}

\begin{thm}\label{thmell7+.1}The following holds.
\begin{enumerate}
\item[(a)] If $N$ is even, and $n\in\eu P_{N/2}$, then $\ell(n,N)\neq 0$.
\item[(b)] If $N$ is odd, and $n\in\eu P_{(N+1)/2}$, then $\ell(n,N)\neq 0$.
\end{enumerate}
\end{thm}
\proof
\medskip\noindent\emph{Case (a): $N=2r$.}  
Suppose $q_1+\dots+q_r=n$ with all $q_i$ prime. Define
\[
\la_{2i-1} = q_i-1,\qquad \la_{2i}=1 \qquad (1\leq i \leq r).
\]
Then each $\la_j$ is positive. Moreover,
\[
\la_{2i-1}+\la_{2i} = (q_i-1)+1=q_i,\quad \la_{2i}+\la_{2i+1} = 1+(q_{i+1}-1)=q_{i+1},
\]
so all adjacent sums are prime. Finally,
\[
\sum_{j=1}^{2r}\la_j = \sum_{i=1}^r (q_i-1+1) = \sum_{i=1}^r q_i = n,
\]
as required.

\medskip\noindent\emph{Case (b): $N=2r+1$.}  
Suppose $q_1+\dots+q_{r+1}=n+1$ with all $q_i$ prime. Define
\[
\la_{2i-1}=q_i-1,\qquad \la_{2i}=1 \qquad (1\leq i \leq r),\qquad \la_{2r+1}=q_{r+1}-1.
\]
Again all $\la_j$ are positive. We compute
\[
\la_{2i-1}+\la_{2i} = q_i,\quad \la_{2i}+\la_{2i+1}=q_{i+1} \quad (1\leq i\leq r),
\]
so the adjacent sums are precisely $q_1,\dots,q_{r+1}$. Moreover
\[
\sum_{j=1}^{2r+1} \la_j = \Big(\sum_{i=1}^{r+1} q_i\Big) - 1 = (n+1)-1 = n.
\]
This completes the proof.
\endproof

\begin{thm}\label{thmell7+.2}
Let $r \geq 4$ be an integer. Then there exists a constant $n_o(r)$ such that every integer $n \geq n_o(r)$ can be written as a sum of $r$ prime numbers. In fact, one may take $n_o(r)=2r+2$ for every $r \geq 4$.
\end{thm}

\begin{proof}
For $r=4$, let $n$ be an even integer with $n \geq 10$. Then $n-3$ is odd and $\geq 7$, hence $n-3 = p_1+p_2+p_3$ with primes $p_i$. Thus
\[
n = (n-3)+3 =3+ p_1+p_2+p_3
\]
is a sum of four primes. Similarly, let $n\geq 9$ be odd. Then $n-2$ is odd and $\geq 7$, hence $n-2=p_1+p_2+p_3$ with primes $p_i$. Thus 
\[n=2+p_1+p_2+p_3
\]is a sum of four primes.

Now let $r \geq 5$. We distinguish two cases.

\smallskip
\emph{Case 1: $n$ odd.}  
Set $k=r-3$ and $m=n-2k$. Since $n$ is odd, also $m$ is odd. If $n \geq 2r+1$ then $m \geq 7$, so $m$ is the sum of three primes, say $m=p_1+p_2+p_3$. Hence
\[
n = m + 2k = p_1+p_2+p_3 + \underbrace{2+\cdots+2}_{k \text{ times}},
\]
a sum of $3+k=r$ primes.

\smallskip
\emph{Case 2: $n$ even.}  
Set $k=r-4$ and $m=n-2k$. Then $m$ is even. If $n \geq 2r+2$ then $m \geq 10$, hence by the $r=4$ case we can write $m=q_1+q_2+q_3+q_4$ with primes $q_i$. Thus
\[
n = m + 2k = q_1+q_2+q_3+q_4 + \underbrace{2+\cdots+2}_{k \text{ times}},
\]
a sum of $4+k=r$ primes.

\smallskip
Therefore every sufficiently large integer $n$ is a sum of $r$ primes for each $r \geq 4$. Explicitly, for $r \geq 4$ the argument above shows that $n_o(r)=2r+2$ works.
\end{proof}

\proof[Proof of Theorem \ref{thm246}]  If $N=3,5$, then the claim follows from Corollary \ref{corell3} and Theorem \ref{thmell5}.

For $N\geq 7$, we have $\ell(n,N)\neq 0$ for $n\geq n_o(N)$ by Theorem \ref{thmell7+.1} and Theorem \ref{thmell7+.2}. In particular, one can take $n_0(N)=N+3$ if $N$ is odd and $n_o(N)=N+2$ is $N$ is even.

Since $\ell(N,N)=1$ and $\ell(N+1,N)=N$ for any $N$, we only need to check that $\ell(N+2,N)\neq 0$ for any $N\geq 7$. But this is easily established, since $N+2$ always admit a composition into $N$ positive parts whose adjacent sums are all prime: for example: $\bm\la=(2,1,1,\dots,1,2)$. 
\endproof

The case $N=6$ turns out to be more delicate.

\begin{thm}\label{thmequivgoldconj}
The following are equivalent.
\begin{enumerate}
\item We have $\ell(n,6)\neq 0$, for any $n\geq 6$.
\item Goldbach conjecture.
\end{enumerate}
\end{thm}
\proof
Assume (1). Let $n\geq 4$ be an even integer. Then $\ell(n+2,6)\neq 0$. Hence there exist $\bm\la=(\la_1,\dots,\la_6)$ such that
\[n+2=\underbrace{\la_1+\la_2}_{p_1}+\underbrace{\la_3+\la_4}_{p_2}+\underbrace{\la_4+\la_5}_{p_3},\qquad p_1,p_2,p_3\in\eu P.
\]Since $n+2$ is even, exactly one of $p_i$'s equals 2, say $p_1=2$. Then $n=p_2+p_3$.

Conversely, assume (2). Let $n\geq 6$ be an arbitrary integer. If $n$ is even, then $n-2\geq 4$, and by Goldbach conjecture $n=2+p+q\in \eu P_3$. Hence $\ell(n,6)\neq 0$ by Theorem \ref{thmell7+.1}. Similarly, if $n$ is odd, then $n\in\eu P_3$ by Helfgott's theorem, and $\ell(n,6)\neq 0$ by Theorem \ref{thmell7+.1}.
\endproof

\begin{cor}
Conjecture \ref{conjecture} is equivalent to Goldbach conjecture.
\end{cor}
\proof By Theorem \ref{thm246}, the condition $\ell(n,N)=0$ implies $N\in\{2,4,6\}$.
If Goldbach conjecture holds true, then necessarily either $N=2$ or $N=4$. If $N=2$, then $n$ is not a prime. If $N=4$, then $n$ cannot be a sum of two primes, by Theorem \ref{thmell7+.1}. Conversely, if Conjecture \ref{conjecture} holds true, then $\ell(n,6)$ is non-zero for any $n\geq 6$, and Goldbach conjecture holds, by Theorem \ref{thmequivgoldconj}.
\endproof

\begin{rem}
Although strong Goldbach conjecture is still open, several deep partial results are known. In \cite{Che73}, J.-R.\,Chen proved that every sufficiently large even integer is the sum of a prime and a product of two primes (a ``semiprime''). His proof has been greatly simplified by P.M.\,Ross \cite{Ros75}. More recently, extensive computations verify the conjecture for all even integers up to $4 \cdot 10^{18}$, see \cite{OSHP14}.
\qrem
\end{rem}

\subsection{Graph \textcyr{P}$_{\sf m}$, transfer matrices, generating functions}\label{secgraphPim} Let ${\sf m}\in\N_{>0}$. Introduce the graph \textcyr{P}$_{\sf m}$, with vertices $\{1,\dots,{\sf m}\}$, connected by an arrow $(i,j)$ whenever $i+j$ is prime. Denote by $P^{[\sf m]}$ its ${\sf m}\times {\sf m}$ adjacency matrix.

\begin{lem}\label{symPim}
The matrix $P^{[\sf m]}$ is symmetric, and $P^{[\sf m]}_{ii}\neq 0$ iff $i=1$.\qed
\end{lem}

Similarly to what done in Section \ref{secgraphGm}, assign the monomial weight $z^j$ to each vertex $j$, and introduce the transfer matrix $Q^{[\sf m]}(z)$, defined by
\beq
Q^{[\sf m]}(z)_{ij}=P^{[\sf m]}_{ij}z^j,\quad i,j=1,\dots,{\sf m}.
\eeq
For any walk $\si$ in \textcyr{P}$_{\sf m}$, we have the monomial ${\sf w}(\si):=\prod_{i\text{ node of }\si}z^i$, and the length $|\si|:=\#\{\text{arrows in }\si\}$.

Set $v_{\sf m}(z)=(z,z^2,\dots,z^{\sf m})^T$, and ${\bf 1}_{\sf m}=(1,\dots,1)^T$.

For each $N\geq 1$, define the generating function
\[\mathscr L^{[\sf m]}_N(z):=v_{\sf m}(z)^T\cdot Q^{[\sf m]}(z)^{N-1}\cdot {\bf 1}_{\sf m},
\]and collect all of them into a single one
\[\mathscr L^{[\sf m]}(z,t):=\sum_{N=1}^\infty\mathscr L^{[\sf m]}_N(z)t^N,\quad {\sf m}\geq 1.
\]

In what follows we relate this generating functions of suitable walks on \textcyr{P}$_{\sf m}$ with the genrating functions of the double sequence $\ell(n,N)$, namely
\beq\label{LNell}
L_N(z)=\sum_{n\geq N}\ell(n,N)z^n,\quad N\geq 1,\qquad L(z,t)=\sum_{N=1}^\infty\sum_{n\geq N}\ell(n,N)z^nt^N.
\eeq

The following results are analogs of those of Sections \ref{secgraphGm} and \ref{secratioLm}, and can be similarly proved: the proofs work verbatim, by replacing $\Gm_{\sf m}\leftrightarrow\Pi_{\sf m}$, $M^{[\sf m]}\leftrightarrow P^{[\sf m]}$, $T^{[\sf m]}\leftrightarrow Q^{[\sf m]}$, $\tlcyr(n,N)\leftrightarrow \ell(n,N)$.
\begin{prop}
The function $\mathscr L^{[\sf m]}_N(z)$ is the polynomial in $z$ given by
\[\mathscr L^{[\sf m]}_N(z)=\sum_{i_1,\dots,i_N=1}^{\sf m}z^{i_1+\dots+i_{N}}\prod_{k=1}^{N-1}P^{[\sf m]}_{i_k,i_{k+1}}=\sum_{\substack{\si\text{ in }\Pi_{\sf m}\\ |\si|=N-1}}{\sf w}(\si).\tag*{\qed}
\]
\end{prop}

\begin{prop}
Let $N \ge 1$ be fixed. Then the limit
$
\mathscr L_N(z) :=\lim_{{\sf m}\to\infty} \mathscr L^{[{\sf m}]}_N(z)
$
exists as a formal power series in $z$.\qed
\end{prop}

\begin{thm}
For any $N\geq 1$, we have $\mathscr L_N(z)=L_N(z)$.\qed
\end{thm}

\begin{cor}\label{corellPigraph}
For any ${\sf m}\geq n-N+1$, we have 
\[
\ell(n,N)=\sum_{\substack{\bm\la\in\Z^N_{>0}\\\la_1+\dots+\la_N=n}}\prod_{i=1}^{N-1}P^{[\sf m]}_{\la_i,\la_{i+1}}.\tag*{\qed}
\]
\end{cor}

\begin{cor}
For any $N,{\sf m}\geq 1$, we have 
\[L_N(z)\equiv \mathscr L^{[\sf m]}_N(z)\mod(z^{{\sf m}+N}).\tag*{\qed}
\]
\end{cor}

\begin{prop}
For any $\sf m\geq 1$, we have 
\[L(z,t)\equiv \mathscr L^{[\sf m]}(z,t)\mod I,
\]where $I$ is the ideal of $\Z[\![z,t]\!]$ generated by $(z^{{\sf m}+h}t^h)_{h\in\N_{>0}}$.\qed
\end{prop}

\begin{thm}\label{thmratmathscrLm}
For any ${\sf m}\geq 1$, we have $\mathscr L^{[\sf m]}(z,t)\in\Q(z,t)$.\qed
\end{thm}

\subsection{Eventual polynomiality of $N\mapsto \ell(N+k,N)$}
In Sections \ref{polcyrsec} and \ref{poltlcyrsec}, we proved, in two different manners, the eventual polynomiality of the sequences $N\mapsto \lcyr(N+k,N), \tlcyr(N+k,N)$ for any $k\geq 0$. For the sequence $\lcyr$, this was a consequence of its rigidity imposed by Pascal rules. For the sequence $\tlcyr$, manifesting only a very weak Pascal rule (Corollary \ref{weakPascaltlcyr}), this was a consequence of the rationality of the generating function $\eu L^{[\sf m]}(z,t)$.

It turns out that the same argument for $\tlcyr$ implies the analog result for the sequence $N\mapsto \ell(N+k,N)$, $k\geq 0$.

\begin{thm}\label{thmpolyell}
The sequence $N\mapsto\ell(N+k,N)$ is eventually polynomial for any $k\geq 0$. There exists $N_0\in\N_{>0}$ and a polynomial $\mathscr P_k\in\Q[n]$, with $\deg \mathscr P_k\leq k$, such that $\ell(N+k,N)=\mathscr P_k(N)$ for any $N\geq N_0$.
\end{thm}

\proof
Consider the generating function
\[\mathscr L^{[\sf m]}(z,t)=\sum_{N=1}^\infty\sum_{n=N}^\infty \ell^{[\sf m]}_{n,N}z^nt^N,\qquad \ell^{[\sf m]}_{n,N}=\sum_{\substack{\bm\la\in\Z^N_{>0}\\\la_1+\dots+\la_N=n}}\prod_{i=1}^{N-1}P^{[\sf m]}_{\la_i,\la_{i+1}}.
\]Given $k\geq 0$, we have
\[[u^k]\mathscr L^{[\sf m]}(u,t/u)=\sum_{N=1}^\infty \ell^{[{\sf m}]}_{N+k,N}t^N,
\]and for ${\sf m}\geq k+1$, we have
\[[u^k]\mathscr L^{[\sf m]}(u,t/u)=\sum_{N=1}^\infty \ell^{[{\sf m}]}_{N+k,N}t^N=\sum_{N=1}^\infty\ell(N+k,N)t^N,
\]by Corollary \ref{corellPigraph}. Moreover, by Theorem \ref{thmratmathscrLm}, $\mathscr L^{[\sf m]}(z,t)$ is a rational function of the form
\[\mathscr L^{[\sf m]}(z,t)=\frac{P(z,t)}{1+\sum_{i=1}^{\sf m}a_i(z)t^{{\sf m}-i}},
\]where $P(z,t)=ztQ(z,t)$ (because, by definition, $\mathscr L^{[\sf m]}(z,t)$ is a multiple of $zt$), and $a_{{\sf m}-i}(z)=(-1)^i{\rm Tr}\bigwedge^iQ^{[\sf m]}(z)$ are the coefficients of the characteristic polynomial of the transfer matrix $Q^{[\sf m]}(z)$. In particular, we have $a_{{\sf m}-1}(z)=-z$, by Lemma \ref{symPim}.

Hence, Lemma \ref{lemtchn1} and Lemma \ref{lemtchn2} apply, and the claim follows. 
\endproof

\begin{example}
For $k\leq 7$, we have
\begin{align*}
\sum_{N=1}^\infty\ell(N,N)t^N&=\frac{t}{t-1},\\
\sum_{N=1}^\infty\ell(N+1,N)t^N&=\frac{t}{(t-1)^2},\\
\sum_{N=1}^\infty\ell(N+2,N)t^N&=\frac{t \left(t^3-4 t^2+3 t-1\right)}{(t-1)^3},\\
\sum_{N=1}^\infty\ell(N+3,N)t^N&=\frac{-t^5+6 t^4-5 t^3+t}{(t-1)^4},\\
\sum_{N=1}^\infty\ell(N+4,N)t^N&=\frac{t \left(t^6-4 t^5-4 t^4+15 t^3-13 t^2+5 t-1\right)}{(t-1)^5},\\
\sum_{N=1}^\infty\ell(N+5,N)t^N&=\frac{-2 t^8+14 t^7-10 t^6-19 t^5+30 t^4-13 t^3+t}{(t-1)^6},\\
\sum_{N=1}^\infty\ell(N+6,N)t^N&=\frac{t \left(t^9-4 t^8-21 t^7+53 t^6-21 t^5-42 t^4+54 t^3-27 t^2+7 t-1\right)}{(t-1)^7},
\end{align*}
\begin{multline*}
\sum_{N=1}^\infty\ell(N+7,N)t^N=\\
\frac{t \left(-2 t^{10}+22 t^9+10 t^8-166 t^7+306 t^6-282 t^5+169 t^4-80 t^3+31 t^2-8 t+1\right)}{(t-1)^8}.
\end{multline*}
The corresponding polynomials $\mathscr P_k(n)$ are
\begin{align*}
\mathscr P_0(n)&=1,\\
\mathscr P_1(n)&=n,\\
\mathscr P_2(n)&=\frac{n^2-3n+2}{2}=\binom{n-1}{2},\\
\mathscr P_3(n)&=\frac{n^3}{6}-\frac{3 n^2}{2}+\frac{16 n}{3}-2,\\
\mathscr P_4(n)&=\frac{n^4}{24}-\frac{3 n^3}{4}+\frac{143 n^2}{24}-\frac{57 n}{4}+9,\\
\mathscr P_5(n)&=\frac{n^5}{120}-\frac{n^4}{4}+\frac{83 n^3}{24}-\frac{79 n^2}{4}+\frac{713 n}{15}-32,\\
\mathscr P_6(n)&=\frac{n^6}{720}-\frac{n^5}{16}+\frac{191 n^4}{144}-\frac{649 n^3}{48}+\frac{12541 n^2}{180}-\frac{1889 n}{12}+122,\\
\mathscr P_7(n)&=\frac{n^7}{5040}-\frac{n^6}{80}+\frac{271 n^5}{720}-\frac{95 n^4}{16}+\frac{18821 n^3}{360}-\frac{4871 n^2}{20}+\frac{38489 n}{70}-466.
\end{align*}
The identity $\ell(N+k,N)=\mathscr P_k(N)$ holds for $N\geq N_0(k)$, for the optimal values 
\[N_0(0)=N_0(1)=1,\quad N_0(2)=N_0(3)=2,\quad N_0(4)=N_0(5)=3,\quad N_0(6)=N_0(7)=4.
\]\qetr
\end{example}

\setcounter{section}{0}
\renewcommand{\thesection}{\Alph{section}}
\renewcommand{\thesubsection}{\thesection.\arabic{subsection}}
\numberwithin{equation}{section} % se vuoi che le equazioni diventino A.1, A.2, ...

\newcommand{\AppendixSection}[1]{%
  \refstepcounter{section}% <-- incrementa e crea ancora per hyperref
  \section*{Appendix \thesection. #1}% <-- titolo visibile (senza numero automatico)
  %\addcontentsline{toc}{section}{Appendix \thesection. #1}% <-- voce nel TOC
  \setcounter{subsection}{0}% reset sottosezioni per la nuova lettera
}

\AppendixSection{General facts about cohomology of bundles}\label{appA}

 Consider a holomorphic fiber bundle $\pi\colon E\to B$ with fiber $F$, where $E,B,F$ are smooth complex projective varieties, or more general compact K\"ahler manifolds.

\subsubsection{Deligne theorem}\label{Deligneapp} Given a field $k=\Q,\R,\C$, consider the sheaves $R^q\pi_*\underline{k}$, with $q\geq 0$, on $B$: these are the sheafifications of the presheaves $U\mapsto \eu H^q(\pi^{-1}(U),k)$ of $k$-vector spaces (here $U\subseteq B$ is an open set). Since the fibration $\pi$ is locally trivial, these are locally constant sheaves. 

The sheaf $R^q \pi_* \underline{k}$ forms a local system on $B$ whose fibers are the cohomology groups $H^q(F,k)$ equipped with the monodromy action $\pi_1(B) \to \mathrm{Aut}(H^\bullet(F, k))$ of the fundamental group $\pi_1(B)$. The groups $H^p(B, R^q \pi_* \underline{k})$ compute the cohomology of $B$ with these twisted coefficients, or equivalently, the group cohomology of $\pi_1(B)$ with values in $H^q(F, k)$. 

In particular, $H^0(B, R^q \pi_* \underline{k})$ consists of monodromy-invariant elements, while higher-degree groups $H^p(B, R^q \pi_* \underline{k})$ for $p > 0$ measure the obstructions to lifting local invariants to global sections.

The Serre--Leray spectral sequence associated to the fiber bundle $(E,B,F)$
with coefficients in the constant sheaf $\underline{k}$ on $E$, has second page
$E_2^{p,q} = H^p(B, R^p\pi_*\underline{k})$,
and converges to
$H^{p+q}(E; \underline{k})$.

In our situation -- total space, base and fiber being compact K\"ahler manifolds -- the Deligne theorem asserts a remarkable stronger fact.

\begin{thm}\cite{Del68}
The Serre--Leray spectral sequence degenerates in $E_2$, so that 
\[H^\bullet(E,k)\cong H^\bullet(B,R^\bullet\pi_*\underline{k}).\tag*{\qed}\]
\end{thm}

The pullback map $\pi^*\colon H^\bullet(B,k)\to H^\bullet(E,k)$ defines a $H^\bullet(B,k)$-algebra structure on $H^\bullet(E,k)$. From Deligne theorem, we deduce the following

\begin{prop}
The cohomology $H^\bullet(E,k)$ is a finitely generated $H^\bullet(B,k)$-module.
\end{prop}
\proof
Since the groups $H^p(B, R^q \pi_* \underline{k})$ are finite-dimensional vector spaces (due to the compactness of $B$ and the finiteness of the cohomology of the fibers), the total cohomology $H^*(E, k)$ decomposes additively as a finite direct sum of these spaces, hence it is finitely generated as a module over $H^*(B, k)$.
\endproof

\subsubsection{Cohomologically decomposable bundles} 
We say that the locally trivial fiber bundle $(E,B,F)$ of smooth projective varieties is {\it cohomologically decomposable} if 
\[ 
H^\bullet(E,k)\cong H^\bullet(B,k)\otimes_k H^\bullet(F,k) 
\]
as $H^\bullet(B,k)$-modules (not necessarily as rings).

\begin{thm}\label{equivs}
The following conditions are equivalent:
\begin{enumerate}
\item The fiber bundle $(E, B, F)$ is cohomologically decomposable;
\item The locally constant sheaves $R^p\pi_*\underline{k}$ are constant;
\item The monodromy action $\pi_1(B, b) \to \mathrm{Aut}(H^\bullet(F, k))$ is trivial for some (and hence any) $b \in B$;
\item The fiber is \emph{totally non-homologous to zero}, i.e., the restriction map $\iota_b^*\colon H^\bullet(E, k) \to H^\bullet(F, k)$ is surjective for any $b\in B$;
\item There exist classes $e_1, \dots, e_n \in H^\bullet(E, k)$ whose restrictions $\iota_b^*e_1, \dots, \iota_b^*e_n \in H^\bullet(F, k)$ form a basis on each fiber $\pi^{-1}(b)\cong F$. %, and such that every element of $H^\bullet(E, k)$ can be written as $\sum_{j=1}^k \pi^*b_j \cup e_j$ for some $b_1, \dots, b_k \in H^\bullet(B, k)$.
\end{enumerate}
If $B$ is simply connected, then all the conditions above are satisfied.
\end{thm}

\proof
By Deligne's theorem, the Leray spectral sequence of the fibration degenerates at $E_2$, and the equivalence of (2), (3), (4) follows, see \cite[Thm.\,14.1]{Bor67}. The equivalence $(4)\Leftrightarrow (5)$ is obvious. The implication $(5)\Rightarrow (1)$ is the Leray--Hirsch theorem: if $s\colon H^\bullet(F,k)\to H^\bullet(E,k)$ is a section of $\iota^*$, then the linear map 
\[
H^\bullet(B,k)\otimes_k H^\bullet(F,k)\to H^\bullet(E,k),\quad \alpha\otimes \beta\mapsto \pi^*\alpha\cup s(\beta)
\]
is an isomorphism of $H^\bullet(B,k)$-modules, see e.g.\ \cite{BT82}. The converse $(1)\Rightarrow (4)$ is clear. 
\endproof

\begin{rem}
The classes $e_1, \dots, e_n \in H^\bullet(E, k)$ of point (5) of Proposition \ref{equivs} are such that every element of $H^\bullet(E, k)$ can be written as $\sum_{j=1}^n \pi^*b_j \cup e_j$ for some $b_1, \dots, b_n \in H^\bullet(B, k)$.  
\qrem\end{rem}

\begin{thm}\label{piinj}\cite[Thm.\,14.2]{Bor67}
If $(E,B,F)$ is cohomologically decomposable, then $\pi^*$ is injective. Moreover, $H^\bullet(F,k)$ is isomorphic to the factor module $H^\bullet(E,k)/I$, where the ideal $I$ is generated by $\pi^*H^\bullet_+(B,k)$, with $H^\bullet_+(B,k)=\bigoplus_{i\geq 1}H^i(B,k)$.\qed
\end{thm}

\subsubsection{Integration along fibers, projection formula}\label{projApp}

Let $\pi \colon E \to B$ be a smooth proper morphism between smooth projective complex varieties, and let $k$ be a field (e.g.\ $\mathbb{Q}$, $\mathbb{R}$, or $\mathbb{C}$). The cohomology groups $H^\bullet(-, k)$ admit a pushforward (or \emph{integration along the fibers}) map
\[
\pi_*\colon H^\bullet(E, k) \to H^{\bullet - 2d}(B, k),
\]
where $d = \dim F$ is the relative dimension of the fibers. By definition of the pushforward in cohomology for proper smooth maps, we have
\[
\int_E \omega = \int_B \pi_*\omega,\quad \omega \in H^{2\dim E}(E)
\]
to which we refer to as the cohomological {\it Fubini formula}, which expresses integration over $E$ as integration over the base after fiberwise integration. The pushforward $\pi_*$ is constructed precisely so that this compatibility with integration holds.

The map $\pi_*$ satisfies the \emph{projection formula} (see e.g.\ \cite[Prop.\,6.15]{BT82}, whose proof can be easily adapted in this setting):
\[
\pi_*(\pi^*\alpha \cup \beta) = \alpha \cup \pi_*\beta \qquad \text{for all } \alpha \in H^\bullet(B, k),\; \beta \in H^\bullet(E, k).
\]
By integration over $B$, we deduce
\[
\int_E \pi^*\alpha \cup \beta = \int_B \alpha \cup \pi_*\beta, \qquad \text{for all } \alpha \in H^\bullet(B, k),\; \beta \in H^\bullet(E, k).
\]

\begin{prop}[Decomposition of integral for decomposable classes]\label{decompint}
Let $\pi: E \to B$ be a proper smooth fibration of complex manifolds with fiber $F$.  
Suppose $\alpha \in H^\bullet(E)$ can be written as
\[
\alpha = \pi^* b \cup a,
\]
where $b \in H^\bullet(B)$ and $a \in H^\bullet(E)$ is a class \emph{vertical}, i.e., whose restriction $a|_{F_b}$ to any fiber is constant (independent of $b$). Then
\[
\int_E \alpha = \left(\int_B b\right) \cdot \left(\int_F a|_F\right).
\]
\end{prop}

\begin{proof}
By the projection formula we have
\[
\pi_*(\pi^* b \cup a) = b \cup \pi_* a.
\]
Since $a$ is vertical and constant on fibers, the pushforward $\pi_* a$ is the cohomology class on $B$ given by the constant function
\[
b \mapsto \int_{F_b} a|_{F_b} = \int_F a|_F.
\]
Thus $\pi_* a = \left(\int_F a|_F\right) \cdot 1_B$, where $1_B \in H^0(B)$ is the unit class.

Then
\[
\pi_*(\alpha) = b \cup \pi_* a = \left(\int_F a|_F\right) \cdot b.
\]

Finally, integrating over $B$, by the Fubini formula,
\[
\int_E \alpha = \int_B \pi_* \alpha = \int_B \left(\int_F a|_F\right) \cdot b = \left(\int_F a|_F\right) \cdot \int_B b,
\]
as claimed.
\end{proof}

\begin{cor}\label{deccorpairing}
If the fiber bundle $(E,B,F)$ is cohomologically decomposable, then under the isomorphism
$H^\bullet(E, \C) \cong H^\bullet(B, \C) \otimes_\C H^\bullet(F, \C)$,
the Poincar\'e pairing $\eta_E$ on $H^\bullet(E)$ decomposes as the tensor product
$\eta_E = \eta_B \otimes \eta_F$,
where $\eta_B$ and $\eta_F$ are the Poincar\'e pairings on the base and fiber, respectively.
\end{cor}

\begin{proof}
Choose bases $b_i$ of $H^\bullet(B)$ and classes $e_j \in H^\bullet(E)$ restricting to a fixed basis of $H^\bullet(F)$. By cohomological triviality, the classes $e_j$ are fiberwise constant.

Using the previous proposition and the projection formula, we compute
\[
\eta_E(\pi^* b_i \cup e_j, \pi^* b_p \cup e_q) = (-1)^{\epsilon} \int_E \pi^*(b_i \cup b_p) \cup (e_j \cup e_q) = (-1)^{\epsilon} \int_B b_i \cup b_p \cdot \int_F \iota^*(e_j \cup e_q),
\]
where the sign $\epsilon$ accounts for super-commutativity.

This shows $\eta_E$ corresponds exactly to $\eta_B \otimes \eta_F$.
\end{proof}

\AppendixSection{Double Schubert polynomials}\label{appB}
\label{appSchub}
General references are \cite{LS82a,LS82b, Las12}.
\subsubsection
{Demazure difference operators.}
Let $\Delta_1, \ldots, \Delta_{n-1}$ be operators acting on functions of $\bm{x} = (x_1, \ldots, x_n)$, defined by
\[
\Delta_a f(\bm{x}) = \frac{f(\bm{x}) - f(x_1, \ldots, x_{a+1}, x_a, \ldots, x_n)}{x_a - x_{a+1}},
\quad \text{for } a = 1, \ldots, n-1.
\]
These satisfy the nil-Coxeter relations:
\[
\Delta_a^2 = 0, \qquad \Delta_a \Delta_b = \Delta_b \Delta_a, \ |a - b| > 1, \qquad 
\Delta_a \Delta_{a+1} \Delta_a = \Delta_{a+1} \Delta_a \Delta_{a+1}.
\]

Given \(\sigma \in S_n\), fix a reduced expression of \(\sigma\) as a product of adjacent transpositions:
\[
\sigma = s_{a_1} s_{a_2} \cdots s_{a_k},
\]
where each \(s_{a_i}\) denotes the transposition exchanging \(a_i\) and \(a_i + 1\), and \(k\) is the length of \(\sigma\). Then the corresponding Demazure operator is defined by
\[
\Delta_\sigma = \Delta_{a_1} \Delta_{a_2} \cdots \Delta_{a_k}.
\]
This definition is independent of the chosen reduced expression, due to the nil-Coxeter relations. 

\subsubsection{Double Schubert polynomials}

The type~\(A\) double Schubert polynomials \(\mathfrak{S}_\sigma\), for \(\sigma \in S_n\), are defined inductively from the longest permutation \(\sigma_0(i) = n + 1 - i\) by
\[
\mathfrak{S}_{\sigma_0}(\bm{x}; \bm{y}) = \prod_{i=1}^{n-1} \prod_{j=1}^{n-i} (x_i - y_j),
\]
and for arbitrary \(\sigma \in S_n\),
\[
\mathfrak{S}_\sigma(\bm{x}; \bm{y}) = \Delta_{\sigma^{-1} \sigma_0} \mathfrak{S}_{\sigma_0}(\bm{x}; \bm{y}),
\]
where \(\Delta_w\) denotes the Demazure operator associated with the permutation \(w\).

The type\,$A$ Schubert polynomials \(\mathfrak{S}_\sigma(\bm x)\), for \(\sigma \in S_n\) depend on the single tuple of variables $\bm x$. They are obtained from the double Schubert polynomials by setting $\mathfrak{S}_\sigma(\bm x)=\mathfrak{S}_{\sigma_0}(\bm{x}; 0)$

\AppendixSection{Weight spaces, dynamical operators, stable envelopes} \label{appdynop}
\subsubsection{Weight spaces, dynamical operators}Let $N, n \in \Z_{>0}$, and let $\bm\lambda = (\lambda_1, \dots, \lambda_N) \in \Z_{\ge 0}^N$ be a composition of $n$, i.e., $|\bm\lambda| := \lambda_1 + \cdots + \lambda_N = n$.

Define $\mc I_{\bm\lambda}$ as the set of all ordered partitions $I = (I_1, \dots, I_N)$ of the set $\{1, \dots, n\}$ into disjoint subsets satisfying $|I_j| = \lambda_j$ for each $j = 1, \dots, N$.

\vskip1mm

Consider the complex vector space $\C^N$ with standard basis vectors $\{v_1, \dots, v_N\}$, where $v_i = (0, \dots, 0, 1_i, 0, \dots, 0)$, for $i = 1, \dots, N$.  
The $n$-fold tensor product $(\C^N)^{\otimes n}$ has a basis indexed by elements $I \in \mc I_{\bm\lambda}$, defined as
\[
v_I := v_{i_1} \otimes \cdots \otimes v_{i_n},
\]
where, for each $j = 1, \dots, n$, the index $i_j = m$ if and only if $j \in I_m$.

\vskip1mm

The space $(\C^N)^{\otimes n}$ carries a natural action of the Lie algebra $\mathfrak{gl}_N$, whose standard basis is given by the elementary matrices $e_{ij}$ for $i, j = 1, \dots, N$. We will denote by $e_{ij}^{(a)}$, $a=1,\dots,n$, the operator induced by $e_{ij}$ acting on the $a$-th copy of $\C^N$ in the tensor product.

This $\mathfrak{gl}_N$-module admits a weight decomposition:
\[
(\C^N)^{\otimes n} = \bigoplus_{|\bm\lambda| = n} (\C^N)^{\otimes n}_{\bm\lambda},
\]
where the weight space $(\C^N)^{\otimes n}_{\bm\lambda}$ is spanned by the vectors $\{v_I \mid I \in \mc I_{\bm\lambda}\}$.
\vskip2mm

Following \cite{TV23}, we introduce the following definitions.
\begin{defn}Given $I\in\mc I_{\bm\la}$, let $a\in I_i$ and $b\in I_j$. We say that $(a,b)$ is 
\begin{itemize}
\item $I$-disordered, if either $a<b$, $i>j$, or $a>b$, $i<j$;
\item $I$-ordered, if either $a<b$, $i<j$, or $a>b$, $i>j$.
\end{itemize}Set $I_{[a,b]}=\bigcup_{r=\min\{i,j\}}^{\max\{i,j\}}I_r$. We say that $(a,b)$ is $I$-{\it admissible} if one of the following holds:
\begin{itemize}
\item $(a,b)$ is $I$-disordered and 
\[\left\{\min\{a,b\}+1,\dots,\max\{a,b\}-1\right\}\cap I_{[a,b]}=\emptyset
\]
\item $(a,b)$ is $I$-ordered and 
\[\left\{1,\dots,\min\{a,b\}-1,\max\{a,b\}+1,\dots,n\right\}\cap I_{[a,b]}=\emptyset.
\]
\end{itemize}
\end{defn}
Given $\si\in S_n$ and $I=(I_1,\dots, I_N)\in\mc I_{\bm\la}$, define $\si(I)=(\si(I_1),\dots, \si(I_N))$. Given $a,b\in\{1,\dots, n\}$, denote by $s_{a,b}$ the transposition of $a,b$.
\begin{defn}
For any $i,j=1,\dots, N$, and $a,b=1,\dots,n$, define the linear operators $Q^{a,b}_{i,j}$ acting on $(\C^N)^{\otimes n}$ by
\begin{align*}
&Q^{a,b}_{i,j}v_I=v_{s_{a,b}(I)}, && \text{if $a\in I_i$, $b\in I_j$, and the pair $(a,b)$ is $I$-admissible,}\\
&Q^{a,b}_{i,j}v_I=0, && \text{otherwise.}
\end{align*}
\end{defn}

Introduce parameters $\bm z=(z_1,\dots, z_n)$ and $\bm p=(p_1,\dots, p_N)$.
\begin{defn} Define the dynamical operators $X_1,\dots, X_N$ acting on $(\C^N)^{\otimes n}$, by the formula
\begin{multline}
X_i(\bm z;\bm p)=\sum_{a=1}^nz_ae^{(a)}_{ii}+\sum_{1\leq b<a\leq n}\left(\sum_{j=i+1}^NQ^{a,b}_{i,j}-\sum_{j=1}^{i-1}\frac{p_i}{p_j}Q^{a,b}_{i,j}\right)\\+\sum_{1\leq a<b\leq n}\left(\sum_{j=i+1}^N\frac{p_j}{p_i}Q^{a,b}_{i,j}-\sum_{j=1}^{i-1}Q^{a,b}_{i,j}\right).
\end{multline}
\end{defn}

\begin{lem}\cite{TV23}
The subspaces $(\C^N)^{\otimes n}_{\bm\la}$ are invariant under the actions of the dynamical operators $X_1,\dots, X_N$. \qed
\end{lem}

\begin{example}
Let $N=n=3$, and $\bm\la=(1,1,1)$. The elements of $\mc I_{\bm\la}$ are
\[
\begin{aligned}
&(\{1\},\{2\},\{3\}),\quad (\{1\},\{3\},\{2\}),\quad (\{2\},\{1\},\{3\}),\\
&(\{2\},\{3\},\{1\}),\quad (\{3\},\{1\},\{2\}),\quad (\{3\},\{2\},\{1\}).
\end{aligned}
\]
In the basis $(v_I)_{I\in\mc I_{\bm\la}}$, the matrices of the dynamical operators $X_1,X_2,X_3$ are respectively:
\[
X_1(\bm z;\bm p)=\begin{pmatrix}
z_1&0&1&0&0&0\\
0&z_1&0&1&1&0\\
p_2/p_1&0&z_2&0&1&0\\
0&0&0&z_2&0&1\\
0&p_2/p_1&0&0&z_3&0\\
p_3/p_1&0&0&p_2/p_1&0&z_3
\end{pmatrix},
\]
\[X_2(\bm z;\bm p)=
\begin{pmatrix}
z_2&1&-1&0&0&0\\
p_3/p_2&z_3&0&0&-1&0\\
-p_2/p_1&0&z_1&1&0&0\\
0&0&p_3/p_2&z_3&0&-1\\
0&-p_2/p_1&0&0&z_1&1\\
0&0&0&-p_2/p_1&p_3/p_2&z_2
\end{pmatrix},
\]
\[X_3(\bm z;\bm p)=\begin{pmatrix}
z_3&-1&0&0&0&0\\
-p_3/p_2&z_2&0&-1&0&0\\
0&0&z_3&-1&-1&0\\
0&0&-p_3/p_2&z_1&0&0\\
0&0&0&0&z_2&-1\\
-p_3/p_1&0&0&0&-p_3/p_2&z_1
\end{pmatrix}.
\]
Notice that 
\[X_1+X_2+X_3=(z_1+z_2+z_3){\rm Id}.
\]\qetr
\end{example}

\subsubsection{Stable envelopes}Set $\la^{(i)}=\sum_{j=1}^i\la_j,\quad i=1,\dots, N$, and
\[
I^{\rm min}_{\bm\la} = \bigl( (\lambda_1 \ldots \lambda^{(1)}), (\lambda^{(1)} + 1 \ldots \lambda^{(2)}), \ldots, (\lambda^{(N-1)} + 1 \ldots n) \bigr),\quad .
\]

For \(I \in \mc I_{\bm\la}\), let \(\sigma_I \in S_n\) be the element of minimal length such that \(\sigma_I(I^{\rm min}_{\bm\la}) = I\).

Set ${\bm \gm}_i=(\gm_{i,1},\dots,\gm_{i,\la_i})$, $i=1,\dots, N$, and $\bm\gm=({\bm\gm}_1,\dots,{\bm\gm}_N)$. Define
\[{\rm Stab}_I(\bm\gm;\bm z):=\frak S_{\si_{\si_0(I)}}(\bm\gm;\bm z_{\si_0}),\qquad {\rm Stab}_I^{\rm op}(\bm\gm;\bm z)=\frak S_{\si_I}(\bm\gm;\bm z).
\]

Consider now the partial flag variety $F_{\bm\la}$, parametrizing chains $0=V_0\subset V_1\subset\dots\subset V_N=\C^n$. The torus $(\C^*)^n$ acts on the space $\C^n$, inducing an action on $F_{\bm\la}$. The quotient bundles $Q_i$, $i=1,\dots,N$, are $(\C^*)^n$-equivariant, with equivariant Chern roots $\bm\gm_i=(\gm_{i,1},\dots,\gm_{i,\la_i})$. If we denote by $\bm z=(z_1,\dots, z_n)$ the equivariant Chern roots of the trivial bundle $\underline{\C^n}\to F_{\bm\la}$, we have a ring presentation
\[H^\bullet_{(\C^*)^n}(F_{\bm\la},\C)\cong \frac{\C[\bm\gamma]^{S_{\bm\lambda}}\otimes\C[\bm z]}{I},\quad I = \left\langle \prod_{i=1}^N \prod_{j=1}^{\lambda_i}(1 + {\sf t}\, \gamma_{i,j}) = \prod_{a=1}^n(1+{\sf t}\,z_a) \right\rangle,
\]where $\sf t$ is a formal variable, and $\C[\bm\gamma]^{S_{\bm\lambda}}$ the ring of block-symmetric polynomials in $\bm\gamma = (\bm\gamma_1,\dots,\bm\gamma_N)$. By setting $z_1=\dots=z_n=0$, we recover the presentation \eqref{prescoh}.

Furthermore, the $(\C^*)^n$-action on $F_{\bm\la}$ can be suitably encoded in its Gromov--Witten theory,  leading to equivariant analogs of quantum cohomology and quantum products. See \cite{Giv96} for details.

Denote by ${\rm Stab}_I$ and ${\rm Stab}_I^{\rm op}$ the classes in $H^\bullet_{(\C^*)^n}(F_{\bm\la},\C)$ defined by the polynomials ${\rm Stab}_I(\bm\gm;\bm z)$ and ${\rm Stab}_I^{\rm op}(\bm\gm;\bm z)$, respectively. Define the {\it stable envelope map}
\[{\rm Stab}_{\bm\la}\colon (\C^N)^{\otimes n}\otimes \C[\bm z]\to H^\bullet_{(\C^*)^n}(F_{\bm\la},\C),\qquad v_I\mapsto {\rm Stab}_I,\quad I\in\mc I_{\bm\la}.
\]

\begin{thm}\cite{TV23}
The map ${\rm Stab}_{\bm\la}$ is an isomorphism of free $\C[\bm z]$-modules. Moreover, the isomorphism ${\rm Stab}_{\bm\la}$ intertwines the dynamical operators $X_1(\bm z;\bm p),\dots, X_N(\bm z;\bm p)$ acting on $(\C^N)^{\otimes n}\otimes \C[\bm z]$ and the operators of equivariant quantum multiplication $c_1(Q_1)*_{\bm q},\dots,$ $c_1(Q_N)*_{\bm q}$, where 
\[\bm q=(q_1,\dots, q_{N-1}),\quad q_i=\frac{p_{i+1}}{p_i},\quad i=1,\dots, N-1.\tag*{\qed}
\]
\end{thm}

In the non-equivariant limit $\bm z=0$, the equivariant quantum multiplication operators $c_1(Q_i)*_{\bm q}$ specialize to $c_1(Q_i)\sq_{\bm q}$. Hence, the dynamical operators $X_1,\dots, X_N$, evaluated at $\bm z=0$ and $\bm p=(1,q_1,q_1q_2,\dots, q_1q_2\dots q_{N-1})$ give the explicit matrix formulas for $c_1(Q_1)\sq_{\bm q},\dots,$ $c_1(Q_N)\sq_{\bm q}$ in the (suitably arranged) Schubert basis $(\frak S_{\si_I}(\bm\gm))_{I\in \mc I_{\bm\la}}$.

\AppendixSection{Riemann reduction, Mellin convolution identities }\label{appD}

\subsubsection{Riemann reduction formula} Let \((a_n)_{n=1}^\infty\) be a sequence of complex numbers such that \(a_n = O(n^k)\) for some \(k \geq 0\). Consider the ordinary and Dirichlet generating functions \(f(z) = \sum_{n=1}^\infty a_n z^n\) and \(F(s) = \sum_{n=1}^\infty a_n/n^s\).

\begin{lem}
The power series \(f(z)\) has radius of convergence \(\geq 1\), and the Dirichlet series \(F(s)\) converges absolutely for \({\rm Re}(s) > k+1\); in particular, its abscissa of absolute convergence satisfies \(\sigma_a \leq k+1\).
\end{lem}

\begin{proof}
Since \(|a_n| \leq C n^k\), we bound \(|f(z)|\) by \(\sum n^k |z|^n=\eu E_k(|z|)(1-|z|)^{-k-1}\), with $\eu E_k$ Eulerian polynomial, which converges for \(|z| < 1\). For \(F(s)\), the bound \(|a_n| \leq C n^k\) gives \(\sum |a_n|/n^s \leq C \sum n^{k - {\rm Re}(s)}\), which converges if \({\rm Re}(s) > k + 1\).
\end{proof}

\begin{thm}
For ${\rm Re}(s)>k+1$ have
\beq\label{Riemred} F(s)\Gm(s)=\int_0^\infty f(e^{-x})x^{s-1}{\rm d}x.
\eeq
\end{thm}
\proof
Let us first prove that the integral on the r.h.s.\,\,is absolutely convergent for ${\rm Re}(s)>k+1$. Let us split the integral
\[\int_0^\infty |f(e^{-x})|x^{s-1}{\rm d}x= \int_0^1 |f(e^{-x})|x^{s-1}{\rm d}x+\int_1^\infty |f(e^{-x})|x^{s-1}{\rm d}x.
\]Since $f(e^{-x})\sim a_1 e^{-x}$ as $x\to\infty$, the integrand in $\int_1^\infty$ behaves like \(e^{-x} x^{{\rm Re}(s)-1}\), which is integrable at infinity if and only if \({\rm Re}(s) > 0\). Moreover, near zero, we have 
\[|f(e^{-x})|\leq C\cdot (1 - e^{-x})^{-(k+1)} \sim C\cdot x^{-(k+1)}.
\]So the integral $\int_0^1$ is dominated by $\int_0^1x^{{\rm Re}(s) - k - 2}{\rm d}x$, which is finite if and only if \({\rm Re}(s) > k + 1\).
Therefore, the r.h.s.\,\,of \eqref{Riemred} is absolutely integrable if and only if ${\rm Re}(s) > \max(0,\, k+1)=k+1$. 
By integrating term-by-term the series, the claim follows.
\endproof
\begin{cor}
For $c>k+1$, we have
\beq\label{Riemred2}
f(e^{-x})=\frac{1}{2\pi\sqrt{-1}\Gm(s)}\int_{\La_c}F(s)\Gm(s)x^{-s}{\rm d}s,
\eeq where $\La_c=\{c+\sqrt{-1}t\colon t\in\R\}$.\qed
\end{cor}

\begin{rem}
Equations \eqref{Riemred} and \eqref{Riemred2} already appear in the famous 1859 paper \cite{Rie59} by B.\,Riemann. They are used to derive the integral representation of $\zeta$, and to consequently obtain its function equation. Albeit it importance, formula \eqref{Riemred} has traditionally no name. Following \cite{Win47}, we call it {\it Riemann's reduction of Dirichlet series to power series}.\qrem
\end{rem}

\subsubsection{Convolution identities for the Mellin transform}

For a function $f$ defined on $\R_{>0}$, the integral $\mathcal{M}[f](s) := \int_0^\infty f(x)x^{s-1}{\rm d}x$ is the {\it Mellin transform of $f$ at $s\in\C$}.

For $p\geq 1$, consider the space $\mc L^p:=L^p(\R_{>0},x^{-1}{\rm d}x)$ of (classes of) functions $f$ such that $\int_0^\infty |f(x)|^px^{-1}{\rm d}x<+\infty$.

If $f_1,f_2\in \mc L^1$, their {\it multiplicative convolution} $f_1*_\times f_2$ is defined as \( (f_1 *_\times f_2)(x) := \int_0^\infty f_1(y)f_2(x/y)\,dy/y \) for $x>0$.
\begin{thm}
We have $f_1*_\times f_2\in \mc L^1$. More generally, if $x^sf_1, x^sf_2\in \mc L^1$, then also $x^s(f_1*_\times f_2)\in \mc L^1$, and 
\[\mc M[f_1*_\times f_2](s)=\mc M[f_1](s)\cdot \mc M[f_2](s).
\]
\end{thm}
\proof
The change of variables  \( x = yz \) and Fubini yield
\[
\mathcal{M}[f_1 *_\times f_2](s) = \int_0^\infty f_1(y)y^s\,dy/y \cdot \int_0^\infty f_2(z)z^{s-1}dz = \mathcal{M}[f_1](s)\cdot\mathcal{M}[f_2](s).\qedhere
\]

If \(\widehat{f_1}, \widehat{f_2}\) are functions defined and integrable on a common vertical line \({\rm Re}(s) = c\), their additive convolution is defined by \( (\widehat{f_1} *_c \widehat{f_2})(s) := \frac{1}{2\pi\sqrt{-1}} \int_{{\rm Re}(w)=c} \widehat{f_1}(w)\widehat{f_2}(s-w)\,dw \).

Let $p,q$ real numbers such that 
\[\frac{p-1}{p}+\frac{q-1}{q}\geq 1.
\]

\begin{thm}\label{mconv} 
Let $s\in\C,c\in\R$. If $x^cf_1\in\mc L^p$ and $x^{s-c}f_2\in\mc L^q$, then 
\beq
\mc M[f_1\cdot f_2](s)=(\mc M[f_1]*_c\mc M[f_2])(s).\tag*{\qed}
\eeq
\end{thm}
See \cite[\S\S 2.7, 3.17, 4.14]{Tit48} for proofs, and detailed discussions.

\begin{rem}
Theorem \ref{mconv} can be formally justified in several ways. For example, assume \(\widehat{f}_1, \widehat{f}_2\) are Mellin transforms of \(f_1,f_2\), analytic and integrable on vertical lines \({\rm Re}(w)=c_1\), \({\rm Re}(z)=c_2\). Expanding \(f_1(x),f_2(x)\) via Mellin inversion:
\[
f_1(x)f_2(x) = \frac{1}{(2\pi\sqrt{-1})^2} \int_{{\rm Re}(w)=c_1} \int_{{\rm Re}(z)=c_2} \widehat{f}_1(w)\widehat{f}_2(z)x^{-w-z}dz\,dw.
\]
Applying \(\mathcal{M}\), the inner integral becomes \( \int_0^\infty x^{s - w - z - 1}dx = 2\pi\sqrt{-1}\,\delta(s - w - z) \) as distributions, so:
\[
\mathcal{M}[f_1 \cdot f_2](s) = \frac{1}{2\pi\sqrt{-1}} \int_{{\rm Re}(w)=c} \widehat{f}_1(w)\widehat{f}_2(s - w)\,dw,
\]where \({\rm Re}(w) = c\) lies in some common vertical strip of convergence of both \(\widehat{f}_1\) and \(\widehat{f}_2\). 

Alernatively, consider directly the integral
\[\frac{1}{2\pi\sqrt{-1}} \int_{{\rm Re}(w)=c} \widehat{f}_1(w)\widehat{f}_2(s - w)\,dw=\frac{1}{2\pi\sqrt{-1}} \int_{{\rm Re}(w)=c} \left(\int_0^\infty f_1(x)x^{w-1}dx\right)\widehat{f}_2(s - w)\,dw.
\]Assuming the integrals can be interchanged, one obtain the thesis.

In all these kind of formal proofs, one need to justify the interchange of the integrals. There are several sets of conditions guaranteeing this. The formulation of Theorem \ref{mconv}, due to E.C.\,Titchmarsh, is only one of the possible. For example, alternatively, it would suffice to assume conditions on $f_1$ and $\widehat{f}_2$:
\[x^cf_1(x)\in\mc L^1,\qquad \int_{-\infty}^{+\infty}|\widehat{f}_2(s-c-\sqrt{-1}t)|dt<\infty.
\]See \cite[Thm.\,89, pag. 118]{Tit48} for a more general statement.\qrem
\end{rem}

\AppendixSection{Fabry--Lindel\"of theorem}\label{appE}

Let \( f(z) = \sum_{k=0}^\infty a_k z^k \) be a power series with radius of convergence equal to \(1\). A boundary point \( e^{\sqrt{-1}\theta} \in \partial \mathbb{D} \) may or may not be a singularity of \( f \). The \emph{Fabry--Lindelöf theorem} provides a criterion to detect boundary singularities of a function analytic in the unit disk based on the asymptotic behavior of its coefficients. 

Define the sequence
\[
b_n(\theta) := \frac{1}{2} \left| \sum_{k=0}^n a_k \binom{n}{n-k} e^{\sqrt{-1} k \theta} \right|^{1/n}, \quad n \geq 0.
\]

\begin{thm}[Fabry--Lindelöf, \cite{Fab97,Lin98}]\label{FLthm}
For every \( \theta \in \mathbb{R} \), we have
\[
\limsup_{n \to \infty} b_n(\theta) \leq 1.
\]
Moreover, equality holds if and only if \( e^{i\theta} \) is a singular point of the function \( f \).
\end{thm}
\proof
Consider the function 
\[g(y)=\frac{1}{1-y}f\left(\frac{e^{\sqrt{-1}\theta}y}{1-y}\right)= \sum_{n=0}^\infty \beta_n y^n,\qquad \beta_n := \sum_{k=0}^n a_k \binom{n}{n - k} e^{\sqrt{-1} k \theta}.%=\sum_{n=0}^\infty \underbrace{\left(\sum_{k=0}^n a_k \binom{n}{n-k} e^{\sqrt{-1} k \theta}\right)}_{\bt_n}y^n.
\] The change of variables \( z = \frac{e^{\sqrt{-1}\theta} y}{1 - y} \) (a Möbius transformation) maps the half-plane \( \operatorname{Re}(y) < \frac{1}{2} \) conformally onto the unit disk \( |z| < 1 \), sending \( y = \frac{1}{2} \) to \( z = e^{\sqrt{-1}\theta} \). Since $g$ is analytic on \( \operatorname{Re}(y) < \frac{1}{2} \), the radius of convergence of the series for \( g(y) \) satisfies
\[R = \left( \limsup_{n \to \infty} |\beta_n|^{1/n} \right)^{-1} = \frac{1}{2} \left( \limsup_{n \to \infty} b_n(\theta) \right)^{-1}\geq \frac{1}{2}.
\]Equality holds (i.e., the radius is exactly \( \frac{1}{2} \)) if and only if \( g \) has a singularity at \( y = \frac{1}{2} \). 
\endproof

\begin{example}
Let $f(z)=\sum_{n=0}^\infty z^n$. We identically have $b_n(\theta)=|1+e^{\sqrt{-1}\theta}|/2$, for any $n\in\N$. We have $\limsup b_n(\theta)=1$ if and only if $\theta=0$. \qetr
\end{example}

See also \cite{Ost26} for further details, and other characterizations of singularities.

\subsubsection{Numerical evidence for a natural boundary of $\Lcyrit_N(z)$} 
Set
\[
b_{n,\theta}=\frac{1}{2}\left|\sum_{k=0}^{n}\lcyr(k,2)\binom{n}{k}e^{\sqrt{-1}k\theta}\right|,
\quad n\in\N_{>0},\quad \theta\in[0,2\pi].
\]
Numerical experiments indicate that, for any sampled value of $\theta$, we definitely have $b_{n,\theta}>1$, so that $\limsup_{n} b_{n,\theta}=1$. This provides numerical evidence for the existence of a natural boundary of $\Lcyrit_2(z)$, by the Fabry--Lindel\"of Theorem \ref{FLthm}, and consequently for every $\Lcyrit_N(z)$, by the first equation in \eqref{HaL}. Here we restrict ourselves to presenting the following table of values of $b_{n,\theta}$ for the sample points $\theta_m=2\pi m/10$, $m=1,\dots,10$, and for $n=P(10^k)$, $k=1,2,3,4$, where $P(a)$ denotes the $a$-th prime number.

\begin{table}[h!]
\centering
\begin{tabular}{c|cccc}
\hline
$m \backslash n$ & $P(10^1)$ & $P(10^2)$ & $P(10^3)$ & $P(10^4)$ \\
\hline
1  & 0.983001 & 1.00502 & 1.00059 & 1.00007 \\
2  & 1.01382  & 1.00498 & 1.00063 & 1.00007 \\
3  & 1.01390  & 1.00498 & 1.00063 & 1.00007 \\
4  & 0.948462 & 1.00502 & 1.00059 & 1.00007 \\
5  & 1.04988  & 1.00746 & 1.00078 & 1.00008 \\
6  & 0.948462 & 1.00502 & 1.00059 & 1.00007 \\
7  & 1.01390  & 1.00498 & 1.00063 & 1.00007 \\
8  & 1.01382  & 1.00498 & 1.00063 & 1.00007 \\
9  & 0.983001 & 1.00502 & 1.00059 & 1.00007 \\
10 & 1.06435  & 1.00753 & 1.00078 & 1.00008 \\
\hline
\end{tabular}
\caption{Values of $b_{n,\theta_m}$ for $\theta_m=2\pi m/10$ ($m=1,\dots,10$) and $n=P(10^k)$ with $k=1,2,3,4$.}
\end{table}

\end{document}